%% file: main.tex
\documentclass[10pt]{amsart}
\usepackage{amssymb,mathrsfs}
\usepackage{mathabx}
\usepackage{color}
\usepackage{graphicx}
\graphicspath{{./figures/}}
\usepackage{float}
\usepackage{tikz}
\usepackage{tabularx}
\usepackage{tabulary}
\usepackage{multirow}

\usetikzlibrary{cd}
\usepackage[bottom]{footmisc}
\usepackage[colorlinks,hyperindex,linkcolor=blue]{hyperref}
\hypersetup{
 pdfauthor = {Alvaro del Pino Gomez, Lauran Toussaint},
 pdftitle= {Wrinkling h-principles for integral submanifolds of jet spaces},
 pdfsubject = {Differential geometry, distributions, h-principle, jet spaces}}
\usepackage{enumerate}

\numberwithin{equation}{subsection}

\newcommand{\SD}{{\mathcal{D}}}

\newcommand{\ST}{{\mathcal{T}}}

\newcommand{\SV}{{\mathcal{V}}}
\newcommand{\SL}{{\mathcal{L}}}
\newcommand{\SJ}{{\mathcal{J}}}
\newcommand{\SF}{{\mathcal{F}}}
\newcommand{\F}{{\mathcal{F}}}

\newcommand{\SR}{{\mathcal{R}}}

\renewcommand{\SS}{{\mathcal{S}}}

\newcommand{\X}{{\mathfrak{X}}}

\newcommand{\wtd}{\widetilde}
\renewcommand{\d}{{\operatorname{d}}}
\renewcommand{\ss}{{s}}

\newcommand{\R}{{\mathbb{R}}}
\newcommand{\Z}{{\mathbb{Z}}}
\newcommand{\NS}{{\mathbb{S}}}
\newcommand{\D}{{\mathbb{D}}}
\newcommand{\N}{\mathbb{N}}

\newcommand{\Op}{{\mathcal{O}p}}
\newcommand{\Sol}{{\operatorname{Sol}}}

\newcommand{\Diff}{{\operatorname{Diff}}}

\newcommand{\Hom}{{\operatorname{Hom}}}
\newcommand{\GL}{{\operatorname{GL}}}

\newcommand{\Gr}{{\operatorname{Gr}}}

\newcommand{\Sym}{{\operatorname{Sym}}}
\renewcommand{\Vert}{{\operatorname{Vert}}}
\newcommand{\can}{{\operatorname{can}}}
\newcommand{\Lift}{{\operatorname{Lift}}}
\newcommand{\meta}{{\operatorname{meta}}}

\newcommand{\Image}{{\operatorname{Image}}}

\newcommand{\id}{{\operatorname{id}}}

\newcommand{\fold}{{\operatorname{Fld}}}
\newcommand{\doubleFold}{{\operatorname{DbFld}}}
\newcommand{\ClosedDoubleFold}{{\operatorname{ClDbFld}}}
\newcommand{\pleat}{{\operatorname{Pleat}}}
\newcommand{\ClosedPleat}{{\operatorname{ClPleat}}}
\newcommand{\Reid}{{\operatorname{Reid}}}
\newcommand{\Stab}{{\operatorname{Stab}}}
\newcommand{\Wrin}{{\operatorname{Wrin}}}
\newcommand{\ClosedWrin}{{\operatorname{ClWrin}}}

\newtheorem{lemma}{Lemma}[section]
\newtheorem{proposition}[lemma]{Proposition}
\newtheorem{theorem}[lemma]{Theorem}
\newtheorem{corollary}[lemma]{Corollary}
\newtheorem{definition}[lemma]{Definition}

\newtheorem{remark}[lemma]{Remark}

\newtheorem{notation}[lemma]{Notation}

\begin{document} 

\title{Wrinkling $h$-principles for integral submanifolds of jet spaces}

\date{\today}


\author{\'Alvaro del Pino}
\address{Utrecht University, Department of Mathematics, Budapestlaan 6, 3584 Utrecht, The Netherlands}
\email{a.delpinogomez@uu.nl}

\author{Lauran Toussaint}
\address{Universite Libre de Bruxelles, Department of Mathematics, Boulevard du Triomphe, 1050 Bruxelles, Belgium}
\email{lauran.toussaint@ulb.be}

\begin{abstract}
Y. Eliashberg and N. Mishachev introduced the notion of wrinkled embedding to show that any tangential homotopy can be approximated by a homotopy of topological embeddings with mild singularities. This concept plays an important role in Contact Topology: The loose legendrian $h$-principle of E. Murphy relies on wrinkled embeddings to manipulate the legendrian front. Similarly, the simplification of legendrian front singularities was proven by D. \'Alvarez-Gavela by defining the notion of wrinkled legendrian.

This paper and its sequel generalise these ideas to general jet spaces. The main theorem in the present paper proves the analogue of the result by Eliashberg and Mishachev: Any homotopy of the $r$--order differential information of an embedding can be approximated by a homotopy of embeddings with wrinkle-type singularities (of order $r$).

The local version of the previous statement, which is of independent interest, says that the holonomic approximation theorem holds over closed manifolds if, instead of sections, we consider multi-valued sections with simple singularities.
\end{abstract}

\maketitle
\setcounter{tocdepth}{1}
\tableofcontents

\input{PaperI-Introduction}
\input{PaperI-JetSpaces}

\input{PaperI-hPrinciple}

\input{PaperI-SingularitiesPrelim}

\input{PaperI-Lifting}

\input{PaperI-Singularities}

\input{PaperI-HolonomicApproximation}

\input{PaperI-HolonomicApproximationParam}

\input{PaperI-WrinkledEmbeddings}

\appendix
\input{PaperI-AppendixOdd}

\input{PaperI-AppendixWrinkling}

\bibliographystyle{abbrv}
\bibliography{Bibliography}
\end{document}

%% file: PaperI-Introduction.tex
\section{Introduction}

\subsection{Wrinkled embeddings}

Let $M \subset N$ be manifolds of dimensions $m$ and $n$, respectively. In many geometrically meaningful situations, we are interested in producing isotopies of $M$ that simplify its position with respect to some geometric structure in $N$. For instance, $N$ may be endowed with a foliation $\SF$ and we want to isotope $M$ so that it becomes transverse.

In general, such a process is obstructed. The tangent space of $M$ defines a section $\Gr(M): M \to \Gr(TN,m)|_M$ into the Grassmannian of $m$-planes of $N$. If we want to make $M$ transverse to $\SF$ by an isotopy, it is certainly necessary that $\Gr(M)$ can be homotoped to be transverse to $\SF$. This obstruction is purely algebraic topological in nature and can be analysed using obstruction theory. One may then ponder whether the vanishing of these obstructions (that we call \emph{formal}) is sufficient for the existence of the desired isotopy. The answer is, in general, no. For instance, if $N$ is a fibration over $\R$ and $\SF$ is the foliation by fibres, making $M$ transverse to $\SF$ would produce a function on $M$ with no critical points, which is impossible if $M$ is closed. This obstruction is geometric and not algebraic.

In situations where singularities may be unavoidable for geometric reasons, we may attempt to make them as mild as possible instead. This was proven, in the aforementioned setting, by Eliashberg and Mishachev \cite{ElMiWrinEmb}: They showed that, if no formal obstruction exists, $M$ can be isotoped to have simple singularities of tangency with respect to $\SF$. The idea of the proof is as follows: Since no formal obstruction exists, we are given a homotopy $\Gr_s$ starting at $G_0 = \Gr(M)$ and finishing at a bundle $G_1$ transverse to $\SF$. Then, instead of isotoping $M$, we produce a homotopy $M_s$ of topological submanifolds that may have cuspidal singularities (or so-called \emph{wrinkles}). Despite being singular, the $M_s$ admit well-defined Grassmannian maps $\Gr(M_s)$ and the heart of the argument is that it is possible to choose $M_s$ so that $\Gr(M_s)$ approximates $\Gr_s$; here the flexibility provided by the cuspidal singularities/wrinkles is key. The proof concludes by smoothing out $M_s$; when we do so, $M_1$ becomes a smooth submanifold and the cusps/wrinkles are traded for (simple!) singularities of tangency with $\SF$.

\subsection{Wrinkled embeddings of higher order}

The starting point of the present article is that the result of Eliashberg-Mishachev is a first order statement. That is: the Gauss map $\Gr(M)$ is the first derivative of $M$, and their theorem states that any homotopy $\Gr_s$ of this first derivative can be approximated by a homotopy $M_s$ of $M$, as long as we allow for simple singularities in the process. Our main result says that this can be done for higher order data as well:
\begin{theorem} \label{thm:wrinkledEmbeddingsInformal}
Fix an integer $r$. Any homotopy of the $r$--order information of an embedding can be approximated by a homotopy of topological embeddings with zig-zags. Similar statements hold parametrically and relatively by allowing zig-zags to appear and disappear.
\end{theorem}
Roughly speaking, a zig-zag is a pair of cusp singularities that sit with respect to each other as depicted in Figure \ref{Fig:HolonomicApprox_Zigzag}. For the first order case, treated in \cite{ElMiWrinEmb}, one can rely on standard semi-cubic cusps. The higher order case needs ``sharper'' higher order cusps. Moreover, during a homotopy, one also needs to introduce birth/death phenomena for such singularities.

This statement will be stated more formally and proven, as Theorem \ref{thm:wrinkledEmbeddings}, in Section \ref{sec:wrinkledEmbeddings}. The parametric version is stated as Theorem \ref{thm:wrinkledEmbeddingsParam}.

\subsection{Holonomic approximation}

One can succinctly state Theorem \ref{thm:wrinkledEmbeddingsInformal} by saying that \emph{holonomic approximation holds for submanifolds with zig-zags}. To put this into perspective, let us recall the standard setup for $h$-principle and geometric PDEs.

Given a smooth bundle $Y \to X$, we can define the \emph{bundle of $r$-jets} $J^r(Y) \to X$. Its fibres consist of $r$-order Taylor polynomials of sections of $Y$. Given any section $f: X \to Y$, we can consider its $r$-order differential data $j^r f: X \to J^r(Y)$; such a section of jet space is said to be \emph{holonomic}. Most sections $F: X \to J^r(Y)$ are not holonomic and, to emphasise this, we call them \emph{formal sections}.

This provides a very convenient setup to discuss partial differential relations (PDRs). Indeed, we can define a PDR of order $r$ to be a subset $\SR \subset J^r(Y)$. It readily follows that a \emph{solution} of $\SR$ is a section $f: X \to Y$ whose $r$-order Taylor polynomial $j^rf$ takes values in $\SR$. More generally, we can define \emph{formal solutions} of $\SR$ to be sections $F: X\to \SR$. The existence of formal solutions is thus a necessary condition for the existence of solutions. One can then compare the spaces of solutions and formal solutions and ask, in particular, whether the two are weakly homotopy equivalent. If this is the case, $\SR$ is said to \emph{satisfy the $h$-principle}.

If the relation $\SR$ we consider is open (as is sometimes the case with relations of geometric origin, like those describing contact or symplectic structures), we could attempt to find solutions of $\SR$ using the following idea: We start with a formal solution $F$ and we find $f: X \to Y$ such that $j^rf$ approximates $F$. If this approximation is good enough, $j^rf$ will land in $\SR$ and $f$ will be a solution. This idea is called, quite descriptively, \emph{holonomic approximation}.

It turns out that this does not work and, indeed, many open relations do not satisfy the $h$-principle (for instance, symplectic structures). However, using his \emph{method of flexible sheaves}, M. Gromov \cite{Gr71} proved that holonomic approximation does hold if we try to approximate only over a subset $X' \subset X$ of positive codimension (in fact, one needs to approximate not quite over the given $X'$ but over a $C^0$-close copy of $X'$ that is more ``wiggly''). Due to the fact that open manifolds can be retracted to their skeleton (which is a positive codimension CW-complex), this can then be used to prove that the $h$-principle applies for any $\SR$ open and Diff-invariant (i.e. invariant under the action of the diffeomorphism group of $X$), as long as $X$ is open. This result applies to ``generic''/``non-degenerate'' geometric structures (like contact or symplectic) and generalises prior results about immersions (due to Hirsch-Smale \cite{Sm57,Hi59}) and submersions (due to Phillips \cite{Ph67}).

One can then pose the question: ``what can we do for closed manifolds?'' We henceforth assume that $X$ is indeed closed.

\subsection{Holonomic approximation for multiply-valued sections}

At each point in $J^r(Y)$, we are given a collection of tautological equations encoding the fact that certain fibre directions should correspond to derivatives of some others. These equations are pointwise linear and define the so-called \emph{Cartan distribution} $\xi_\can$ in jet space. From its construction, it follows that holonomic sections can be characterised as those sections tangent to $\xi_\can$. This led R. Thom \cite{Thom1} to define \emph{generalised solutions} of $\SR$ as maps $X \to J^r(Y)$ (not necessarily sections!) tangent to $\xi_\can$ and taking values in $\SR$. To emphasise the fact that these are not sections, one may note that it also makes sense to consider general tangent maps $M \to (J^r(Y),\xi_\can)$, where $M$ is some other manifold of the same dimension as $X$.

Recall that holonomic sections $j^rf: X \to J^r(Y)$ are in correspondence with their underlying sections $f: X \to Y$. This is (almost) true for generalised solutions as well: $\phi: M \to (J^r(Y),\xi_\can)$ can be uniquely recovered from its \emph{front projection} $\pi_f \circ \phi: M \to Y$ whenever the \emph{base projection} $\pi_b \circ \phi: M \to X$ is a immersion. It follows that, as long as $\phi$ is graphical over $X$ in a dense set, it will be uniquely recovered from $\pi_f \circ \phi$. However, the fibre over a point $x \in X$ may intersect the image of $\pi_f \circ \phi$ in a number of points different than one. This allows us to regard such front projections as \emph{multi-sections} $M \to Y$ and the tangent maps $\phi$ themselves as holonomic lifts of multi-sections.

We then ask whether it is possible to approximate any formal section by a holonomic lift of multi-section. We phrase this as ``holonomic approximation holds, even over closed manifolds, in the class of multi-sections''. Remarkably, Thom, preceding Gromov's work, proved this in \cite{Thom1}, using ideas that seem to be a precursor of the pleating/wrinkling approaches to $h$-principles \cite{ElMiWrinI,ElMiWrinII,ElMiWrinIII,ElMiWrinEmb}. However, Thom's argument is somewhat incomplete (particularly regarding higher jets) and, more importantly for us, his method produces generalised solutions with uncontrolled singularities. Later, Gromov provided an argument \cite[Section 4.3]{Gro96}, based on microflexibility/pleating, to construct generalised solutions whose only singularities of mapping are folds. His approach applies to more general manifolds endowed with bracket-generating distributions, as long as certain dimensional constraints hold.

Our result produces multi-sections whose holonomic lifts are smooth embeddings. To do so, we use pleating in the front projection, instead of in jet space itself (unlike Gromov). More concretely: our multi-sections are topological embeddings whose only singularities are order-$r$ cusps (that project to the base manifold $X$ as folds and that are smooth when lifted to $r$-jet space). Our second result reads:
\begin{theorem}\label{thm:wrinkledSectionsNonParametricInformal}
Let $F: X \to J^r(Y)$ be an arbitrary formal section. Then, for any $\varepsilon > 0$, there exists a tangent embedding $f: X \to (J^r(Y),\xi_\can)$ satisfying:
\begin{itemize}
\item $f$ is the lift of a multi-section with zig-zags;
\item $|f - F|_{C^0} < \varepsilon$.
\end{itemize}
\end{theorem}
This is restated as Theorem \ref{thm:holonomicApproxZZ} in Section \ref{sec:holonomicApproxZZ}. It then follows immediately:
\begin{corollary}
Let $\SR \subset J^r(X)$ be an open differential relation admitting a formal solution $F$. Then, $\SR$ admits a generalised solution $f$, $C^0$-close to $F$.
\end{corollary}
That is, even though the $h$-principle does not hold for arbitrary open differential relations, it does hold when we allow generalised solutions. As before, both these statements can be extended to the parametric (in which $F$ varies in a family $F_s$) and relative settings (where all the $F_s$ are already holonomic in some part of the domain, and some of them are everywhere holonomic). These versions are stated and proven in Section \ref{sec:holonomicApproxParam}. 

The reader may wonder what is the relation between Theorem \ref{thm:wrinkledEmbeddingsInformal} and Theorem \ref{thm:wrinkledSectionsNonParametricInformal}. The answer is that the latter is the local version of the former. That is: when trying to isotope a submanifold $M \subset N$, we work in its tubular neighbourhood $\Op(M)$, which we identify with the normal bundle $\nu(M) \to M$. Then, sufficiently small homotopies of $M$ as a submanifold with zig-zags correspond to homotopies of the zero section as a multi-section with zig-zags. That is, Theorem \ref{thm:wrinkledSectionsNonParametricInformal} implies Theorem \ref{thm:wrinkledEmbeddingsInformal}; this is carried out in detail in Section \ref{sec:wrinkledEmbeddings}.

\subsection{Organisation of the paper}

We review jet spaces, the $h$-principle philosophy, and some singularity theory in Sections \ref{sec:jetSpace}, \ref{sec:hPrinciple}, and \ref{sec:singularitiesPrelim} respectively. The reader familiar with $h$-principles can skip these sections safely.

The main technical ingredient of the paper is introduced in Section \ref{sec:metasymplectic}. It is the so-called \emph{metasymplectic projection}, which is the natural generalisation of the lagrangian projection to general jet spaces. Every map tangent to $\xi_\can$ appearing in this paper will be the lift of a map into the metasymplectic projection. 

In Section \ref{sec:singularities} we use this lifting approach to provide models of maps tangent to $\xi_\can$. These models correspond to different singularities of tangency with respect to the fibres of jet space. In a couple of instances, the models also have singularities of mapping. 

Theorem \ref{thm:wrinkledSectionsNonParametricInformal} is stated in its non-parametric form in Subsection \ref{sec:holonomicApproxZZ}. The parametric version is addressed in Subsection \ref{sec:holonomicApproxParam}. Theorem \ref{thm:wrinkledEmbeddingsInformal} is stated and proven in Section \ref{sec:wrinkledEmbeddings}.

A feature of our constructions is that our multi-sections/singular embeddings have singularities with a very concrete topology. Namely, in the non-parametric case, they form spheres of cusps. This differs slightly from the standard wrinkling approach, where the singularities appear in a wrinkle configuration. In Appendix \ref{sec:surgeries} we explain a surgery procedure to pass to wrinkles. This is well-known in wrinkling and goes under the name of \emph{chopping}.

Another subtlety is that some of our singularities (namely, the birth/death event of a zig-zag when $r$ is odd) are singularities of mapping (and tangency!). This is problematic, since we want our tangent mappings into jet space to be embeddings. This can be addressed as long as the jet space in question is not contact. The argument relies on a surgery construction, which we provide in Appendix \ref{sec:desingularisationOdd}. The main result is Theorem \ref{thm:fishSurgerySec}.

\subsection{Tangent submanifolds in jet spaces}

The main motivation behind this article was to understand better the space of embeddings $M \to (J^r(Y),\xi_\can)$ tangent to $\xi_\can$. This is extremely natural: Recall that the first jet space of functions $(J^1(X,\R),\xi_\can)$ is a contact manifold and therefore tangent embeddings are precisely legendrians. The study of legendrians is a driving theme in Contact Topology, where $h$-principle results provide flexibility (i.e. classification results) and generating functions/pseudoholomorphic curves/sheaves provide rigidity (in the form of obstructions/invariants).

There are two results on legendrians that we can highlight as inspiration for the present work: The first is the celebrated classification of loose legendrians due to E. Murphy \cite{Mur}, where it is shown that a certain subfamily of legendrians in higher dimensions satisfy the $h$-principle. The second is the simplification of legendrian singularities of tangency due to D. \'Alvarez-Gavela \cite{Gav2}; this generalises the work of Eliashberg and Mishachev on wrinkled submanifolds to the contact setting.

Similar results are not yet available for higher order jets (or for bundles with larger fibre). One exception is the case of curves tangent to Engel structures (which are locally modelled on $(J^2(\R,\R),\xi_\can)$), and which was treated in \cite{PP16,CP}.

The theorems in the present paper are partial results in this direction. Namely, they deal with integral submanifolds that are $C^0$-small deformations of formal sections in jet space and additionally have simple singularities of tangency with the front projection. We will study more general tangent submanifolds in the sequel \cite{PT2}, still under the assumption that their singularities of tangency are of corank-$1$. These upcoming results rely on the tools introduced in the present paper. We postpone further discussion till then, but there is one concrete observation that we want to emphasise:
\begin{remark}
The results in this paper indicate that, as long as they are assumed to have simple singularities of tangency, the classification up to homotopy of integral embeddings into non-contact jet spaces $(J^r(Y),\xi_\can)$ displays greater flexibility than the classification of legendrians.

However, even though the topological aspects of the theory seem rather flexible, there is non-trivial geometry to be understood. Namely, as of yet, the authors do not know how to handle higher singularities of tangency. \hfill$\triangle$
\end{remark}

\textbf{Acknowledgments:} The first author was funded, for the duration of the project, by the NWO grant 016.Veni.192.013. The second author was funded by the NWO grant 639.033.312, as well as the F.R.S-FNRS and the FWO under the Excellence of Science programme (grant No.
30950721).

%% file: PaperI-JetSpaces.tex
\section{Jet spaces} \label{sec:jetSpace}

In this Section we recall some elementary notions about jet spaces. A standard reference in the Geometry of PDEs literature is \cite[Chapter IV]{KrLyVi}, but we also recommend \cite[Section 2]{Vi14}. The two standard $h$-principle references also treat jet spaces, namely \cite[Section 1.1]{Gr86} and \cite[Chapter 1]{ElMi}.

\subsection{Jet spaces of sections} \label{ssec:jetSpace}

Let $X$ be an $n$-dimensional manifold and let $\pi: Y \to X$ be a submersion with $\ell$-dimensional fibres. We write $\Vert(Y) \subset TY$ for the vector subbundle consisting of vectors tangent to the fibres of $Y$.

We write $J^r(Y)$ for the space of $r$-jets of sections $X \to Y$. When $Y$ is the trivial $\R^\ell$-bundle over $X$ we often denote it by $J^r(X,\R^\ell)$. The spaces of $r$-jets, for varying $r$, fit in a tower of affine bundles:
\begin{equation}\label{eq:jetBundles}
\begin{tikzcd}
J^r(Y) \arrow[r,"\pi_{r,r-1}"] & J^{r-1}(Y) \arrow[r,"\pi_{r-1,r-2}"] &  \dots \arrow[r,"\pi_{1,0}"] & J^0(Y) = Y.
\end{tikzcd}
\end{equation}
For notational convenience, we single out the \textbf{front projection} and the \textbf{base projection} which are given, respectively, by the forgetful maps:
\[ \pi_f := \pi_{r,0}: J^r(Y) \to Y, \qquad \pi_b: J^r(Y) \to X. \]

Given a section $f: X \to Y$, we will write $j^rf: X \to J^r(Y)$ for its holonomic lift.

\subsubsection{Local coordinates} \label{ssec:coordinatesJetSpace}

By working locally we may assume that the base $X$ is a $n$-dimensional vector space, denoted by $B$, and that the fibre of $Y$ is a $\ell$-dimensional vector space, denoted by $F$. In this local setting the jet space $J^r(Y)$ can be identified with $J^r(B,F)$. To be explicit, we choose coordinates $x:=(x_1,\cdots,x_n)$ in $B$ and coordinates $y:=(y_1,\cdots,y_\ell)$ in $F$. We use $(x,y)$ to endow $J^r(B,F)$ with coordinates, as we now explain. 

A point $p \in J^r(B,F)$ is uniquely represented by an $r$-order Taylor polynomial based at $\pi_b(p) \in X$. Now, the $r$-order Taylor polynomial of a map $f: B \to F$ at $x$ reads:
\[ f(x+h) \cong \sum_{0 \leq |I| \leq r} (\partial^I f(x))\dfrac{dx^{\odot I}}{I!}(h,\dots,h), \]
where $I = (i_1,\dots,i_n)$ ranges over all multi-indices of length at most $r$. Here $\odot$ denotes the symmetric tensor product and we use the notation
\[ \d x^{\odot I} := \d x_{i_1} \odot \dots \odot \d x_{i_n}, \quad I  = (i_1,\dots,i_n).\]
This tells us that $J^r(B,F) \to B$ is a vector bundle and that, formally, we can use the monomials 
\[ \dfrac{dx^{\odot I}}{I!} \otimes e_j, \qquad 0 \leq |I| \leq r' \quad 1 \leq j \leq \ell \]
as a framing; here $\{e_j\}_{1\leq j \leq \ell}$ is the standard basis of $F$ in the $(y)$-coordinates.

We can write $z_j^{(I)}$ for the coordinate dual to the vector $\frac{dx^{\odot I}}{I!} \otimes e_j \in \Sym^{|I|}(B,F)$. This definition depends only on the choice of coordinates $(x,y): Y \to B \times F$. We give these coordinates a name:
\begin{definition}\label{def:JetBundlesCoordinates}
The coordinates 
\[ (x,y,z) := (x,y=z^0,z^1,\dots,z^r), \qquad z^{r'} := \{z_j^{(I)} \mid |I| = r', 1 \leq j \leq \ell\}, \]
in $J^r(Y)$ are said to be \textbf{standard}.
\end{definition}

The monomials above with $|I|=r'$ form a basis of $\Sym^{r'}(B,F)$, the space of a symmetric tensors with $r'$ entries in $B$ and values in $F$. This leads us to identify:
\[ J^r(B,F) = B \times F \times \Hom(B,F) \times \Sym^2(B,F) \times \dots \times \Sym^r(B,F). \]
In particular, $\pi_{r,r-1}$ is an affine bundle with fibres modelled on $\Sym^r(B,F)$.

\subsection{The Cartan distribution}

The \textbf{Cartan distribution} $\xi_\can$ in $J^r(Y)$ is uniquely defined by the following universal property: A section $X \to J^r(Y)$ is tangent to $\xi_\can$ if and only if it is holonomic. The subbundle $V_\can := \ker(d\pi_{r,r-1}) \subset \xi_\can$ is called the \textbf{vertical distribution}.

We also introduce the notation $\SF_\can := \ker(d\pi_f) \subset TJ^r(Y)$ to denote the subbundle tangent to the fibres of the front projection. The following equality holds: $\xi_\can \cap \SF_\can = V_\can$.

\subsubsection{Submanifolds}

A submanifold of $J^r(Y)$ tangent to $\xi_\can$ is said to be \textbf{integral}. These are the objects we are interested in studying.

By construction, the image of a holonomic section $j^rf: X \to J^r(Y)$ is an integral submanifold everywhere transverse to $V_\can$. These are, in fact, the easiest integral manifolds one can deal with, since the closed constraint of being tangent is automatic from the exactness condition of being holonomic. In particular, do note that any integral manifold $C^1$-close to the image of $j^rf$ is still the graph of a holonomic section.

\subsubsection{Local coordinates} \label{sssec:CartanInCoordinates}

In terms of the standard coordinates $(x,y,z) \in J^r(B,F)$ defined above, the holonomic lift of a map $f: B \to F$ reads:
\begin{align*}
j^rf: B \quad\to\quad & J^r(B,F) =\quad   B \times F \times  \Hom(B,F) \times       \Sym^2(B,F) \times     \dots \times  \Sym^r(B,F), \\
x       \quad\to\quad & j^rf(x) \quad=\quad   (x,      y = f(x), z^1 = (\partial f)(x), z^2 = (\partial^2 f)(x), \dots,        z^r = (\partial^r f)(x)).
\end{align*}
That is, a holonomic section satisfies the relations
\[ z_j^{(I)}(x) = (\partial^{I} y_j)(x), \qquad I = (i_1,\dots,i_n),\, 0 \leq |I| \leq r,\, 1 \leq j \leq \ell. \]
Equivalently, the tautological distribution $\xi_\can$ is the simultaneous kernel of the \textbf{Cartan $1$-forms}:
\begin{equation} \label{eq:Cartan1Forms}
\alpha_j^I = dz_j^{(I)} - \sum_{a=1}^n z_j^{(i_1,\cdots,i_a+1,\cdots,i_n)} dx_a, \qquad I = (i_1,\dots,i_n),\, 0 \leq |I| < r, \, 1 \leq j \leq \ell.
\end{equation}
A submanifold $N$ of $J^r(B,F)$ is integral if and only if these forms restrict as zero to $N$.

\subsection{Jet spaces of submanifolds} \label{ssec:jetsSubmanifolds}

Let $Y$ be a smooth manifold and fix an integer $n < \dim(Y) = n + \ell$. We say that two $n$-submanifolds have the same \textbf{$r$-jet} at $p \in Y$ if they are tangent at $p$ with multiplicity $r$. We denote the space of $r$-jets of $n$-submanifolds as $J^r(Y,n)$. We have, just like in the case of sections, a sequence of forgetful projections
\[ \pi_{r,r'}: J^r(Y,n) \to J^{r'}(Y,n), \]
with $\pi_f := \pi_{r,0}$ the \textbf{front projection}. There is no base projection, as there is no base manifold.

The \textbf{holonomic lift} of an $n$-submanifold $X \subset Y$ is the submanifold $j^rX \subset J^r(Y,n)$ consisting of all the $r$-jets of $X$ at each of its points. The \textbf{Cartan distribution} $\xi_\can \subset TJ^r(Y,n)$ in $J^r(Y,n)$ is the smallest subbundle that is tangent to every holonomic lift.

\begin{remark}
If $\ell=1$ and $r=1$, the structure we just constructed is precisely the \textbf{space of contact elements}. In general, the space $J^1(Y,n)$ is the Grassmannian of $n$-planes $\Gr(TY,n)$. \hfill$\triangle$
\end{remark}

\subsection{Automorphisms} \label{ssec:automorphisms}

The most general notion of automorphism in $(J^r(Y),\xi_\can)$ is that of a \emph{contact transformation}, meaning a $\xi_\can$-preserving diffeomorphism. A more restrictive notion of symmetry is the following:
\begin{definition}\label{def:point_symmetry}
Let $Y \to X$ be a submersion. Let $\Psi: Y \to Y$ be a fibre-preserving diffeomorphism lifting a diffeomorphism $\psi: X \to X$. The \textbf{point symmetry} lifting $\Psi$ is defined as:
\begin{align*}
j^r\Psi: (J^r(Y),\xi_\can) \quad\to\quad & (J^r(Y),\xi_\can) \\
j^rf(x)                   \quad\to\quad & (j^r\Psi)(j^rf(x)) := j^r(\Psi \circ f \circ \psi^{-1})(\psi(x)).
\end{align*}
\end{definition}
Point symmetries form a subgroup of the group of contact transformations. It is well-known in Contact Geometry that the space of contact transformations of $J^1(X,\R)$ is strictly larger than the space of point symmetries. However, we recall the following classic fact (see for instance \cite[Chapter VI]{KrLyVi}):
\begin{lemma} \label{lem:automorphismsAreLifts}
Assume $r>1$ or $\ell > 1$. Any contact transformation of $J^r(Y)$ is the lift of a contact transformation of $J^{r-1}(Y)$.

In particular, every contact transformation is a point symmetry if $\ell > 1$.
\end{lemma}
The idea behind the proof is that, as long as $r'>1$ or $r=1$ if $\ell>1$, the fibres of $\pi_{r,r'}$ are subbundles of $TJ^r(Y)$ intrinsically associated to $\xi_\can$ as a distribution.

\subsubsection{Morphisms}

More generally, we fix submersions $Y \to X$ and $Y' \to X'$ and open subsets $A \subset J^r(Y)$ and $A' \subset J^r(Y')$. Suppose $\dim(Y) = \dim(Y')$ and $\dim(X) = \dim(X')$. A map $A \to A'$ is said to be \textbf{isocontact} if it preserves $\xi_\can$.

Similarly, let $\Psi: Y \to Y'$ be an embedding lifting a mapping $\psi: X \to X'$. Then, $\Psi$ lifts to an isocontact map
\begin{align*}
j^r\Psi: (J^r(Y),\xi_\can) \quad\to\quad & (J^r(Y'),\xi_\can) \\
j^rf(x)                   \quad\to\quad & j^r(\Psi \circ f \circ \psi^{-1})(\psi(x)),
\end{align*}
that we call the \textbf{point symmetry} associated to $\Psi$.

\subsubsection{Jet spaces of submanifolds}

Suppose $Z$ is a smooth manifold and $X \subset Z$ is a submanifold. Given an embedding of the tubular neighbourhood $\nu(M) \to M$ of $M$ into $Z$, we can restrict our attention to those submanifolds with image in $\nu(M)$ that are graphical over $M$. We recall the following useful result:
\begin{lemma} \label{lem:zeroSection}
There is an isocontact embedding $(J^r(\nu(M)),\xi_\can) \to (J^r(Z,n),\xi_\can)$ acting on $M$ as the identity.
\end{lemma}
That is, jet spaces of submanifolds are locally modelled on jet spaces of sections.

\subsubsection{Front symmetries}

The notion of point symmetry does not make sense for jet spaces of submanifolds, but the isocontact embedding produced in Lemma \ref{lem:zeroSection} is nonetheless special: it commutes with the front projection. This motivates us to consider the following notion: Let $A$ and $A'$ be open subsets in respective jet spaces (of either sections or submanifolds) of the same dimension. An isocontact map $A \to A'$ lifting a map $\pi_f(A) \to \pi_f(A')$ between the fronts is said to be a \textbf{front symmetry}.

By construction, every point symmetry is a front symmetry and every front symmetry is an isocontact map.

\subsubsection{Models around holonomic sections} \label{sssec:zeroSection}

A useful corollary of Lemma \ref{lem:zeroSection} is the following: Given a submersion $Y \to X$ and a holonomic section $j^rf: X \to J^r(Y)$, we can consider the normal bundle $\nu(f) \subset Y$ of $\Image(f)$. We see $\Image(f) \cong X$ as the zero section in $\nu(f) \to \Image(f)$. The Lemma yields then a point symmetry $J^r(\nu(f)) \to J^r(Y)$ mapping the zero section to $j^rf$. This is useful in order to carry out local manipulations of the integral submanifold $\Image(j^rf)$, as we will see in Subsection \ref{ssec:principalProjection}.

%% file: PaperI-hPrinciple.tex
\section{The $h$-principle} \label{sec:hPrinciple}

The $h$-principle is a collection of techniques and heuristic approaches whose purpose is to describe spaces of solutions of partial differential relations. This Section provides a quick overview, and readers familiar with $h$-principles are invited to skip ahead.

In Subsection \ref{ssec:PDRs} we review the notion of differential relation. Then we go over two classic $h$-principle techniques: \emph{holonomic approximation} in Subsection \ref{ssec:holonomicApprox} and \emph{triangulations in general position} in Subsection \ref{ssec:Thurston}. Both will be used in the proofs of our main results.

For a panoramic view of $h$-principles we refer the reader to the two standard texts \cite{ElMi} and \cite{Gr86} (which we suggest to check in that order).

\subsection{Differential relations} \label{ssec:PDRs}

Let $Y \to X$ be a submersion. A \textbf{partial differential relation} (PDR) of order $r$ is a subset $\SR \subset J^r(Y)$. This provides a framework for PDRs of sections, but one can define PDRs of $n$-submanifolds of $Y$ as subsets of $J^r(Y,n)$ as well.

Endow $\Gamma(J^r(Y))$ with the weak $C^0$-topology. We may use the inclusion
\[ j^r: \Gamma(Y) \quad\to\quad \Gamma(J^r(Y)), \]
to pull it back and endow the domain with its usual weak $C^r$-topology. This makes $j^r$ a continuous map. We write $\Sol^f(\SR)$ for the subspace of sections in $\Gamma(J^r(Y))$ with image in $\SR$, i.e. the space of formal solutions. Similarly, we write $\Sol(\SR)$ for the space of solutions, which is a subspace of $\Gamma(Y)$. 
\begin{definition}
We say that the \textbf{(complete) $h$-principle} holds for $\SR$ if the inclusion
\begin{align*}
\iota_\SR: \Sol(\SR) \quad\to\quad & \Sol^f(\SR) \\
f                    \quad\to\quad & \iota_\SR(f) := j^rf
\end{align*}
is a weak homotopy equivalence.
\end{definition}

\subsubsection{Flavours of $h$-principle} \label{sssec:flavoursHPrinciple}

The $h$-principle is \textbf{relative in the domain} when the following property holds: Any family of formal solutions of $\SR$, which are already honest solutions in a neighbourhood of a closed set $A$, can be homotoped to become solutions over the whole of $U$ while remaining unchanged over $\Op(A)$.

Similarly, the $h$-principle is \textbf{relative in the parameter} when: Any family of formal solutions $\{F_k\}_{k \in K}$, parametrised by a closed manifold $K$, and with $F_{k'}$ holonomic for every $k'$ in an open neighbourhood of a fixed closed subset $K' \subset K$, can be homotoped to be holonomic relative to $\Op(K')$.

\subsection{Holonomic approximation}\label{ssec:holonomicApprox}

One of the cornerstones of the classical theory of $h$-principles is the holonomic approximation theorem. It states that any formal section of a jet bundle can be approximated by a holonomic one in a neighbourhood of a perturbed CW-complex of codimension at least $1$. The precise statement reads as follows:
\begin{theorem}[\cite{ElMi}] \label{thm:holonomicApprox}
Let $Y \to X$ be a fiber bundle, $K$ a compact manifold, $A \subset M$ a polyhedron of positive codimension, and $(F_{k,0})_{k \in K}: X \to J^r(Y)$ a family of formal sections. Then, for any $\varepsilon >0$ there exists
\begin{itemize}
\item a family of isotopies $(\phi_{k,\ss})_{\ss \in [0,1]}: X \to X$,
\item a homotopy of formal sections $(F_{k,\ss})_{k \in K, \ss \in [0,1]}: X \to Y$,
\end{itemize}
satisfying:
\begin{itemize}
\item $F_{k,1}$ is holonomic in $\Op(\phi_{k,1}(A))$,
\item $|\phi_{k,\ss} - \id|_{C^0} < \varepsilon$ and is supported in a $\varepsilon$-neighbourhood of $A$,
\item $|F_{k,\ss} - F_{k,0}|_{C^0} < \varepsilon$.
\end{itemize}
Moreover the following hold:
\begin{itemize}
\item If $V \in \X(\Op(A))$ is a vector field transverse to $A$, then we can arrange that $\phi_{k,\ss}$ is a flow tangent to the flowlines of $V$, for all $\ss$ and $k$.
\item If the $F_{k,\ss}$ are already holonomic in a neighborhood of a subcomplex $B \subset A$, then we can take $F_{k,\ss} = F_{k,0}$ and $\phi_{k,\ss} = \id$ on $\Op(B)$, for all $k$.
\item If $F_{k,\ss}$ is everywhere holonomic for every $k$ in a neighbourhood of a CW-complex $K' \subset K$, then we can take $F_{k,\ss} = F_{k,0}$ and $\phi_{k,\ss} = \id$ for $k \in \Op(K')$.
\end{itemize}
\end{theorem}
\begin{remark}
Note that in the above statement, the inequalities
\[ |\phi_{k,\ss} - \id|_{C^0} < \varepsilon,\quad |F_{k,\ss} - F_{k,0}|_{C^0} < \varepsilon,\]
depend on a choice of Riemannian metric on $X$ and $Y$. \hfill$\triangle$
\end{remark}
For the proof and a much longer account of its history, we refer the reader to \cite{ElMi}. Essentially, this theorem recasts the method of flexible sheaves due to M. Gromov (itself a generalisation of the methods used by S. Smale in his proof of the sphere eversion and the general $h$-principle for immersions) in a different light. Let us go over the statement. 

The starting point is the family of formal sections $F_{k,0}$, which we want to homotope until they become holonomic. This is not possible, but the theorem tells us that at least we can achieve holonomicity in a neighbourhood of a set of positive codimension. We are not allowed to fix this set. Instead, we begin with a polyhedron $A$, which we deform in a $C^0$ small way to yield an isotopic polyhedron $\phi_{k,1}(A)$. This isotopy occurs in the normal directions of $A$ (which we may prefix by taking a transverse vector field $V$), and essentially produces a copy $\phi_{k,1}(A)$ of $A$ of greater length. This process is called, descriptively, \textbf{wiggling}. The room we gain by wiggling is what allows us to achieve holonomicity: the main idea is that, at each point $p \in A$, we approximate $F_{k,0}$ by the corresponding Taylor polynomial $F_{k,0}(p)$ and then we use the directions normal to $A$ to interpolate between these polynomials keeping control of the derivatives. Hence, we can take the $F_{k,\ss}$ to be arbitrarily close to our initial data, and the wiggling to be $C^0$-small. However, if we desire better $C^0$-bounds, we will be forced to wiggle more aggressively, i.e. the isotopies $\phi_{k,\ss}$ will become $C^1$-large.

\subsection{Thurston's triangulations} \label{ssec:Thurston}

An important step in the application of many $h$-principles (including ours), is the reduction of the global statement (global in the manifold $M$), to a local statement taking place in a small ball. These reductions allow us not to worry about (global) topological considerations, making the geometric nature of the arguments involved more transparent. Working on small balls (i.e. ``zooming-in'') usually has the added advantage of making the geometric structures we consider seem ``almost constant''; this will play a role later on.

A possible approach to achieve this is to triangulate the ambient manifold $M$ and then work locally simplex by simplex. A small neighbourhood of a simplex is a smooth ball which can be assumed to be arbitrarily small if the subdivision is sufficiently fine; thus, this achieves our intended goal. When we deal with parametric results, we want to zoom-in in the parameter space $K$ too. This requires us to triangulate in parameter directions as well (in a manner adapted to the projection $K \times M \to K$).

Let $(M,\SF)$ be a manifold of dimension $m=n+k$ endowed with a foliation of rank $n$ . Given a triangulation $\ST$, we write $\ST^{(i)}$ for the collection of $i$-simplices, where $i=0,\dots,m$. We think of each $i$-simplex $\sigma \in \ST^{(i)}$ as being parametrised $\sigma: \Delta^i \to M$, where the domain is the standard simplex in $\R^i$. The parametrisation $\sigma$ allows us to pull-back data from $M$ to $\Delta^i$. In particular, if $\sigma$ is a top-dimensional simplex, it is a diffeomorphism with its image and we may assume that $\sigma$ extends to an embedding $\Op(\Delta^n) \to M$ of a ball.

If the image of $\sigma$ is sufficiently small, we would expect that the parametrisation $\sigma$ can be chosen to be reasonable enough so that $\sigma^*\SF$ is almost constant. This can be phrased as follows: 
\begin{definition}
A top-dimensional simplex $\sigma$ is in \textbf{general position} with respect to the foliation $\SF$ if the linear projection 
\[ \Delta^m/(\sigma^*\SF)_p \to \R^k \]
restricts to a map of maximal rank over each subsimplex of $\sigma$; here we use the identification $T_p\R^m = \R^m$. In particular, $\sigma^*\SF$ is transverse to each subsimplex.

The triangulation $\ST$ is in \textbf{general position} with respect to $\SF$ if all of its top-simplices are in general position.
\end{definition}

\begin{theorem} \label{thm:Thurston}
Let $(M,\SF)$ be a foliated manifold. Then, there exists a triangulation $\ST$ of $M$ which is in general position with respect to $\SF$.
\end{theorem}
This statement was first stated and proven by W. Thurston in \cite{Th2,Th1}, playing a central role in his $h$-principles for foliations.


%% file: PaperI-SingularitiesPrelim.tex
\section{Some singularity theory} \label{sec:singularitiesPrelim}

The motto behind the wrinkling approach to $h$-principles is that, as long as there are no homotopical obstructions, we can restrict our attention to maps with simple singularities. The precise meaning of ``simplicity'' depends on the problem at hand. In this paper, we will study integral maps into jet space whose singularities are indeed simple.

Before we get there, we need to review some of the basics on Singularity Theory. In Subsection \ref{ssec:ThomBoardman} we define singularities of mapping and of tangency and we introduce the Thom-Boardman hierarchy. In Subsection \ref{ssec:equivalenceSmooth} we recall the notion of (left-right) equivalence for smooth maps. We then discuss equivalence for integral maps in Subsection \ref{ssec:equivalenceIntegral}; this relies on the various notions of jet space morphisms introduced in Subsection \ref{ssec:automorphisms}. 

All this background will allow us to introduce the singularity models relevant for our $h$-principles in Section \ref{sec:holonomicApproxZZ}.

\subsection{The Thom-Boardman hierarchy} \label{ssec:ThomBoardman}

\subsubsection{Singularities of tangency} \label{sssec:ThomBoardman}

Let $M$ and $N$ be manifolds of dimension $a$ and $b$, respectively. Suppose $N$ is endowed with a foliation $\SF$ of rank $c$, and let $f:M \to N$ be an immersion. A point $p \in M$ is a \textbf{singularity of tangency} with respect to $\SF$ if $d_pf(TM)$ and $\SF_{f(p)}$ are not transverse to one another as linear subspaces of $T_{f(p)}N$.

We define the \textbf{locus of singularities of tangency of corank $j$}
\[ \Sigma^j(f,\SF) := \{p \in M \,\mid\, \dim(df(T_pM) \cap \SF_{f(p)}) - \max(a+c-b,0) = j\} \]
as the set of points where the dimension of the intersection $df(T_pM) \cap \SF_p$ surpasses the transverse case by $j$. 

Assuming that $\Sigma^j(f,\SF)$ is a submanifold, one can recursively define higher tangency loci of corank $J = j_0,\dots,j_l$ by setting
\[ \Sigma^J(f,\SF) := \Sigma^{j_l}(f|_{\Sigma^{j_0 j_1 \cdots j_{l-1}}(f,\SF)},\SF) \subset M.\]
Thom \cite{Th55} and Boardman \cite{Bo67} proved that one may perturb $f$ so that all the $\Sigma^J(f,\SF)$ are smooth submanifolds of appropriate dimensions forming a stratification $\SS$ of $M$. One should think of $\SS$ as the pullback along $f$ of the universal stratification of $\Gr(TN,a) \to N$ defined by intersection with $\SF$.

\subsubsection{Singularities of mapping}

One may similarly consider the locus of singularities of mapping of a map $g: M \to N$ given by:
\[ \Sigma^j(g) := \{p \in M \,\mid\, \dim(dg(T_pM)) - \min(a,b) = j\}. \]
As well as the recursively defined $\Sigma^J(f) := \Sigma^{j_l}(g|_{\Sigma^{j_0 j_1 \cdots j_{l-1}}(f)})$.

It can be checked that the singularities of tangency of a map $f: M \to (N,\SF)$ correspond to the singularities of mapping of $\pi \circ f: M \to N/\SF$, where $\pi: N \to N/\SF$ is the projection to the leaf space (locally-defined on foliation charts).

\subsection{Equivalence in the smooth setting} \label{ssec:equivalenceSmooth}

Two maps are equivalent if they agree up to the action of the diffeomorphism groups of the source and the target. In detail:
\begin{definition} \label{def:equivalenceSmooth}
Fix manifolds $X$, $X'$, $Y$, and $Y'$, as well as subsets $A \subset X$ and $A' \subset X'$ (usually submanifolds, possibly with boundary).

Two maps $f: X \to Y$ and $g: X' \to Y'$ are said to be \textbf{equivalent} along $A$ and $A'$ if there are diffeomorphisms $\phi: \Op(A) \to \Op(A')$ and $\Phi: \Op(f(A)) \to \Op(g(A'))$ such that:
\begin{itemize}
\item $\phi$ restricts to a homeomorphism $A \to A'$.
\item $g \circ \phi = \Phi \circ f$.
\end{itemize}
\end{definition}
Equivalence implies that $\dim Y = \dim Y'$, $\dim X = \dim X'$, and (if it makes sense) that $\dim A = \dim A'$.

\subsubsection{Equivalence in the fibered setting}

We fix smooth maps $M \to K$, $Z \to K$, $M' \to K'$ and $Z' \to K'$, and subsets $A \subset M$ and $A' \subset M'$. Two maps $f: M \to Z$ and $g: M' \to Z'$ are equivalent along $A$ and $A'$, in a fibered manner over $K$ and $K'$, if there are:
\begin{itemize}
\item A diffeomorphism $\phi: \Op(A) \to \Op(A')$ identifying $A$ with $A'$,
\item and a diffeomorphism $\Psi: \Op(f(A)) \to \Op(g(A'))$,
\item both of them fibered, meaning that they lift the same (locally-defined) diffeomorphism from $K$ to $K'$,
\end{itemize}
such that $g \circ \phi = \Phi \circ f$.

In general, we do not require the maps to $K$ and $K'$ to be fibrations, although this will be the case in the parametric setting. In that case, $K$ and $K'$ will play the role of parameter spaces and the fibrations will be trivial. Allowing more general maps covers other cases of interest as well. For instance, our statements in Section \ref{sec:holonomicApproxZZ} deal with maps $M \to Y$, where $Y \to X$ is a fibre bundle. We think of these as maps into the front projection of $J^r(Y)$. We will characterise the singularities of such maps up to fibered equivalence over the base.

\subsection{Equivalence in the integral setting} \label{ssec:equivalenceIntegral}

In Subsection \ref{ssec:automorphisms} we discussed isocontact maps, front symmetries, and point symmetries. Each of these definitions leads to a different notion of equivalence for integral maps. The setting is as follows: We fix manifolds $M$ and $M'$, subsets $A \subset M$ and $A' \subset M'$ (usually submanifolds, possibly with boundary), jet spaces (of sections or submanifolds) $B$ and $B'$ of the same dimension, and integral maps $f: M \to B$ and $g: M' \to B'$.

\begin{definition} \label{def:equivalenceContact}
The maps $f$ and $g$ are \textbf{contact equivalent} along $A$ and $A'$ if there are
\begin{itemize}
\item a diffeomorphism $\phi: \Op(A) \to \Op(A')$,
\item and an isocontact embedding $\Psi: \Op(f(A)) \to \Op(g(A'))$, 
\end{itemize}
such that $\Psi \circ f = g \circ \phi$.
\end{definition}

If we require $\Psi$ to be a front symmetry, we obtain the following definition:
\begin{definition} \label{def:equivalenceFront}
The maps $f$ and $g$ are \textbf{front equivalent} along $A$ and $A'$ if their front projections are equivalent.
\end{definition}

Lastly, if we ask $\Psi$ to be a point symmetry:
\begin{definition} \label{def:equivalencePoint}
Suppose $B$ and $B'$ are jet spaces of sections with base manifolds $X$ and $X'$, respectively. The maps $f$ and $g$ are \textbf{point equivalent} along $A$ and $A'$ if their front projections are equivalent in a fibered manner over $X$ and $X'$.
\end{definition}

In this paper we will be concerned with the last two definitions. The first one will play a more central role in the sequel \cite{PT2}.

\subsubsection{Equivalence for families} \label{sssec:equivalenceParam}

Let $M$, $M'$, $B$ and $B'$ as above. Fix additionally compact manifolds $K$ and $K'$ serving as parameter spaces, subsets $A \subset K \times M$ and $A' \subset K' \times M'$, and families of integral maps 
\[ f = (f_k)_{k \in K}: M \to B, \qquad g = (g_{k'})_{k' \in K'}: M' \to B'. \]
We think of $f$ as a map $K \times M \to K \times B$ fibered over $K$. Similarly, we think of $g$ as a fibered-over-$K'$ map.

The maps $f$ and $g$ are front equivalent if their front projections are equivalent in a fibered manner over $K$ and $K'$.

Similarly, they are point equivalent if their front projections are equivalent in a fibered manner simultaneously with respect to the pair $(K,K')$ and the pair $(K \times X,K' \times X')$. Here $X$ is the base of $B$ and $X'$ the base of $B'$.

We say that $f$ and $g$ are contact equivalent along $A$ and $A'$ if there are fibered diffeomorphisms
\begin{itemize}
\item $\phi: \Op(A) \to \Op(A')$ identifying $A$ with $B$,
\item and $\Psi: \Op(f(A)) \to \Op(g(B))$ that is fibrewise isocontact,
\end{itemize}
such that $\Psi \circ f = g \circ \phi$.

%% file: PaperI-Lifting.tex
%
%
%
%
%
%

\section{Metasymplectic projections} \label{sec:metasymplectic}

In Contact Topology it is standard to manipulate legendrian submanifolds using the front and lagrangian projections. In this section we introduce an analogue of the latter for general jet spaces and we explain how it can be used to construct and deform integral submanifolds.

We denote $\dim(X) = n$ and $\dim(Y) = k$, where $J^r(Y) \to X$ is the jet space of interest. Our manipulations will be carried out locally in charts, so we introduce vector spaces $B$ and $F$ as local replacements of $X$ and the fibres of $Y$, respectively. Thus, we work in $J^r(B,F)$, which we endow with standard coordinates $(x,y,z)$.

We will project $J^r(B,F)$ to so-called \emph{standard metasympletic space}. Morally speaking, this amounts to projecting to $\xi_\can$ endowed with its curvature (seen as a vector-valued 2-form). This is explained in Subsection \ref{ssec:metasymplectic}. As far as the authors are aware, the first explicit reference to metasymplectic space appeared in \cite{Ly80b}, although it was probably known earlier to some experts in geometric PDEs and Differential Topology.

In Subsection \ref{ssec:integralLift} we prove Proposition \ref{prop:integralLift}: exact isotropic submanifolds in standard metasympletic space can be uniquely lifted to integral submanifolds of $J^r(B,F)$. This is sufficient to manipulate 1-dimensional integral submanifolds; see Subsection \ref{ssec:liftingCurves}.

For higher-dimensional integral submanifolds the story is more complicated, because it is non-trivial to manipulate their metasymplectic projections directly. To address this, we work ``\emph{one direction at a time}'', effectively thinking about them as parametric families of curves. This is done in Subsections \ref{ssec:principalProjection} and \ref{ssec:principalProjectionGeneral}.

In Section \ref{sec:singularities} we will introduce concrete metasymplectic projection models that we then translate to the front projection.

\subsection{Standard metasymplectic space} \label{ssec:metasymplectic}

Recall the Cartan $1$-forms defining $\xi_\can$, as introduced in Subsection \ref{sssec:CartanInCoordinates}:
\[ \alpha_j^I = dz_j^{(I)} - \sum_{a=1}^n z_j^{(i_1,\cdots,i_a+1,\cdots,i_n)} dx_a, \qquad I = (i_1,\dots,i_n),\, |I| = r-1, \, 1 \leq j \leq k. \]
Since we focus on $|I| = r-1$, these forms only depend on the coordinates $z^r$. Their differentials are the $2$-forms:
\[ \Omega_j^I = \sum_{a=1}^n dx_a \wedge dz_j^{(i_1,\cdots,i_a+1,\cdots,i_n)} , \qquad I = (i_1,\dots,i_n),\, |I| = r-1, \, 1 \leq j \leq k, \]
which, by construction, are pullbacks of forms in the product $B \oplus \Sym^r(B,F)$ (which have the same coordinate expression, so we abuse notation and denote them the same). We package all these $2$-forms as follows, where $I$ and $j$ index the coordinates on $\Sym^r$:
\begin{definition} \label{def:metasymplectic}
The  \textbf{standard metasymplectic structure} in $B \oplus \Sym^r(B,F)$ is the $2$-form:
\[ \Omega_\can := (\Omega_j^I)_{|I| = r-1, \, 1 \leq j \leq k}: \quad \wedge^2(B \oplus \Sym^r(B,F)) \quad\longrightarrow\quad \Sym^{r-1}(B,F). \]
The pair $(B \oplus \Sym^r(B,F),\Omega_\can)$ is called \textbf{standard metasymplectic space}.
\end{definition}

We remark that we can regard standard metasymplectic space as a vector space endowed with a (vector-valued) linear 2-form, or as a manifold endowed with a translation-invariant differential 2-form. The tangent fibres of the latter are isomorphic to the former. We readily check:
\begin{lemma} \label{lem:metasymplecticStructure}
Fix a point $p \in B \oplus \Sym^r(B,F)$ and vectors $v_i+A_i \in T_p(B \oplus \Sym^r(B,F)) \cong B \oplus \Sym^r(B,F)$. Then:
\[ \Omega_\can(v_0+A_0,v_1+A_1) = \iota_{v_0} A_1 - \iota_{v_1} A_0. \]
\end{lemma}
I.e. the standard metasymplectic structure is precisely the contraction map of tensors with vectors. When $r=k=1$, the standard metasymplectic space $(B \oplus B^*,\Omega_\can)$ is simply $\R^{2n}$ endowed with its linear symplectic form.

\subsubsection{The metasymplectic projection}

We then generalise the lagrangian projection:
\begin{definition} \label{def:metasymplecticProjection}
The \textbf{metasymplectic projection} is the map
\begin{align*}
\pi_\meta: J^r(B,F) \quad\longrightarrow\quad & B \oplus \Sym^r(B,F) \\
(x,y,z)         \quad\mapsto\quad & \pi_\meta(x,y,z) := (x,z^r).
\end{align*}
\end{definition}
By construction, the differential at each point
\[ d_p\pi_\meta: T_pJ^r(B,F) \quad\longrightarrow\quad T_{\pi_\meta(p)} (B \oplus \Sym^r(B,F)) \]
is an epimorphism that restricts to an isomorphism $(\xi_\can)_p \to T_{\pi_\meta(p)} (B \oplus \Sym^r(B,F))$. Furthermore, using the duality between distributions and their annihilators, it readily follows that:
\begin{lemma} \label{lem:metasymplecticProjection}
The differential is an isomorphism of metasymplectic linear spaces:
\[ d_p\pi_\meta: ((\xi_\can)_p,\Omega(\xi_\can)) \to (T_{\pi_\meta(p)} (B \oplus \Sym^r(B,F)),\Omega_\can), \]
where $\Omega(\xi_\can)$ is the curvature of $\xi_\can$.
\end{lemma}
We note that $d_p\pi_\meta$ identifies the vertical bundle $(V_\can)_p \subset (\xi_\can)_p$ with $\Sym^r(B,F)$. In light of this, we will say that the directions in $B \oplus \Sym^r(B,F)$ contained in $\Sym^r(B,F)$ are \textbf{vertical}.

\subsection{Isotropic submanifolds and integral lifts} \label{ssec:integralLift}

A vector subspace $V$ of the standard linear metasymplectic space is said to be an \textbf{isotropic element} if $(\Omega_\can)|_V = 0$. An isotropic element is maximal if it is not contained in a larger isotropic subspace. Similarly, a submanifold of standard metasymplectic space is \textbf{isotropic} if all its tangent subspaces are isotropic elements. 
\begin{corollary} \label{lem:lagrangianProjectionImmersion}
Let $f: N \to J^r(B,F)$ be a map. Then:
\begin{itemize}
\item $f$ is integral if and only if $\pi_\meta \circ f$ is isotropic.
\item Suppose $f$ is integral. Then, $f$ is an immersion if and only if $\pi_\meta \circ f$ is an immersion.
\end{itemize}
\end{corollary}
We now explore the converse: how to lift isotropic submanifolds to integral ones.

\subsubsection{The standard Liouville form}

First we need an auxiliary concept:
\begin{definition} \label{def:Liouville}
The \textbf{standard Liouville form}
\[ \lambda_\can \in \Omega^1(B \oplus \Sym^r(B,F);  \Sym^{r-1}(B,F)) \]
is defined, at a point $(v,A)$ in standard metasymplectic space, by the following tautological expression:
\[ \lambda_\can(v,A)(w,B) := -\iota_{w}A. \]
\end{definition}

From our explicit description of $\Omega_\can$ it follows that:
\begin{lemma} \label{lem:Liouville}
Then following statements hold:
\begin{itemize}
\item The Liouville form can be explicitly written as:
\[ \lambda_\can(x,z^r) = \left( -\sum_{a=1}^n z_j^{(i_1,\cdots,i_a+1,\cdots,i_n)} dx_a \right)_{|(i_1,\cdots,i_a,\cdots,i_n)| = r-1}. \]
\item The Cartan $1$-forms $\alpha^r \in \Omega^1(J^r(B,F); \Sym^{r-1}(B,F))$ are given by the expression
\[ \alpha_r(x,y,z) = dz_{r-1} + \lambda_\can(x,z^r). \]
\item In particular, $d\lambda_\can = \Omega_\can$.
\end{itemize}
\end{lemma}
That is, the familiar properties for the Liouville form in the symplectic/contact setting hold as well in more general jet spaces.

\subsubsection{Exact isotropics}

In the contact/symplectic setting, it is possible to produce a lift whenever the isotropic submanifold in question is exact. In this generalised setting, we need exactness at every step. This is guaranteed whenever the submanifold to be lifted is contractible, which is enough for our purposes:
\begin{proposition} \label{prop:integralLift}
Let $N$ be a disc. Given an isotropic map
\[ g: N \to (B \oplus \Sym^r(B,F),\Omega_\can) \] 
there exists an integral map
\[ \Lift(g): N \to J^r(B,F) \]
satisfying $\pi_\meta \circ \Lift(g) = g$. The lift $\Lift(g)$ is unique once we fix $\Lift(g)(x_0)$ for some $x_0 \in N$.
\end{proposition}
\begin{proof}
Write $g(p) = (x(p),z^r(p))$. By construction, $g^*\Omega_\can = 0$. We deduce that each component of $g^*\lambda_\can$ is closed and thus exact. We choose primitives, which we denote suggestively by 
\[ z^{r-1}: N \to \Sym^{r-1}(B,F). \]
These functions are unique once their value at the point $x_0$ is given.

We put $g$ together with the chosen primitives to produce a map
\[ h := (x,z^{r-1},z^r): N \to B \oplus \Sym^{r-1}(B,F) \oplus \Sym^r(B,F). \]
We can readily check, using Lemma \ref{lem:Liouville}, that, by construction:
\[ h^*\alpha^r = dz^{r-1} + g^*\lambda_\can = 0. \]

Consider now the $2$-form in $B \oplus \Sym^{r-1}(B,F) \oplus \Sym^r(B,F)$ with values in $\Sym^{r-2}(B,F)$:
\[ \Omega_\can^{r-1} :=  \left( \sum_{a=1}^n dx_a \wedge dz_j^{(i_1,\cdots,i_a+1,\cdots,i_n)} \right)_{|(i_1,\cdots,i_a,\cdots,i_n)| = r-2}. \]
It involves differentials on the coordinates $(x,z^{r-1})$ only and its pullback to $J^r(B,F)$ is the curvature of $\xi_\can^{(1)}$. Using our definition of $z^{r-1}$, and the fact that cross derivatives agree, we deduce:
\[ h^*\Omega_\can^{r-1} = h^*\left( -\sum_{a,b=1}^n z_j^{(i_1,\cdots,i_a+1,\cdots,i_b+1,\cdots,i_n)} dx_a \wedge dx_b \right) = 0. \]
This computation tells us that the map
\[ (x,z^{r-1}): N \to B \oplus \Sym^{r-1}(B,F) \]
is isotropic. Therefore, the argument can be iterated for decreasing $r$ to produce a lift.
\end{proof}

\subsection{Lifting curves} \label{ssec:liftingCurves}

Let us particularise now to the case $\dim(B) = 1$. Then, in standard coordinates $(x,y=z^0,z)$ the Cartan $1$-forms read
\[ \alpha^l = dz^l - z^{l+1}dx, \qquad l=0,\dots,r-1. \]
The particular flexibility of curves (compared to higher dimensional integral submanifolds) stems from the fact that any
\[ g(t) = (x(t),z_r(t)) : [0,1] \to B \oplus \Sym^r(B,F) \]
is automatically isotropic. Then, following the recipe given in the proof of Proposition \ref{prop:integralLift}, we solve for the $z^{r-1}$ coordinates using $\alpha^r$:
\[ g^*\alpha^r = z_{r-1}(t)dt - z_r(t)x'(t)dt \]
leading to the integral expression
\[ z_{r-1}(t) = z_{r-1}(0) + \int_0^t  z_r(s)x'(s)ds \]
which uniquely recovers $z_{r-1}$ up to the choice of lift $z_{r-1}(0)$. Proceeding decreasingly in $l$ we can solve for all the $z^l(t)$, effectively lifting $g$ to an integral curve $\Lift(g): [0,1] \to J^r(B,F)$.

According to Lemma \ref{lem:lagrangianProjectionImmersion}, the lift $\Lift(g)$ is immersed if and only if $g$ was immersed. Assuming $g$ is immersed and isotropic, the following loci are in correspondence with one another:
\begin{itemize}
\item The tangencies of $g$ with respect to the vertical directions $\Sym^r(B,F)$.
\item The tangencies of $\Lift(g)$ with respect to the vertical bundle $V_\can$.
\item The singularities of mapping of the front $\pi_f \circ \Lift(g)$.
\end{itemize}
The advantage is that the singularities of $g$ with respect to $\Sym^r(B,F)$ are easier to deal with, since $g$ is a smooth curve in metasymplectic space with no constraints.

\subsection{Principal projections in charts} \label{ssec:principalProjection}

Higher-dimensional isotropic/integral submanifolds are, generally speaking, difficult to manipulate directly due to the differential constraints they have to satisfy. However, this is not the case for submanifolds arising as the graph of a holonomic section of $J^r(B,F)$. The differential constraint associated to being integral is then an automatic consequence of being holonomic. We will now explain how a holonomic section can be manipulated to yield integral embeddings with tangencies with the vertical. These tangencies will be of corank-$1$.

\subsubsection{Principal directions}

The idea is to focus on pure derivatives of order $r$:
\begin{definition} \label{def:principalProjection}
The \textbf{principal projection} associated to the standard coordinates $(x,y,z)$ is the map:
\begin{align*}
\pi_\meta^n: J^r(B,F) \quad\to\quad & B \oplus \Sym^r(\R,F) \\
(x,y,z)           \quad\to\quad & (x,z^{(0,\dots,0,r)}).
\end{align*}
\end{definition}
The aim is to construct integral submanifolds by manipulating the pair $(x_n,z^{(0,\dots,0,r)})$. The reader can think of this as ``modifying a single pure derivative'', but not necessarily in a graphical manner over $B$.

The counterpart of modifying the pure derivative $z^{(0,\dots,0,r)}$ is freezing all other derivatives. This is often formalised as follows; see \cite[p. 170]{Gr86}. We say that two sections of $B \to F$ have the same $\bot$-jet  at $p \in B$, with respect to the principal direction $dx_n$, if, with the exception of the pure $r$-order derivative along $x_n$, their $r$-order Taylor polynomials at $p$ agree. The space of $\bot$-jets is denoted by $J^\bot(B,F)$. We write $j^\bot f: B \to J^\bot(B,F)$ for the $\bot$-jet of a section $f: B \to F$.

$\bot$-Jets play a central role in any $h$-principle relying on manipulating sections ``one derivative at a time''. The most salient example is \emph{convex integration} \cite{Gr73}, but see also \cite{Gav1} for its appearance in holonomic approximation and \cite{ElMiWrinEmb,Gav2} for its appearance in wrinkling.

\subsubsection{A lifting statement}

The upcoming result constructs integral maps that have tangencies with the vertical in the $x_n$-direction. The reader should think of it as a fibered analogue of the case of curves. In order to clarify the statement, let us run through the setup first.

We fix a map
\begin{align*}
g: B                              \quad\longrightarrow\quad & B \oplus \Sym^r(\R,F) \\
(t) = (\widetilde t, t_n) = (t_1,\dots,t_n) \quad\mapsto\quad & (\widetilde x = \widetilde t,x_n(t), z^{(0,\dots,0,r)}(t)).
\end{align*}
It is fibered over the $\widetilde t$ variables. In particular, note that the restriction of $\pi_B \circ g$ to the hypersurface $H := \{t_n = 0\}$ is an embedding into $B$ transverse to the $x_n$ direction.

Our goal is to lift $g$ to an integral map into $J^r(B,F)$. We will do this by integrating $z^{(0,\dots,0,r)}$ with respect to $x_n$ using the Liouville form, as we did for curves. This lift will be unique as long as we fix some initial datum along $H$. Namely, we assume that we are also given an integral map $h: H \to J^r(B,F)$ compatible with $g$ in the sense that $\pi_\meta^n \circ h = g|_H$. Since $h|_H$ is integral, its image is the image of a holonomic section of $J^r(\pi_B \circ g(H),F)$.

\begin{proposition} \label{prop:principalProjection}
Under the assumptions above, there exists a unique integral map 
\[ \Lift(g,h): B \to J^r(B,F) \]
that satisfies:
\begin{itemize}
\item $\pi_\meta^n \circ \Lift(g,h) = g$,
\item $\Lift(g,h)|_H = h$
\item The singularities of mapping of $\Lift(g,h)$ are in correspondence with those of $g$.
\item The singularities of tangency of $\Lift(g,h)$ with respect to the vertical are in correspondence with those of $g$.
\item $\Lift(g,h)$ is immersed if and only if $g$ is immersed.
\item $\Lift(g,h)$ is embedded if $g$ is embedded.
\end{itemize}
Furthermore, the construction of $\Lift(g,h)$ depends smoothly on $g$ and $h$.
\end{proposition}
\begin{proof}
We will prove the stated properties at the end. First we provide a recipe for the claimed integral lift $\Lift(g,h)$. Each of its entries $z^{(i_1,\dots,i_{n-1},i_n)}$ will be a function of $t$ defined uniquely from the given $x_n$, $z^{(0,\dots,0,r)}$, and $h$.

According to the Cartan $1$-form
\[ dz^{(0,\dots,0,l)} - z^{(0,\dots,0,l+1)}dx_n - \sum_{a=1}^{n-1} z^{(0,\dots,1,\dots,0,l)}dx_a, \]
the derivative of $z^{(0,\dots,0,l)}(t)$ in the direction of $t_n$ must agree with $z^{(0,\dots,0,l+1)}(t) \frac{\partial x_n}{\partial t_n}(t)$. It follows that we have to define $z^{(0,\dots,0,l)}(t)$ using the integral expression:
\[ z^{(0,\dots,0,l)}(t) \,:=\, z^{(0,\dots,0,l)}(\widetilde t, 0) +  \int_0^{t_n} z^{(0,\dots,0,l+1)}(\widetilde t,s) \dfrac{\partial x_n}{\partial t_n}(\widetilde t,s) ds, \]
where the initial value $z^{(0,\dots,0,l)}(\widetilde t,0)$ is given by the corresponding entry in $h$. We apply this process inductively for decreasing $l$. Integrating $r$ times defines $y(t) = z^{(0,\dots,0,0)}(t)$.

The remaining entries are defined, morally speaking, using differentiation with respect to the $\widetilde t$ variables. However, this has to be done taking into account the Cartan forms once again:
\[ dz^{(i_1,\dots,i_n)} - \sum_{a=1}^n z^{(i_1,\dots,i_a+1,\dots,i_n)}dx_a. \]
Their pullbacks by the lift read:
\begin{align*}
       & dz^{(i_1,\dots,i_n)} -\sum_{a=1}^{n-1} z^{(i_1,\dots,i_a+1,\dots,i_n)}dt_a - \sum_{b=1}^n z^{(i_1,\dots,i_{n-1},i_n+1)} \dfrac{\partial x_n}{\partial t_b}dt_b \\
=\quad &  dz^{(i_1,\dots,i_n)} - z^{(i_1,\dots,i_{n-1},i_n+1)} \dfrac{\partial x_n}{\partial t_n}dt_n - \sum_{a=1}^{n-1} \left( z^{(i_1,\dots,i_a+1,\dots,i_n)} + z^{(i_1,\dots,i_n+1)} \dfrac{\partial x_n}{\partial t_a} \right) dt_a.
\end{align*}
These expressions force us to define:
\[ z^{(i_1,\dots,i_a+1,\dots,i_n)}(t) \,:=\, (\partial_a z^{(i_1,\dots,i_a,\dots,i_n)})(t) - z^{(i_1,\dots,i_a,\dots,i_n+1)}(t)\dfrac{\partial x_n}{\partial t_a}(t).  \]
This definition is inductive in the size $|(i_1,\dots,i_a,\dots,i_{n-1},0)|$. Note that, a priori, it is not clear whether ``cross-derivatives with respect to the Cartan forms'' commute. We have to verify, for any two given subindices $a$ and $b$, that the inductive definition of $z^{(i_1,\dots,i_a+1,\dots,i_b+1,\dots,i_n)}$ does not depend on the order of the subindices:
\begin{align*}
z^{(i_1,\dots,i_a+1,\dots,i_b+1,\dots,i_n)} \quad=\quad &  \partial_b z^{(i_1,\dots,i_a+1,\dots,i_b,\dots,i_n)} - z^{(i_1,\dots,i_a+1,\dots,i_b,\dots,i_n+1)}\dfrac{\partial x_n}{\partial t_b} \\
\quad=\quad & \partial_b\left[\partial_a z^{(i_1,\dots,i_a,\dots,i_b,\dots,i_n)} - z^{(i_1,\dots,i_a,\dots,i_b,\dots,i_n+1)}\dfrac{\partial x_n}{\partial t_a} \right] - \\
						&	\left[ \partial_a z^{(i_1,\dots,i_a,\dots,i_b,\dots,i_n+1)} - z^{(i_1,\dots,i_a,\dots,i_b,\dots,i_n+2)}\dfrac{\partial x_n}{\partial t_a} \right]\dfrac{\partial x_n}{\partial t_b} \\
\quad=\quad & \partial_b\partial_a z^{(i_1,\dots,i_a,\dots,i_b,\dots,i_n)} - \partial_b\left(z^{(i_1,\dots,i_a,\dots,i_b,\dots,i_n+1)}\dfrac{\partial x_n}{\partial t_a}\right) - \\
						& \partial_a z^{(i_1,\dots,i_a,\dots,i_b,\dots,i_n+1)} \dfrac{\partial x_n}{\partial t_b} + z^{(i_1,\dots,i_a,\dots,i_b,\dots,i_n+2)}\dfrac{\partial x_n}{\partial t_a}\dfrac{\partial x_n}{\partial t_b} \\
\quad=\quad &  \partial_a z^{(i_1,\dots,i_a,\dots,i_b+1,\dots,i_n)} - z^{(i_1,\dots,i_a,\dots,i_b+1,\dots,i_n+1)}\dfrac{\partial x_n}{\partial t_a}.
\end{align*}
The general claim follows inductively.

By construction, the Cartan forms vanish on $\Lift(g,h)$, so the lift is integral. The argument also shows that it depends smoothly on $g$ and $h$ and is unique once these are fixed. Now we check the other properties. The first one is automatic by construction. For the second one we recall that the initial values $z^{(0,\dots,0,l)}|_H$ were given by $h$. Due to integrality, all other entries $z^{(i_1,\dots,i_n)}|_H$ of the lift are obtained from the former by ``differentiation with respect to the Cartan forms''. This applies as well to $h$, since we assumed it was integral. It follows that the two must agree. The last four properties are immediate from the fact that $(\xi_\can,\Omega(\xi_\can))$ projects isomorphically to metasymplectic space.
\end{proof}
Recall that, according to Lemma \ref{lem:zeroSection}, any holonomic section can be identified with the zero section up to a point symmetry. In applications, this will allow us to assume that $h$ is identically zero.

\subsection{Principal projections} \label{ssec:principalProjectionGeneral}

In the previous subsection we used the splitting $B = \R^{n-1} \times \R$ to implement the idea of modifying an integral map using the pure $r$-order derivative associated to the last coordinate. One can reason similarly for more general splittings and, in fact, this will be necessary later on for some of our constructions.

Let $H$ be an $(n-1)$-dimensional manifold. Let $F$ be a $\ell$-dimensional vector space. We consider the split $n$-manifold $H \times \R$. The associated space of $r$-jets splits as
\[ J^r(H \times \R,F) \,\cong\, J^\bot(H \times \R,F) \times \Sym^r(\R,F), \]
where the last term is given by the pure $r$-order derivative along the $\R$ component. This allows us to introduce the projection
\[ \pi: J^r(H \times \R,F) \to H \times \R \times \Sym^r(\R,F). \]
Then, Proposition \ref{prop:principalProjection} generalises to the following statement:
\begin{proposition} \label{prop:principalProjectionGeneral}
Fix a map 
\[ g: H \times \R \to H \times \R \times \Sym^r(\R,F) \]
fibered over $H$ and a map $h: H \times \{0\} \to J^r(H \times \R,F)$ satisfying $\pi_\meta^n \circ h = g|_{H \times \{0\}}$.

Then, there is an integral mapping
\[ \Lift(g,h): H \times \R \to J^r(H \times \R,F) \]
that satisfies:
\begin{itemize}
\item $\pi_\meta^n \circ \Lift(g,h) = g$,
\item $\Lift(g,h)|_{H \times \{0\}} = h$
\item The singularities of mapping of $\Lift(g,h)$ are in correspondence with those of $g$.
\item The singularities of tangency of $\Lift(g,h)$ with respect to the vertical are in correspondence with those of $g$.
\item $\Lift(g,h)$ is immersed if and only if $g$ is immersed.
\item $\Lift(g,h)$ is embedded if $g$ is embedded.
\end{itemize}
Furthermore, the construction of $\Lift(g,h)$ depends smoothly on $g$ and $h$.
\end{proposition}
\begin{proof}
Covering $H$ by charts allows us to replace the global projection $\pi$ by local principal projections $\pi_i$ between Euclidean spaces. We apply Proposition \ref{prop:principalProjection} to each of these. Due to the uniqueness of the individual lifts given by each $\pi_i$, they all patch together to the claimed global lift.
\end{proof}

%% file: PaperI-Singularities.tex
%
%
%
%
%
%
\section{Singularities of integral submanifolds} \label{sec:singularities}

In this section we introduce the singularity models for integral submanifolds needed for our $h$-principles. The naming conventions that we follow are introduced in Subsection \ref{ssec:conventions}. In each subsequent subsection we address a concrete singularity/model, which we construct using the lifting ideas from Section \ref{sec:metasymplectic}.

\subsection{Conventions} \label{ssec:conventions}

In this section we introduce integral mapping germs defined along submanifolds. Often, the role of the base space will be played by $H \times \R$, with $H$ some manifold, and the maps in question will be fibered over $H$.

More concretely, our models will be fibered-over-$H$ integral mappings of the form
\[ H \times \R \longrightarrow J^r(H \times \R,\R). \]
To make the distinction between source and target clear (while also emphasising the fibered nature of the constructions), we will use $x = (\widetilde x, x_n)$ to denote a point in the base $H \times \R$ of the target, and $(\widetilde x,t)$ to denote a point in the source. When $H$ is a piece of Euclidean space, a point $\widetilde x$ will often be expanded in coordinates as $(x_1,\dots,x_{n-1})$. We write $y$ for the fibre coordinate in $J^0(H \times \R,\R)$. As in standard coordinates, we write $z$ to denote the rest of the fibre coordinates in $J^r(H \times \R,\R)$.

Due to the fibered nature of the constructions, we will work with the principal projection associated to the $x_n$-coordinate:
\[ (x,z^{(0,\dots,r)}) \in H \times \R \times \Sym^r(\R,\R). \]
All of our models of integral mapping will be lifts of maps into this principal projection (as in Proposition \ref{prop:principalProjectionGeneral}). We write $\SV_\can$ for the subbundle of vectors tangent to the $\Sym^r(\R,\R)$ factor.   

\begin{notation}
Our naming conventions for singularities will reflect the behaviour of the integral maps themselves, not their front projections. The chosen names mostly refer to the singularities of tangency with respect to the vertical distribution. These are equivalent, according to Proposition \ref{prop:principalProjectionGeneral}, to the singularities of tangency of the principal projection with $\SV_\can$. When singularities of mapping are present, we point it out explicitly. \hfill$\triangle$
\end{notation}

\subsubsection{Stabilising singularities} \label{sssec:stabilisingSing}

Once an integral mapping/singularity
\[ g: H \times \R \longrightarrow J^r(H \times \R,\R) \]
is given, we produce stabilised versions of $g$ using the following standard recipe. Note that we stabilise both in domain and target.

Given a manifold $V$ and a vector bundle $Z \to V \times H \times \R$, we consider the vector bundle $\underline{\R} \oplus Z \to V \times H \times \R$, i.e. the stabilisation of $Z$ using the trivial $\R$ bundle. We then introduce the integral section
\begin{align*}
G: V \times H \times \R \quad\longrightarrow\quad & J^r(\underline{\R} \oplus Z) \\
(v,\wtd{x},t) \quad\mapsto\quad & (v,\wtd{x},g(\wtd{x},t),0)
\end{align*}
which has the following properties:
\begin{itemize}
\item It is fibered over $V \times H$.
\item Its components mapping into $J^r(\underline{\R})$ are given by $g$ and therefore do not depend on the coordinate $v \in V$.
\item It maps into $J^r(Z)$ as the holonomic lift of the zero section in $Z$.
\end{itemize}

\begin{notation}
$G$ is said to be the \textbf{$(V,Z)$-stabilisation} of the map $g$. When the precise nature of $Z$ is not important for the argument, we simply talk about $V$-stabilisation. 

We will always treat a singularity and its stabilisations in a unified manner. That is: If a singularity is defined as the germ of $g$ along the subset $A \subset H \times \R$, then exactly the same name will be used for the germ of $G$ along $B \times H$, where $B \subset V$ is any subset. \hfill$\triangle$
\end{notation}

\subsection{The fold} \label{ssec:fold}

Take $H$ to be a single point. Then, the map
\begin{align*}
f: \R \quad\longrightarrow\quad & \R \oplus \Sym^r(\R,\R) \\ 
(t)   \quad\mapsto\quad & \left(x(t) = \dfrac{t^2}{2}, z^{(r)}(t) = t\right)
\end{align*}
has a fold singularity of tangency with respect to $\SV_\can$ at the origin: 
\[ \Sigma(f,\SV_\can) = \Sigma^{1,0}(f,\SV_\can) = \{0\}. \]

We can now apply the lifting Proposition \ref{prop:principalProjection} in order to integrate $f$ to the integral mapping:
\begin{align*}
\fold(t) := \Lift(f,0)(t)  \quad=\quad \left(x(t) = \dfrac{t^2}{2}, \right.
					y(t) 			 \quad=\quad & \dfrac{t^{2r+1}}{(2r+1)(2r-1)\cdots 1}, \dots \\
  				z^{(i)}(t) \quad=\quad & \dfrac{t^{2r-2i+1}}{(2r-2i+1)(2r-2i-1)\cdots 1}, \dots \\
					z^{(r)}(t) \quad=\quad & \left. t \right) .
\end{align*}
\begin{definition} \label{def:fold}
The \textbf{fold} is the germ at the origin of the integral map $\fold$, see Figure \ref{fig:Singularities_Fold}.
\end{definition}
By construction, $\fold$ is a integral embedding with a singularity of tangency with respect to $V_\can \subset \xi_\can$ at the origin $\Sigma(f,\SV_\can)$. This singularity becomes a singularity of mapping for its front projection; namely, an $A_{2r}$-cusp. Concretely, if $r=1$, the front is the usual semi-cubic cusp. If $r=2$, the front is given by the semi-quintic cusp, etc.

A point $p \in M$ is said to belong to the \textbf{fold locus} of the integral mapping
\[ g: M \to (J^r(Y),\xi_\can) \]
if $g$ is point equivalent along $p$ to (a stabilisation of) $\fold$.

\begin{figure}[ht]
\centering
\includegraphics[width = \linewidth ]{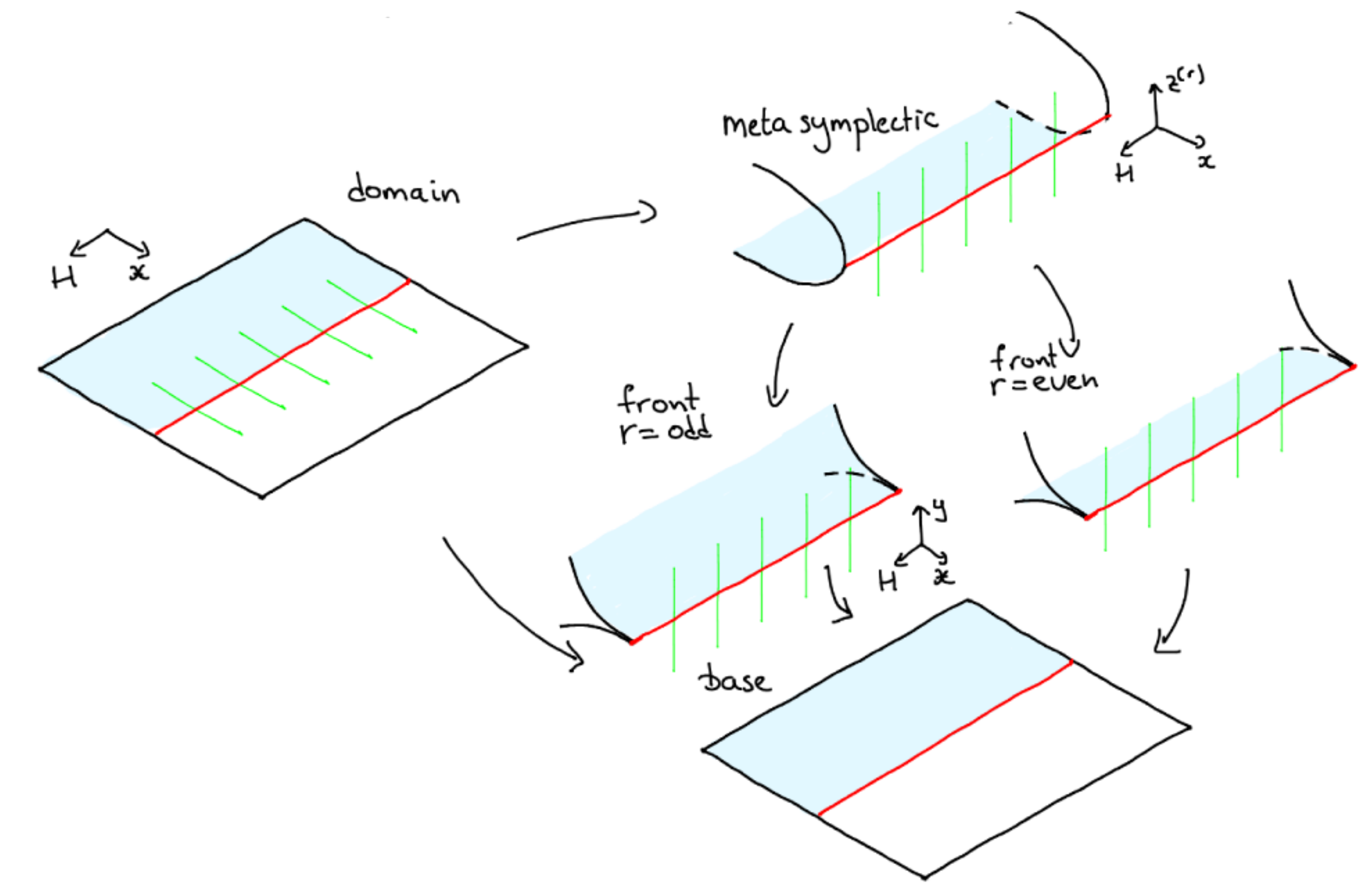}
\caption{Illustration of a fold. The orientation of the front projection changes depending on the parity of $r$. The ``sharpness'' of the cusp seen in the front projection also depends on $r$. The green lines depict the kernel of $d(\pi_b \circ \fold)$ (in the domain), its image by $d\fold$ (in the metasymplectic projection), and the associated vertical direction in $\Vert(Y)$ (in the front).}\label{fig:Singularities_Fold}
\end{figure}

\subsubsection{Line fields along a fold} \label{sssec:lineFields}

In our model fold, we encounter three line bundles: The vertical bundle $\Vert$ (trivialised by $\partial_y$), the line bundle of pure derivatives $V_\can \subset \xi_\can$ (which is isomorphic, using the metasymplectic projection, to the vertical subbundle $\SV_\can$ and thus trivialised by $\partial_{z^{(0,\dots,r)}}$), and lastly $\ker(f)$ (trivialised by $\partial_t$). We claim that there are canonical isomorphisms between these bundles that do not rely on standard coordinates.

Suppose $f: M \to (J^r(Y),\xi_\can)$ is an integral map with fold locus $\Sigma$ passing through $p$. This implies that $d_pf(TM)$ intersects $V_\can$ in a line. Globally, this yields a line subbundle 
\begin{equation}\label{eq:FoldDirectionBundle}
\SL_\Sigma \subset (f|_\Sigma)^*V_\can.
\end{equation}
Furthermore, the base map $\pi_b \circ f$ of $f$ is singular along $\Sigma$ and defines a kernel line field 
\[ \ker_\Sigma \subset TM|_\Sigma. \]
The differential $df|_\Sigma$ defines an isomorphism $\ker_\Sigma \to \SL_\Sigma$.

According to our model, the line bundle $\SL_\Sigma$ defines a principal direction at each point. As such, there are:
\begin{itemize}
\item A codirection field
\[ \lambda_\Sigma: \Sigma \to (\pi_b \circ f)^*T^*X, \]
defined up to $\R^*$-scaling, and spanning the annihilator of $d(\pi_b \circ f)(T\Sigma)$,
\item a line field 
\[ \SJ_\Sigma \subset (\pi^r_0 \circ f)|_\Sigma^*\Vert(Y), \]
\item and a canonical isomorphism
\[ \SL_\Sigma \cong \langle \lambda_\Sigma \rangle^{\otimes r} \otimes \SJ_\Sigma. \]
\end{itemize}
It follows that, for $r$ even, since $\langle \lambda_\Sigma \rangle^{\otimes r}$ trivialises, there is a canonical isomorphism $\SL_\Sigma \cong \SJ_\Sigma$ up to homotopy. For $r$ odd, we may choose $\lambda$ to be outward pointing with respect to $\pi_b \circ f(M)$. This yields an isomorphism between $\SJ_\Sigma$ and $\SL_\Sigma$ that is unique up to $\R^+$-scaling in each fibre and thus unique up to homotopy as well. See Figure \ref{fig:Singularities_Fold}.

\subsection{The double fold} \label{ssec:doubleFold}

Consider the principal mapping:
\begin{align*}
f: \R \quad\longrightarrow\quad & \R \times \Sym^r(\R,\R) \\
t          \quad\mapsto\quad & (x(t) = t^3/3-t; z^{(r)}(t) = t).
\end{align*}
Its tangency singularity locus consists of two fold points:
\[ \{t = \pm 1\} = \Sigma(f,\SV_\can) = \Sigma^{1,0}(f,\SV_\can). \]
We apply Proposition \ref{prop:principalProjectionGeneral} to produce an integral lift:
\begin{definition} \label{def:doubleFold}
The \textbf{double fold} is the germ of $\doubleFold := \Lift(f,0)$ along the interval $[-1,1]$.
\end{definition}
A double fold is thus not just two consecutive folds, but rather two folds that sit in a formally cancelling position with respect to each other, as witnessed by the model around the whole interval $[-1,1]$. We explain this in detail in Subsection \ref{sssec:regularisation} below.

The front of a double fold is the map:
\begin{equation*}
(t) \to \left(t^3/3-t; \int_0^t \int_0^{s_1} \dots \int_0^{s_{r-1}} s_r \prod_j(s_j^2-1)  ds_r\dots ds_1 \right)
\end{equation*}
whose mapping singularity locus consists of two $A_{2r}$-cusps. See Figure \ref{fig:Singularities_DoubleFold}.

Once we stabilise, the situation is as follows: Let $H$ be a closed manifold. Let $A \subset M$ be diffeomorphic to $H \times [-1,1]$. An integral mapping $g: M \to (J^r(Y),\xi_\can)$ has a double fold along $A$ if it is point equivalent along $A$ to (a $H$-stabilisation of) $\doubleFold$ along $H \times [-1,1]$. The annulus $A$ is called the \textbf{membrane}.

\begin{figure}[ht]
\centering
\includegraphics[width = \linewidth ]{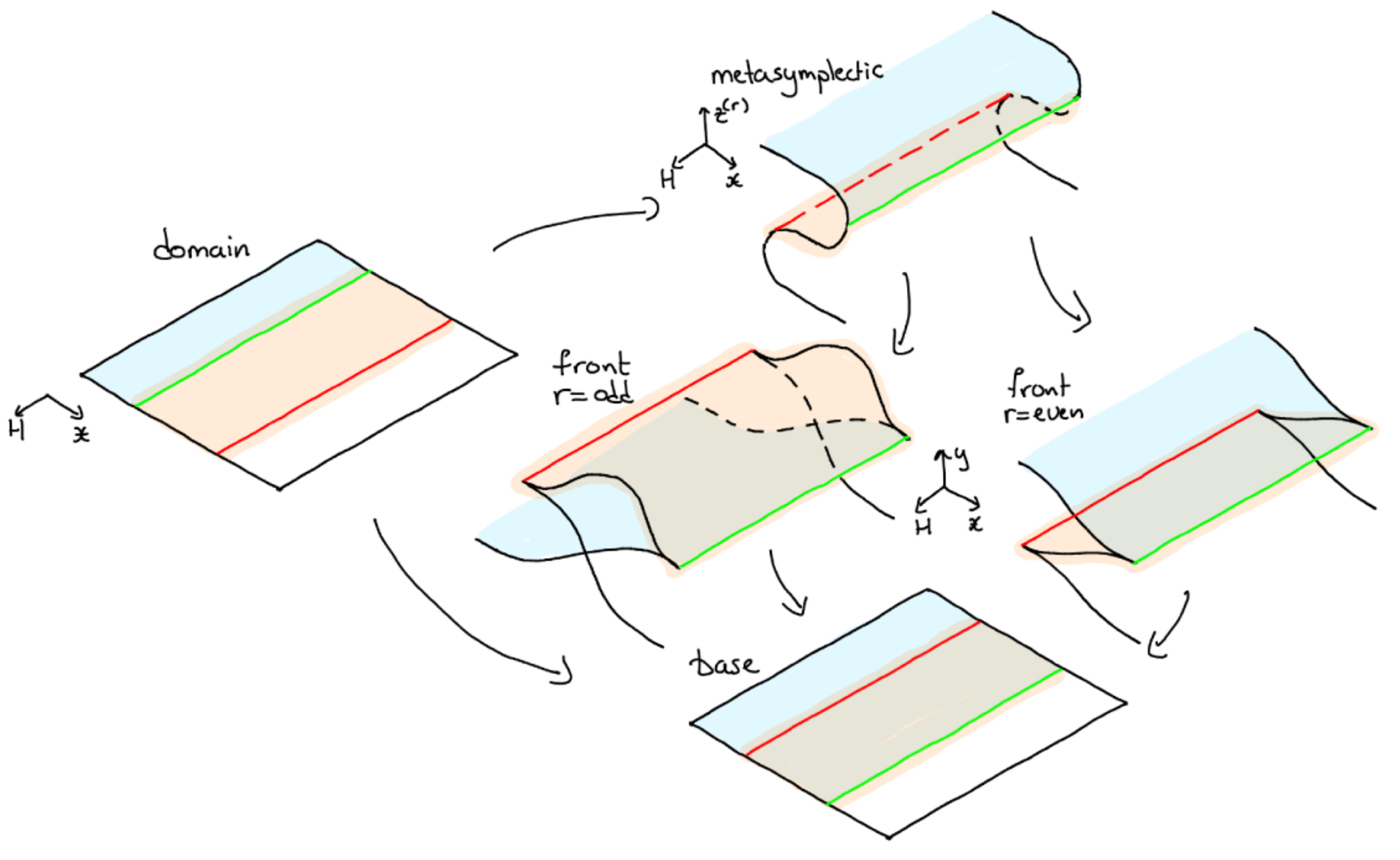}
\caption{A ($[0,1]$-stabilization of a) double fold in various projections. Depending on the parity of $r$, the singularities in the front projection are in an open/closed configuration; see also Figure \ref{fig:Singularities_Switching}.  The membrane is shown in orange and the fold locus in green and red. Note that the double fold is the germ along the membrane; it follows that the self-intersections seen in the front for $r$ odd are not part of the model.
	}\label{fig:Singularities_DoubleFold}
\end{figure}


\begin{figure}[ht]
\centering
\includegraphics[width = \linewidth ]{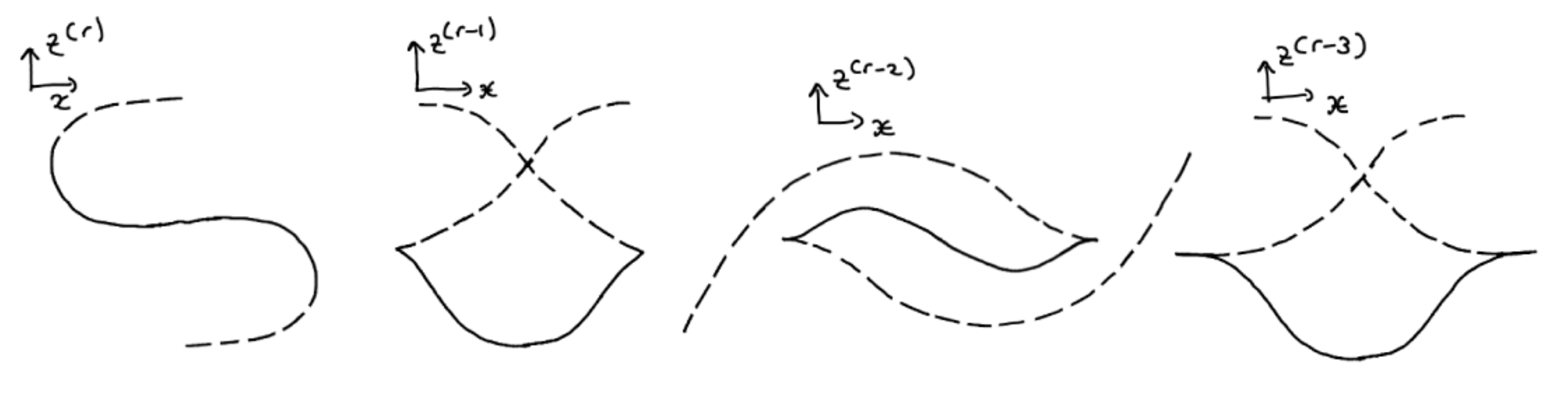}
\caption{On the left, the metasymplectic projection of a double fold. Subsequent pictures depict lower derivatives and the cusps become sharper. The membranes are drawn in solid black. Integrating repeatedly, the two singularities swap from an open to a closed configuration and viceversa; we call this the \emph{switching phenomenon}.}\label{fig:Singularities_Switching}
\end{figure}

\subsubsection{Regularisation of double folds} \label{sssec:regularisation}

Suppose $g: M \to (J^r(Y),\xi_\can)$ is an integral mapping with a double fold along $A \cong H \times [-1,1]$. Denote $H_\pm = H \times \{\pm 1\}$. According to Subsection \ref{sssec:lineFields}, there are line fields $\SL_\pm \subset (g|_{\partial A})^*V_\can$ along $H_\pm$ and an isomorphism:
\[  \Psi_\pm: \ker(\pi_b \circ g|_{\partial A}) \longrightarrow \SL_\pm, \]
In the $1$-dimensional model this isomorphism reads:
\[  \Psi_\pm: \ker(\pi_b \circ \doubleFold|_{\{\pm 1\}}) = \langle \partial_t \rangle \longrightarrow \SL_\pm = \langle \partial_{z^{(r)}} \rangle \]
and, up to homotopy, it takes $\partial_t$ to $\partial_{z^{(r)}}$. It therefore extends to an isomorphism
\[  \Psi_\pm: T[-1,1] = \langle \partial_t \rangle \longrightarrow \doubleFold^*V_\can = \langle \partial_{z^{(r)}} \rangle \]
defined over the whole membrane. Stabilisation defines an analogous isomorphism for any double fold.

Using this isomorphism, and still in the $1$-dimensional model, we can define a homotopy of bundle monomorphisms:
\[ (\rho_s)_{s \in [0,1]}: T[-1,1] \to \doubleFold^*\xi_\can \]
such that:
\begin{itemize}
\item $\rho_0 = d\doubleFold$,
\item $\rho_1(\partial_t)$ is graphical over $\partial_x$,
\item $\rho_s = d\doubleFold$ outside of $\Op([-1,1])$.
\end{itemize}
This homotopy is called the \textbf{regularisation} of $d\doubleFold$; it is depicted in Figure \ref{fig:Singularities_Regularization}. It tells us that the singularities of tangency of the double fold are homotopically inessential.

A stabilised double fold can be regularised using the stabilisation of the regularisation of the $1$-dimensional case. Concretely, let $g: M \to (J^r(Y),\xi_\can)$ be an integral map with a double fold along $A$. The regularisation $\rho_s: \Op(A) \to \Hom(TM,g^*\xi_\can)$ homotopes $dg =\rho_0$, relatively to the boundary of $\Op(A)$, to a monomorphism $\rho_1$ that is transverse to $V_\can$ over $\Op(A)$. Furthermore, the homotopy $\rho_s$ is, by construction, principal. I.e. for a fixed $p \in M$, all the planes $(\rho_s(T_pM))_{s \in [0,1]}$ intersect in a graphical $(n-1)$-plane.

\begin{figure}[ht]
\centering
\includegraphics[width = \linewidth ]{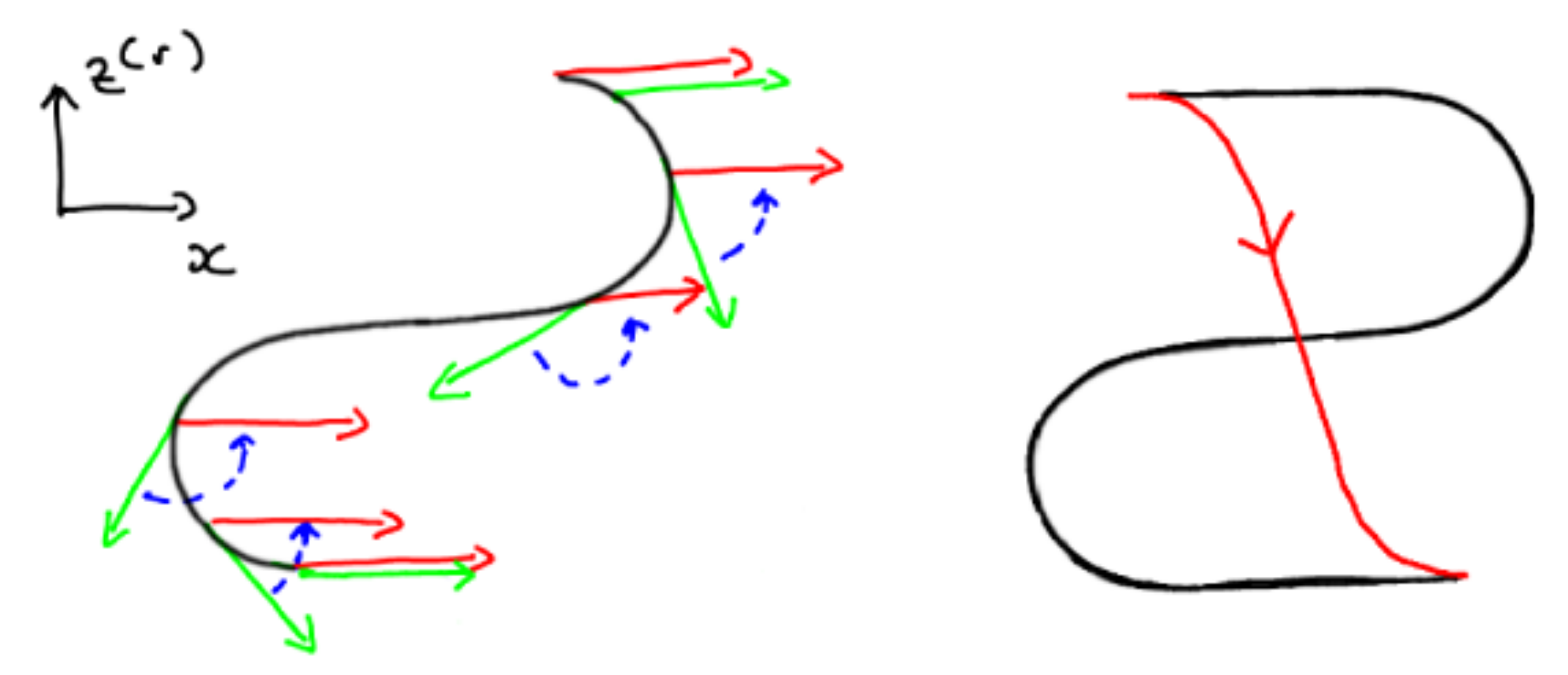}
\caption{On the left, the regularising homotopy of a double fold, as seen in the metasymplectic projection. In green, the image of the Gauss map. In red, the monomorphism corresponding to the regularisation at time $s=1$. The homotopy itself is depicted in blue. On the right, we see a double fold that is ``removable'': i.e. the model extends sufficiently beyond the membrane so that we are able to replace the double fold by a graphical map; see Subsection \ref{ssec:removable}. We think of the process depicted on the left as a formal version of the process depicted on the right.}\label{fig:Singularities_Regularization}
\end{figure}

\begin{remark} \label{rem:Maslov}
In the contact setting, the two folds that form a double fold have opposite Maslov coorientations. For higher jet spaces in which the fibres of $Y$ have dimension $1$, a similar statement can be given. This requires a description of the Grassmannian of \emph{integral elements} (i.e. those subspaces of $\xi_\can$ that may be tangent to an integral submanifold) and its Maslov cycle. For jet spaces in which the fibres of $Y$ are higher dimensional, a similar analysis has to be carried out, but the Maslov homotopical obstructions are of a different nature, due to high-dimensionality (i.e. they take place in higher homotopy groups). We postpone this discussion to the sequel \cite{PT2}. \hfill$\triangle$
\end{remark}


\begin{remark}
Fix an annulus $A \cong H \times [-1,1] \subset M$. Let $g: M \to (J^r(Y),\xi_\can)$ be an integral embedding with folds along $\partial A$. As explained above, $g|_A$ need not be a double fold. A concrete case is of interest is when additional singularities are present in the interior of $A$, as this precludes the possibility of extending the fold models along $\partial A$ to a double fold along $A$.

In Section \ref{sec:wrinkledEmbeddings} we will encounter the following situation: We are given an integral embedding $g$ with a double fold along $A$ which we then modify, within $A$, in order to introduce an additional double fold along $A' \subset A$. We will then speak of \emph{nested singularities}. Unfortunately, Definition \ref{def:doubleFold} is not well-suited for this purpose. Instead of generalising the notion of double fold in order to allow nesting (which is possible, but would require us to elaborate on the idea of regularisation and how certain homotopic data extends from $\partial A$ to $A$), we will rely instead on \emph{wrinkles} (Subsection \ref{ssec:wrinkles}), which are designed to allow for nesting. \hfill$\triangle$
\end{remark}

\subsection{The pleat} \label{ssec:pleat}

Let $H=\R$ and consider the following mapping into the principal projection:
\begin{align*}
f: H \times \R \quad\longrightarrow\quad & H \times \R \times \Sym^r(\R,\R) \\ 
(x_1,t)     \quad\mapsto\quad & \left(x_1, x_2(x_1,t) = t^3/3-x_1t, z^{(0,r)}(x_1,t) = t \right).
\end{align*}
It is the standard pleat (as a singularity of tangency with respect to the vertical), see Figure \ref{fig:Singularities_Pleat}. Its singularity locus reads:
\[ \Sigma^1(f,\SV_\can) = \{t^2 - x_1=0\}, \qquad \Sigma^{1,1}(f,\SV_\can) = \{t,x_1 = 0 \}. \]
\begin{definition} \label{def:pleat}
The \textbf{pleat} is the germ at the origin of the integral mapping $\pleat := \Lift(f,0)$.
\end{definition}
A point $p \in M$ belongs to the \textbf{pleat locus} of an integral mapping 
\[ g: M \to (J^r(Y),\xi_\can) \]
if $g$ is point equivalent along $p$ to (a stabilisation of) the pleat.

It is not particularly enlightening to provide an explicit formula for $\pleat$, but its front projection reads:
\begin{align*}
\pi_f \circ \pleat: H \times \R \quad\longrightarrow\quad & J^0(H \times \R,\R)  \\ 
(x_1,t)        \quad\mapsto\quad & \left(x_1, x_2 = t^3/3 - x_1t, y = \int_0^t \int_0^{s_1} \dots \int_0^{s_{r-1}} s_r \prod_j(s_j^2-x_1)  ds_r\dots ds_1 \right). \nonumber
\end{align*}

\begin{figure}[ht]
\centering
\includegraphics[width = \linewidth ]{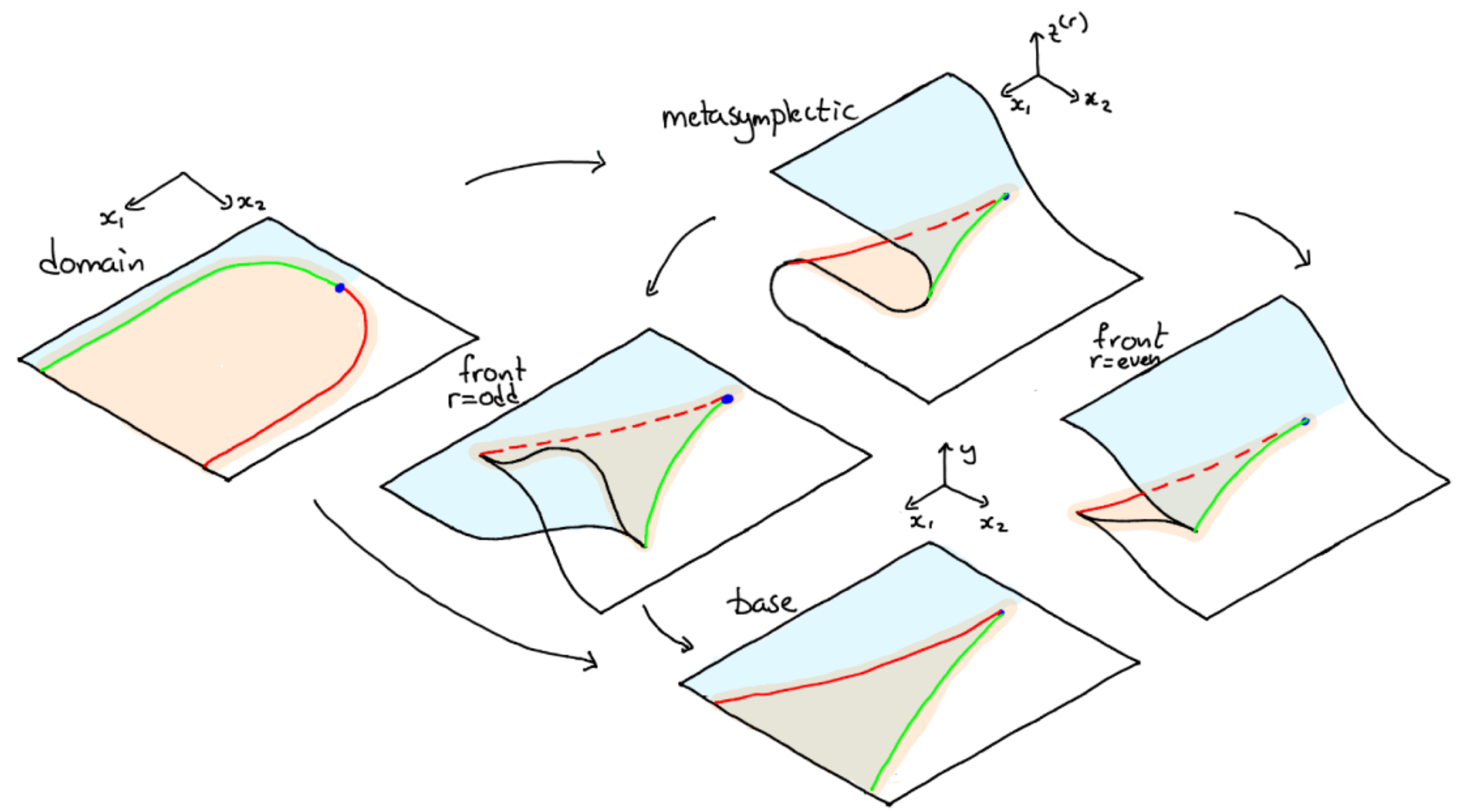}
\caption{Illustration of the pleat in various projections. If $r$ is even, the front is a topological embedding. if $r$ is odd, we see self-intersections of the front in any small neighbourhood of the pleat.}\label{fig:Singularities_Pleat}
\end{figure}

\subsection{The Reidemeister I move} \label{ssec:ReidemeisterI}

Consider the map $f$ introduced in Subsection \ref{ssec:pleat}, which allowed us to define the pleat. Regard its first factor $H = \R$ as a parameter space. Consider then the family of mappings:
\[ f|_{\{x_1\} \times \R}: \R \to J^r(\R,\R) \]
given by freezing the parameter $x_1 \in H$.

If $x_1 > 0$, the map has no singularities and is graphical over its domain $\R$. If $x_1 < 0$, the map is a double fold. At $x_1 = 0$, a birth/death phenomenon takes place:
\[ t \mapsto \left(\dfrac{t^3}{3}, t\right), \]
i.e. a cubic singularity of tangency with respect to the vertical. Its unfolding is the pleat itself. Following the standard terminology in the contact setting we define:
\begin{definition}
The family of integral embeddings 
\[ (\Reid_{x_1} \,:=\, \Lift(f|_{\{x_1\} \times \R},0))_{x_1 \in H} \]
is called the \textbf{first Reidemeister move}.

The lift $\Lift(f|_{\{0\} \times \R},0)$ is called the \textbf{cubic}.
\end{definition}
The front of the cubic reads:
\[ t \mapsto \left(\dfrac{t^3}{3}, \dfrac{t^{3r+1}}{(3r+1)(3r-2)\cdots 1}\right). \]

\subsection{The closed double fold} \label{ssec:closedDoubleFold}

The switching phenomenon seen in Figures \ref{fig:Singularities_DoubleFold} and \ref{fig:Singularities_Switching} says that the front projection of a double fold is in a closed configuration if $r$ is odd. For the purposes of flexibility, this is not very convenient, and we would prefer to have an open configuration. This is well-known in Contact Topology, where the open configuration corresponds to the stabilisation\footnote{Not to be confused with the notion of stabilising a singularity, as in Subsection \ref{sssec:stabilisingSing}.}. We therefore consider the following mapping: 
\begin{align*}
f: \R \quad\longrightarrow\quad & \R \times \Sym^r(\R,\R) \\
t          \quad\mapsto\quad & (x(t) = t^3/3-t; z^{(r)}(t) = t^2).
\end{align*}
It is an immersion with a self-intersection at $t=\pm \sqrt{3}$. Its singularity locus of tangency consists of the two fold points
\[ \{t = \pm 1\} = \Sigma(f,\SV_\can) = \Sigma^{1,0}(f,\SV_\can). \]
We apply Proposition \ref{prop:principalProjectionGeneral} to produce an integral lift:
\begin{definition} \label{def:closedDoubleFold}
The \textbf{closed double fold} is the germ of $\ClosedDoubleFold := \Lift(f,0)$ along the interval $[-1,1]$.
\end{definition}
Which is an embedded integral map. I.e. the double point in the principal projection disappears upon lifting. Recall that the vertical displacement (upon integrating once) of a lifted interval is the area it bounds. Here it is indeed the case that $f([-1,1])$ is a closed curve bounding non-zero area.

Let us emphasise that \emph{closed} here refers to the fact that the two folds are indeed in a closed configuration \emph{in the principal projection}. This is true for the front projection as well if and only if $r$ is even.

In general: Let $H$ be a closed manifold. Let $A \cong H \times [-1,1] \subset M$ be an annulus with base $H$. An integral mapping $g: M \to (J^r(Y),\xi_\can)$ has a closed double fold along $A$ if it is point equivalent along $A$ to (a stabilisation of) $\ClosedDoubleFold$ along $H \times [-1,1]$.

\begin{figure}[ht]
\centering
\includegraphics[width = \linewidth ]{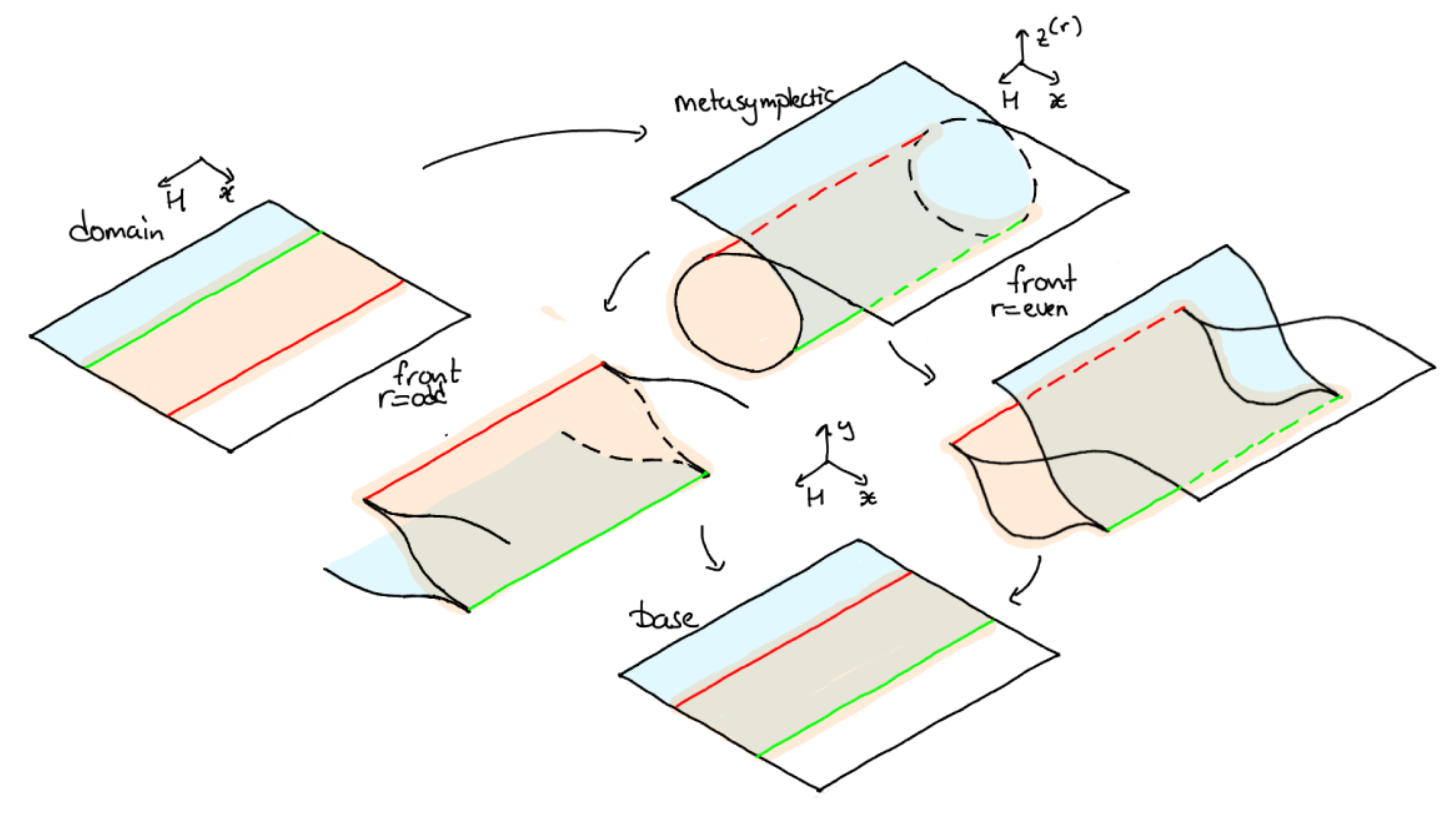}
\caption{The ($[0,1]$-stabilization of a) closed double fold. Depending on the parity of $r$, the front projection is in an open or a closed configuration.}\label{fig:Singularities_ClosedDoubleFold}
\end{figure}

\subsubsection{Regularisation} \label{sssec:regularisationClosed}

Suppose $g: M \to (J^r(Y),\xi_\can)$ is an integral mapping with a closed double fold along $A \cong H \times [-1,1]$. Denote $H_\pm = H \times \{\pm 1\}$. As for the double fold, there are line fields $\SL_\pm \subset (g|_{\partial A})^*V_\can$ along $H_\pm$ and an isomorphism:
\[  \Psi_\pm: \ker(\pi_b \circ g|_{\partial A}) \longrightarrow \SL_\pm. \]
In the $1$-dimensional model, the isomorphism reads
\[  \Psi_\pm: \ker(\pi_b \circ \ClosedDoubleFold|_{\{\pm 1\}}) = \langle \partial_t \rangle \longrightarrow \SL_\pm = \langle \partial_{z^{(r)}} \rangle \]
and, up to homotopy, it takes $\partial_t$ to $\pm \partial_{z^{(r)}}$ at $t=\pm 1$. It follows that the isomorphism does not extend to the membrane.

Nonetheless, we can define a bundle monomorphism
\[ \rho: T[-1,1] \to \ClosedDoubleFold^*\xi_\can \]
satisfying
\begin{itemize}
\item $\rho(\partial_t)$ is graphical over $\partial_x$,
\item $\rho = d\ClosedDoubleFold$ outside of $\Op([-1,1])$.
\end{itemize}
We call this the \textbf{regularisation} of $d\ClosedDoubleFold$; it is depicted in Figure \ref{fig:Singularities_ClosedRegularization}. It can be readily observed that $\rho$ is \emph{not} homotopic to $d\ClosedDoubleFold$.

\begin{figure}[ht]
\centering
\includegraphics[width = \linewidth ]{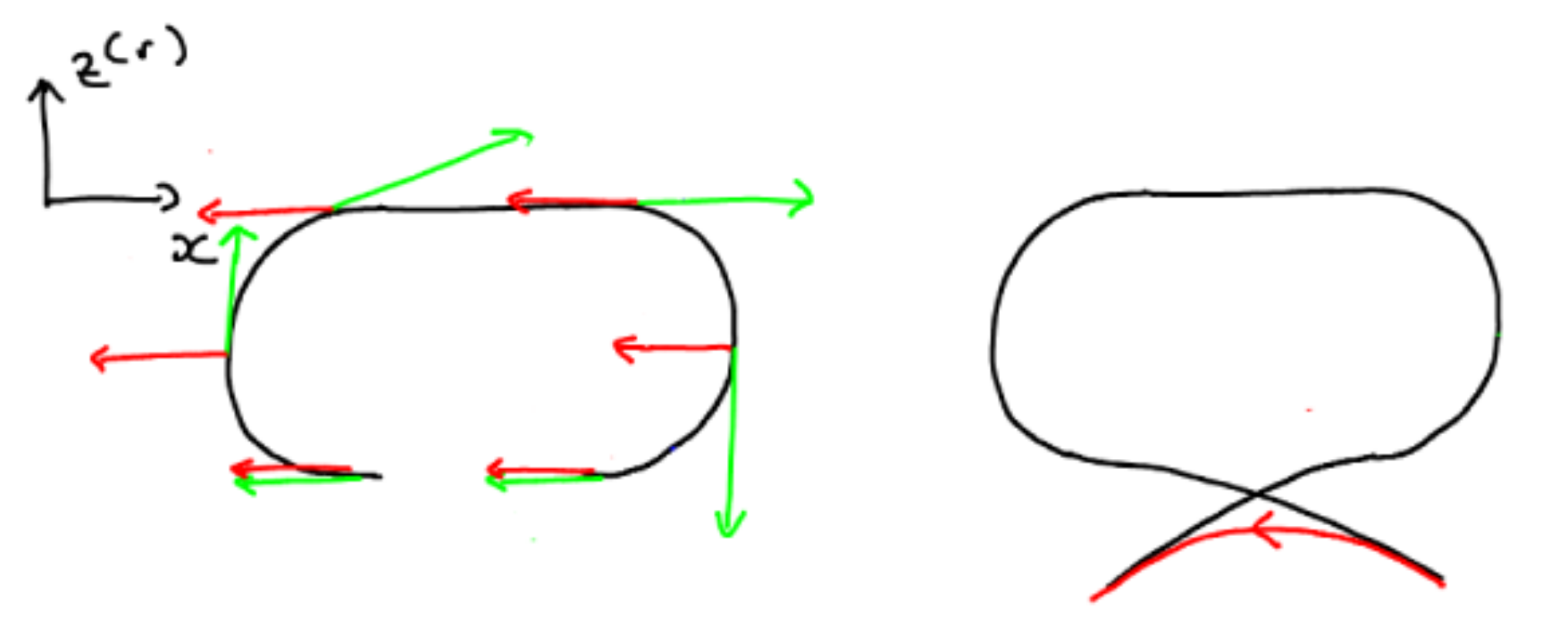}
\caption{On the left, the regularisation of the closed double fold. The Gauss map is shown in green and the regularisation in red. The two are not homotopic to one another, as the latter ``makes an additional turn''. On the right, a closed double fold that is ``removable'' is replaced by a graphical map (in red); see Subsection \ref{ssec:removable}. The Gauss map of this graphical map is precisely in the homotopy class of the regularisation.}\label{fig:Singularities_ClosedRegularization}
\end{figure}

\subsection{The stabilisation} \label{ssec:stabilisation}

We now introduce the analogues of the pleat and the Reidemeister I move for the closed double fold.

\subsubsection{The closed pleat}

Let $H=\R$ and consider the following mapping into the principal projection:
\begin{align*}
f: H \times \R \quad\longrightarrow\quad & H \times \R \times \Sym^r(\R,\R) \\ 
(x_1,t)     \quad\mapsto\quad & \left(x_1, x_2(x_1,t) = t^3/3-x_1t, z^{(0,r)}(x_1,t) = t^2 \right).
\end{align*}
This is the usual unfolding of the planar cusp. Unlike all our previous singularities, we have a singularity of mapping at $(0,0)$, the preimage of the cusp point. The locus of singularities of tangency reads:
\[ \Sigma^1(f,\SV_\can) = \{t^2 - x_1=0\}. \]
We note that the tangent space at the cusp point is well-defined and it is vertical. We define:
\begin{definition} \label{def:closedPleat}
The \textbf{closed pleat} is the germ at the origin of the integral mapping $\ClosedPleat := \Lift(f,0)$.
\end{definition}
A point $p \in M$ belongs to the \textbf{closed pleat locus} of an integral mapping 
\[ g: M \to (J^r(Y),\xi_\can) \]
if $g$ is point equivalent at $p$ to (a stabilisation of) the closed pleat.

\begin{figure}[ht]
\centering
\includegraphics[width = \linewidth ]{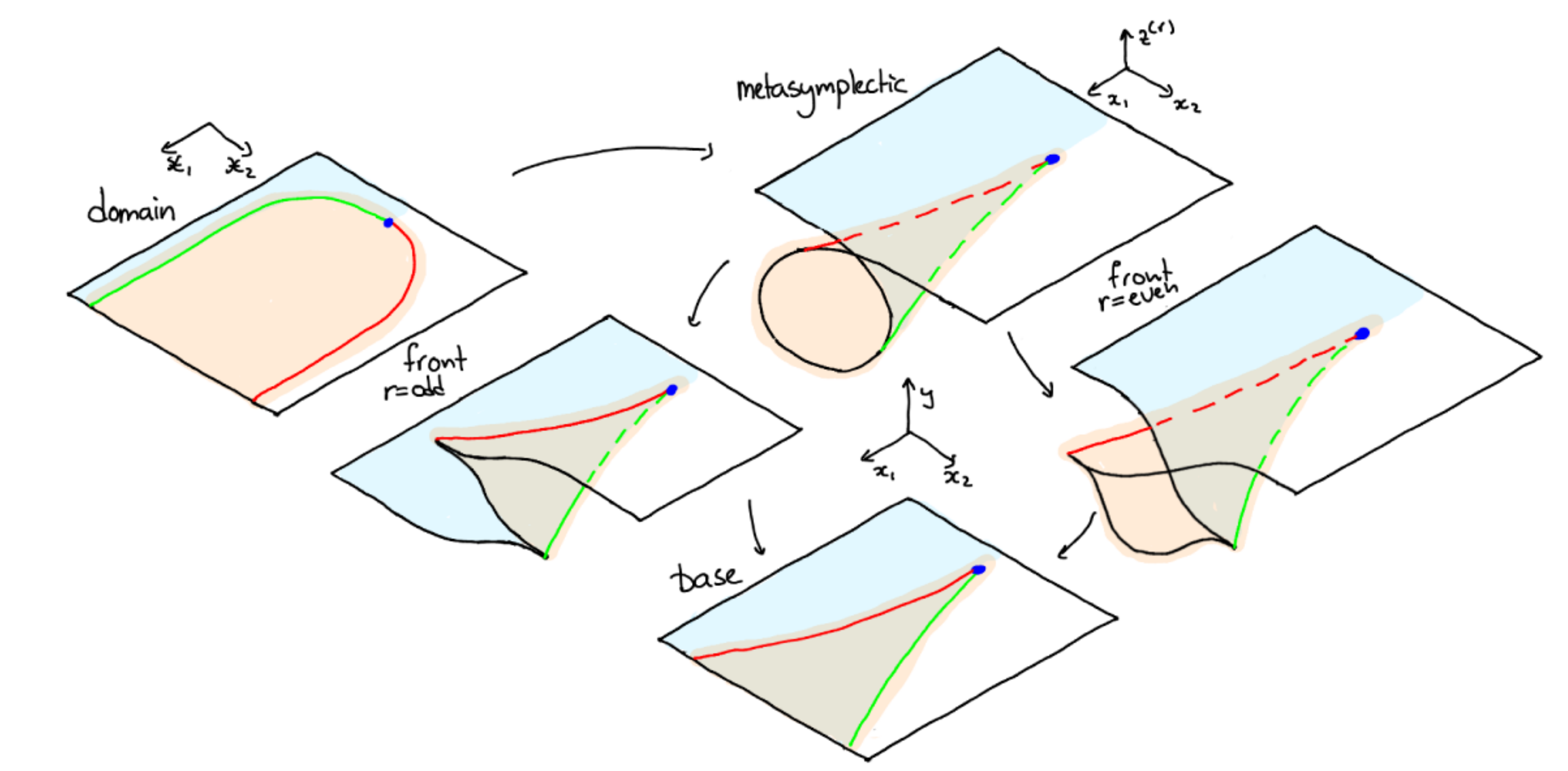}
\caption{The closed pleat. When $r$ is odd, the front is a topological embedding.}\label{fig:Singularities_ClosedPleat}
\end{figure}

\subsubsection{The stabilisation} \label{sssec:stabilisation}

We now study the family of mappings
\[ f|_{\{x_1\} \times \R}: \R \to J^r(\R,\R) \]
given by freezing $x_1 \in H$, seen as a parameter. The maps have no singularities of tangency if $x_1 > 0$, they have closed double folds if $x_1 < 0$, and it is given by 
\[ t \mapsto \left(\dfrac{t^3}{3}, t^2\right), \]
when $x_1 = 0$. Following standard notation in Contact Topology, we define:
\begin{definition} \label{def:stabilisation}
The \textbf{stabilisation} is the family of integral maps 
\[ (\Stab_{x_1} \,:=\, \Lift(f|_{\{x_1\} \times \R},0))_{x_1 \in H}. \]

The \textbf{vertical cusp} or \textbf{closed double fold embryo} is the integral map $\Lift(f_{\{0\} \times \R},0)$.
\end{definition}
Much like the closed pleat, the vertical cusp is an integral map, but it is singular. Its front reads:
\[ (t) \mapsto \left(\dfrac{t^3}{3}, \dfrac{t^{3r+2}}{(3r+2)(3r-1)\cdots 1}\right). \]

\begin{remark} \label{rem:stabilisation}
In the Contact Topology setting ($r=1$ and the fibres of $Y$ are $1$-dimensional) this is well-known: even though the stabilisation is not a homotopy through embedded legendrians, it is a homotopy through singular legendrians (i.e. topologically embedded maps tangent to the contact structure). \hfill$\triangle$
\end{remark}

\subsection{Wrinkles} \label{ssec:wrinkles}

Pleats and Reidemeister I moves allow us to introduce double folds in domain and parameter directions, respectively. We now go further and consider self-cancelling families of maps containing double folds and pleat-like singularities that disappear in Reidemeister-like events. In Subsection \ref{ssec:closedWrinkles} we will consider the analogous notion for the closed double fold; these will go under the name of \emph{closed wrinkles}.

\begin{remark} \label{rem:comparisonWrinklesI}
We warn the reader that our definition of a wrinkle, even though inspired by the wrinkled embeddings introduced by Eliashberg and Mishachev \cite{ElMiWrinEmb}, differs from theirs in various key aspects.

A first (rather cosmetic) difference is that we allow the base of the wrinkle to be an arbitrary submanifold with boundary, as opposed to a ball. One of the mottos of this paper is that it is sometimes more convenient to work with other bases, as it simplifies some arguments (e.g. our main theorem). Further, it is always possible to carry out a surgery to replace a wrinkle with wrinkles whose base is a ball; see Appendix \ref{sec:surgeries}. Similar ideas are present already in the wrinkling saga.

A more fundamental difference is that our wrinkles are smooth maps. The name ``wrinkle'' is meant to reflect the nature of the singularities of tangency instead. In particular, upon projecting to the base, our model wrinkle matches the (equidimensional) wrinkled submersions from \cite{ElMiWrinI}.

The front projections of our wrinkles do have wrinkle-type singularities of mapping. These are thus closer to classic wrinkled embeddings, which may be regarded as front projections of integral mappings into $J^1(Y)$. However, as explained by Murphy \cite{Mur} (in the contact case), the lift of a wrinkled embedding is a singular integral map (in the contact case, a singular legendrian) in which the singularity corresponds to the birth of a pair of folds with the same Maslov coorientation (i.e. a closed pleat, as introduced in Definition \ref{def:closedPleat}). These objects are thus not wrinkles (according to our definition), but \emph{closed wrinkles} (to appear in Subsection \ref{ssec:closedWrinkles}). \hfill$\triangle$
\end{remark}

The setup is as follows. We fix compact manifolds $K$ and $H$, the former serving as parameter space. We also fix $D \subset K \times H$, a submanifold with boundary and non-empty interior. We write $H_k := H \cap (\{k\} \times H)$ and $D_k := D \cap H_k$. Lastly, we pick a function $\rho: K \times H \to [-1,1]$ that is strictly negative in the interior of $D$, strictly positive outside, and has $\partial D$ as a regular level set. With this data we define a map
\begin{align*}
f: K \times H \times \R \quad\longrightarrow\quad & K \times H \times \R \times \Sym^r(\R,\R) \\ 
(k,\widetilde x,t)      \quad\mapsto\quad & \left(k,\widetilde x,\dfrac{t^3}{3}+\rho(k,\widetilde x)t,t \right)
\end{align*}
which we think of as a family parametrised by $K$. This family can then be lifted, parametrically in $K$, with respect to the $x$-coordinates:
\[ \Lift(f,0): K \times H \times \R \quad\longrightarrow\quad K \times J^r(H \times \R,\R). \]

\begin{definition} \label{def:wrinkles}
The \textbf{model wrinkle} $\Wrin$ with base $D$ and height $\rho$ is the germ of $\Lift(f,0)$ along $A = \{t^2 + \rho(k,\widetilde x) \leq 0\}$.
\end{definition}
The region $A$ is said to be the \textbf{membrane} of the wrinkle. Its boundary $\partial A = \{t^2 + \rho(k,\widetilde x) = 0\}$ is the tangency locus $\Sigma(\Wrin,V_\can)$;  it is diffeomorphic to the double of $D$. The singularity locus $\partial A$ is tangent to the $t$-coordinate along the submanifold $(\partial D) \times \{0\}$; we call this the \textbf{equator} of the wrinkle. Even though $\Wrin$ depends on $D$ and $\rho$, we leave them implicit to keep notation light. If we want to emphasise them, we write $\Wrin^{D,\rho}$.

\begin{remark}
The definition given is coherent with our naming conventions. Namely, a stabilisation (with respect to the domain) of a model wrinkle is still a model wrinkle. \hfill$\triangle$
\end{remark}

We write $\Wrin_k$ for the restriction of $\Wrin$ to $\Op(A_k)$, where $A_k := A \cap (S_k \times \R)$. For all $k \in K$, $\Wrin_k$ has a double fold over the interior of $D_k$ (but do note that $D_k$ may be empty). Furthermore, whenever $\partial D_k$ is a regular level set of $\rho|_{D_k}$, we have that $\partial A_k$ is diffeomorphic to the double of $D_k$; in this case the equator $\partial D_k \times \{0\} \subset A_k$ is the pleat locus. If $\partial D_k$ is not a regular level set, the singularities of $\Wrin_k$ along the equator may be complicated; still, they are as stable as the tangencies of $D$ itself with the fibres $S_k$.

\begin{figure}[ht]
\centering
\includegraphics[width = \linewidth]{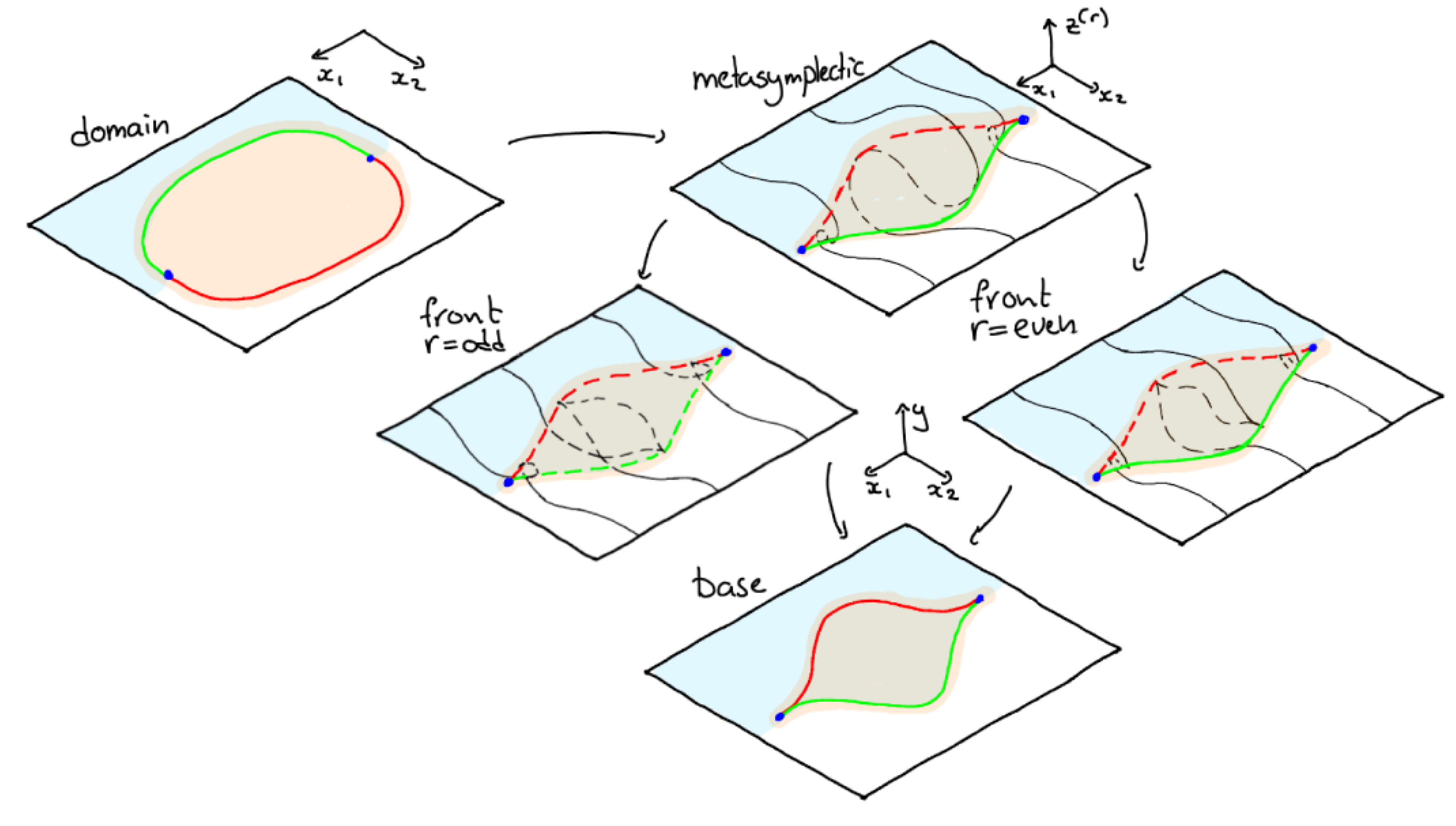}
\caption{Illustration of a wrinkle with ball base. The two fold loci (the red and green curves) change position when mapped to the base of the target jet space. They meet at the equator, shown as the two blue dots.}\label{fig:Singularities_Wrinkle}
\end{figure}

In order to allow nesting, we introduce:
\begin{definition}
Fix a model wrinkle $\Wrin$ with base $D$, membrane $A$, and height $\rho$. Let $\Sigma(\Wrin)$ be the singularity locus.

Fix a subset $A' \subset K \times M$. A fibered-over-$K$ family of integral embeddings
\[ g = (g_k)_{k \in K}: M \to (J^r(Y),\xi_\can) \]
has a \textbf{wrinkle} of base $D$ and height $\rho$ along $A'$ if $g$ and $\Wrin$ are point equivalent along $\partial A'$ and $\partial A$. It is a model wrinkle if this equivalence can be extended to the interior.
\end{definition}
A wrinkle has a well-defined membrane $A'$ and a well-defined equator $E \subset \partial A'$. Close to $E$, the wrinkle has double folds. Do note that the two components of $\partial A' \setminus E$ may not form a double fold due to the presence of additional singularities within $A'$.

In practice, we always require that $D$ has non-empty boundary. This implies that the equator $E$ is non-empty and therefore that there is a region in which we have double folds.

\subsubsection{Some relevant examples}

Definition \ref{ssec:wrinkles} is very general. Some concrete choices of $K$, $H$ and $D$ will keep appearing in our constructions. For instance:
\begin{itemize}
\item A double fold is a model wrinkle in which $D = K \times H = \{\text{point}\}$.
\item A pleat is a model wrinkle for which $K = \{\text{point}\}$, $H = \R$, and $D =  \{\text{point}\} \times \R^{\geq 0}$.
\item A Reidemeister I move is a model wrinkle with $K = \R$, $H = \{\text{point}\}$, and $D = \R^{\geq 0} \times  \{\text{point}\}$.
\end{itemize}

Furthermore, we denote:
\begin{definition}
A \textbf{wrinkle with cylinder base} is a wrinkle with $H=\NS^{n-1}$, $K=\D^k$, and $D = K \times H$.
\end{definition}
One could generalise this notion slightly by letting $H$ be a closed manifold and $D$ be $K' \times H$, where $K' \subset K$ is a submanifold with boundary and non-empty interior. Such a wrinkle is, up to point symmetries, the $H$-stabilisation of a wrinkle with base $K' \times \{\text{point}\}$. Along the equator, the singularities are a $\partial K'$-family of ($H$-stabilisations of) the cubic.

Another interesting case is:
\begin{definition}
A \textbf{wrinkle with ball base} is a wrinkle with $K$ and $H$ Euclidean spaces and $D$ the unit ball.
\end{definition}
If $\Wrin$ is a wrinkle with ball base, each $\Wrin_k$, $k \in K$, is either a wrinkle with ball base, a smooth map, or has a single birth/death singularity located at the origin. This birth/death germ at the origin is called the embryo of wrinkle with ball base.

\subsubsection{Regularisation} \label{sssec:regularisationWrin}

Let $g: M \to (J^r(Y),\xi_\can)$ be an integral mapping with a model wrinkle with membrane $A$. Along $\partial A$, the differential $dg$ provides an isomorphism between $\ker(\pi_b \circ g|_{\partial A})$ and the image line field $\SL \subset g^*V_\can$. In terms of the standard coordinates of the model, the isomorphism between the two extends to the membrane as $\Psi(\partial_t) = \partial_{z^{(r)}}$.

There is then a homotopy $(\rho_s)_{s \in [0,1]}$ of bundle monomorphisms, supported in $\Op(A)$, that connects $dg$ with a monomorphism $\rho_1$ with image transverse to $V_\can$. This homotopy modifies only the component $dg(\partial_t)$, which rotates using $\Psi$. We call this the \textbf{formal regularisation} of the model wrinkle. See Figure \ref{fig:Singularities_WrinkleRegularization}.

\begin{figure}[ht]
\centering
\includegraphics[width = \linewidth ]{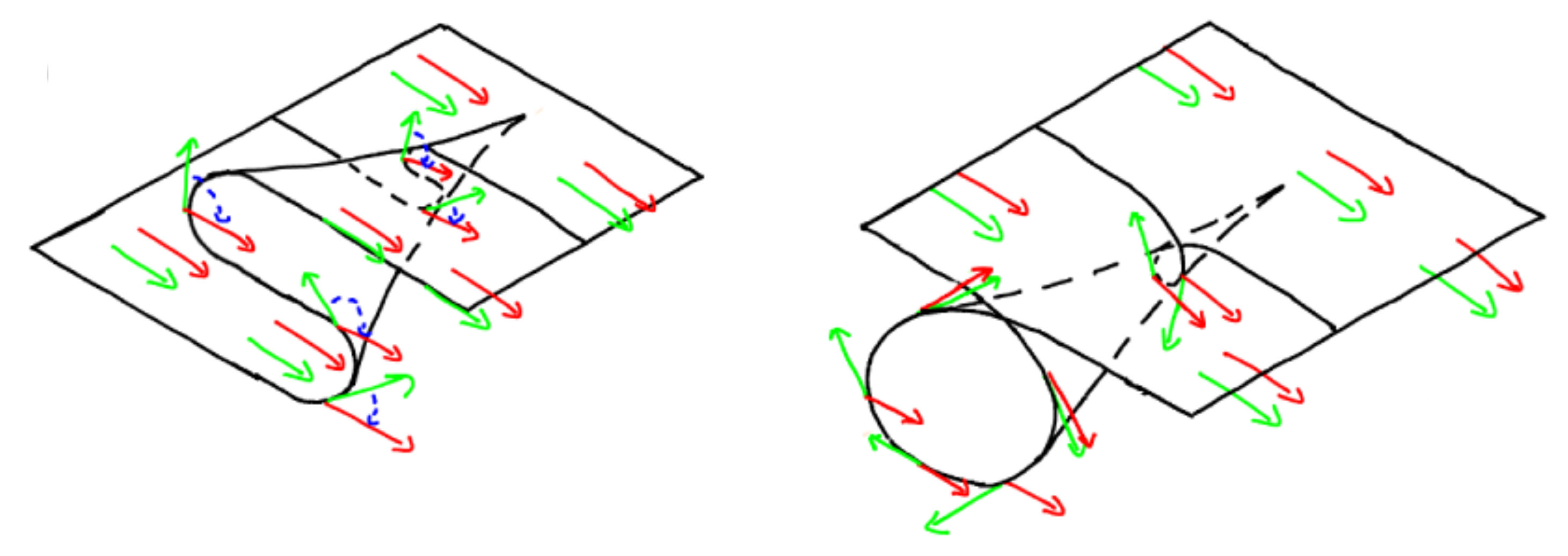}
\caption{Regularisation of (one half of) the model wrinkle (on the left) and (one half of) the model closed wrinkle (on the right); as shown in the metasymplectic projection. The Gauss maps and the non-singular monomorphisms are shown in green and red, respectively. For the model wrinkle, the two are homotopic, as shown in blue. This is not the case for the closed wrinkle, as they differ by one turn.}\label{fig:Singularities_WrinkleRegularization}
\end{figure}

\subsection{Closed wrinkles} \label{ssec:closedWrinkles}

We continue using the setup and notation from Subsection \ref{ssec:wrinkles}. Consider the lift $\Lift(f,0)$ of the $K$-family
\begin{align*}
f: K \times H \times \R \quad\longrightarrow\quad & K \times H \times \R \times \Sym^r(\R,\R) \\ 
(k,\widetilde x,t)      \quad\mapsto\quad & \left(k,\widetilde x,\dfrac{t^3}{3}+\rho(k,\widetilde x)t,t^2 \right)
\end{align*}
with respect to the $x$-coordinates.
\begin{definition}
The \textbf{model closed wrinkle} $\ClosedWrin$ with base $D$ and height $\rho$ is the germ of $\Lift(f,0)$ along the membrane $A = \{t^2 + \rho(k,\widetilde x) \leq 0\}$.

A fibered-over-$K$ family of integral embeddings
\[ g = (g_k)_{k \in K}: M \to (J^r(Y),\xi_\can) \]
has a \textbf{closed wrinkle} along $A'$ if $g$ and $\ClosedWrin$ are point equivalentt along $\partial A'$ and $\partial A$. It is a model closed wrinkle if this equivalence can be extended to the interior.
\end{definition}
A closed wrinkle has a well-defined membrane and equator. Away from the equator, it has closed double folds. Furthermore, whenever $\partial D_k$ is cut transversely, the restriction $\ClosedWrin_k$ has closed pleats along the corresponding equator.

As for wrinkles, we single out some cases of interest:
\begin{itemize}
\item A closed double fold, a closed pleat, and a stabilisation are all model closed wrinkles.
\item A closed wrinkle has cylinder base if $H=\NS^{n-1}$, $K=\D^k$ and $D = K \times H$. Then, the germ of $\ClosedWrin_k$ along the equator is an $H$-stabilisation of the vertical cusp.
\item  A closed wrinkle has ball base if $D$ is the unit ball in the Euclidean space $K \times H$. Then $\ClosedWrin_k$ is smooth if $|k| > 1$, a closed wrinkle with base a ball if $k < 1$, and a birth/death event located at the origin (called the embryo of closed wrinkle with ball base) if $|k|=1$.
\end{itemize}

\begin{remark} \label{rem:comparisonWrinklesII}
We continue the discussion started in Remark \ref{rem:comparisonWrinklesI}. If we restrict to first order jet spaces and impose that $D$ is a ball (as in the third item above), the front projection of a closed wrinkle will be precisely the standard \emph{wrinkle embedding} introduced by Eliashberg and Mishachev in \cite{ElMiWrinEmb}.

If additionally $Y$ has one-dimensional fibres (i.e. we restrict to the contact setting), the closed wrinkle itself is the model \emph{wrinkled legendrian} defined by Murphy in \cite{Mur}.

Lastly, if we pass to second order jets, take $Y$ with one-dimensional fibres, and still take $D$ to be a ball, the projection to $J^1$ of a closed wrinkle will be the model \emph{wrinkled legendrian} defined by \'Alvarez-Gavela in \cite{Gav2}. Do note that this is different from Murphy's definition.\hfill$\triangle$
\end{remark}

\begin{figure}[ht]
\centering
\includegraphics[width = \linewidth ]{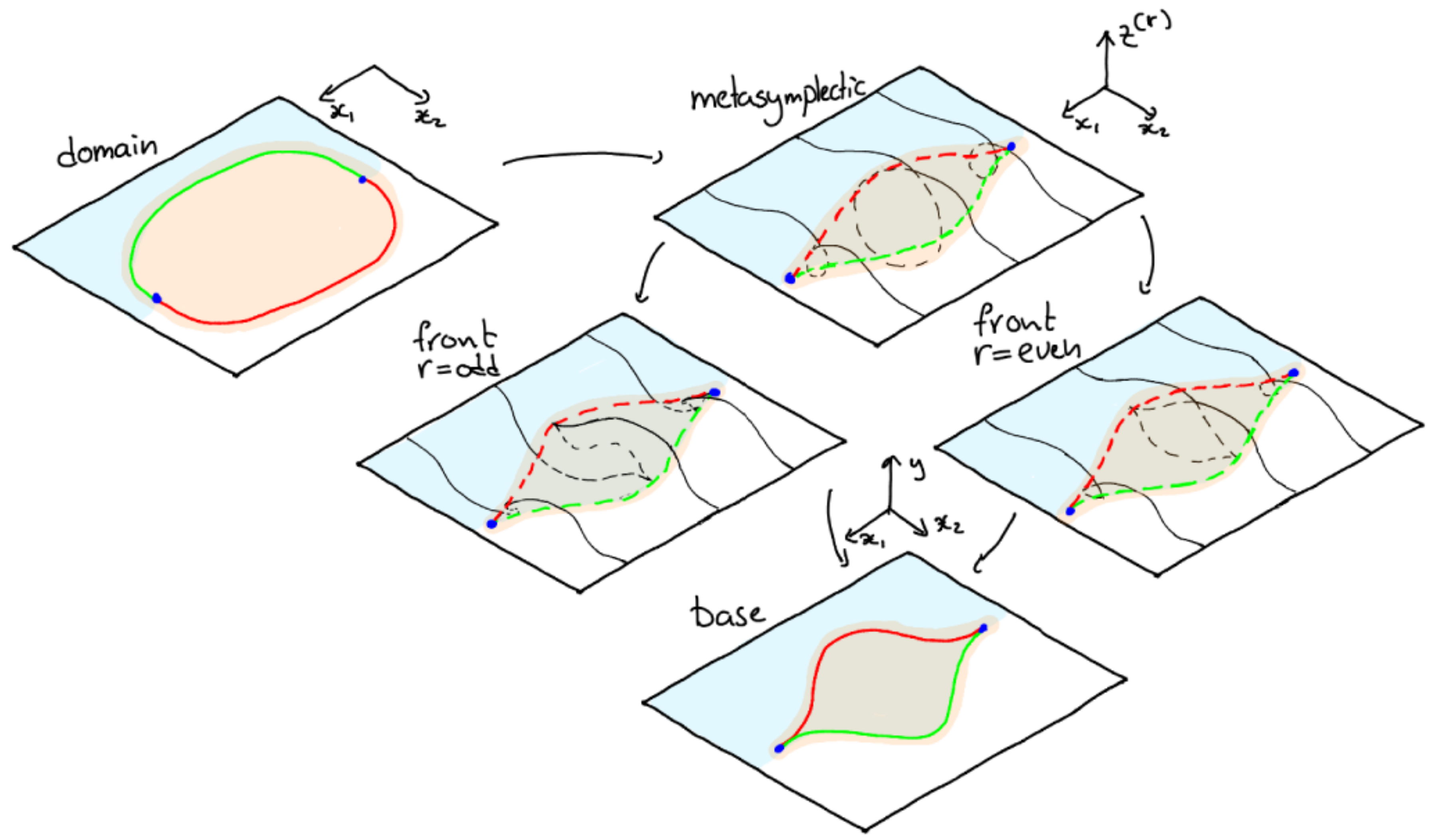}
\caption{Illustration of a closed wrinkle. Note that the configuration of the front (open or closed) is precisely opposite to that of the wrinkle, see Figure \ref{fig:Singularities_Wrinkle}.}\label{fig:Singularities_ClosedWrinkle}
\end{figure}

\subsubsection{Regularisation}

Let $g: M \to (J^r(Y),\xi_\can)$ be an integral mapping with a model closed wrinkle with membrane $A$. Reasoning as in Subsection \ref{sssec:regularisationClosed} we can replace $dg|_{\Op(A)}$, yielding a bundle monomorphism $\rho: TM \to \xi_\can$ with $\rho(TM|_{\Op(A)})$ transverse to $V_\can$. The two are not homotopic relative to the boundary of the model. We call this the \textbf{formal regularisation}. See Figure \ref{fig:Singularities_WrinkleRegularization}.

%% file: PaperI-HolonomicApproximation.tex
%
%
%
%
%
%
\section{Holonomic approximation by multi-sections} \label{sec:holonomicApproxZZ}

The main result of this section (Theorem \ref{thm:holonomicApproxZZ}, Subsection \ref{ssec:holonomicApproxZZ}) is an $h$-principle with PDE flavour. It states that holonomic approximation applies to closed manifolds as long as we are willing to allow some extra flexibility and consider multi-sections. A particular consequence is that any open partial differential relation admits a solution in the class of multi-sections (Corollary \ref{cor:holonomicApproxZZ}). Furthermore, following the wrinkling philosophy, it turns out to be sufficient to work with multi-sections with simple singularities.

A precise description of what a multi-section is and of the singularities that we use is given in Subsection \ref{ssec:multisections}. The key object behind our arguments, the zig-zag bump function, is introduced in Subsection \ref{sssec:zigzagBump}. Its properties are described in Subsection \ref{ssec:interpolating} and their construction is explained in Subsection \ref{ssec:zigzagBumpConstruction}. We complete the proof of Theorem \ref{thm:holonomicApproxZZ} in Subsection \ref{ssec:holonomicApproxProof}. Its parametric and relative analogues will be presented in Section \ref{sec:holonomicApproxParam}.

As in previous sections, we fix a smooth fibre bundle $Y \to X$. We work on the jet space $J^r(Y)$. In order to quantify how close two sections of $J^r(Y)$ are, we fix an auxiliary metric.

\subsection{Multi-sections with zig-zags} \label{ssec:multisections}

We dedicated Section \ref{sec:singularities} to the study of integral maps, relying particularly on principal projections. We now change our viewpoint and favour instead the front projection:
\begin{definition} \label{def:multisection}
Let $M$ be a manifold of dimension $\dim(X)$. A map $f: M \to Y$ is a ($r$-times differentiable) \textbf{multi-section} if:
\begin{itemize}
\item[a.] There is a dense subset $U \subset M$ such that $f|_U$ is an immersion transverse to the fibres of $Y$.
\item[b.] There is a smooth integral lift $j^rf: M \to J^r(Y)$.
\end{itemize}
\end{definition}
First note that being immersed and transverse to the fibres of $Y$ is an open condition. It follows that the dense subset $U \subset M$ given in Property (a) may be assumed to be open. In general, $\pi_b \circ f|_U: U \to X$ will be an immersion and not an embedding, even if we restrict to path-components of $U$. Nonetheless, given $p \in U$, there exists an open $p \in V \subset U$ such that $\pi_b \circ f|_V$ is an embedding of $V$ into $X$ that factors through $\Image(f|_V) \subset Y$. It follows that $\Image(f|_V)$ can be uniquely parametrised by an actual section $g_V: \pi_b \circ f(V) \to Y$, which has a unique holonomic lift $j^rg_V: \pi_b \circ f(V) \to J^r(Y)$. This provides (the unique) integral lift of $f|_V$ to $J^r(Y)$. Invoking the density of $U$, we deduce:
\begin{lemma}
Let $f$ be a multi-section. Then, its lift $j^rf$ is unique.
\end{lemma}
That is, the integral lift is defined automatically from Property (a.). However, it may fail to be smooth (or even continuous) in the complement of $U$, which is why we need Property (b.).


\begin{figure}[ht]
\centering
\includegraphics[width = 0.8\linewidth ]{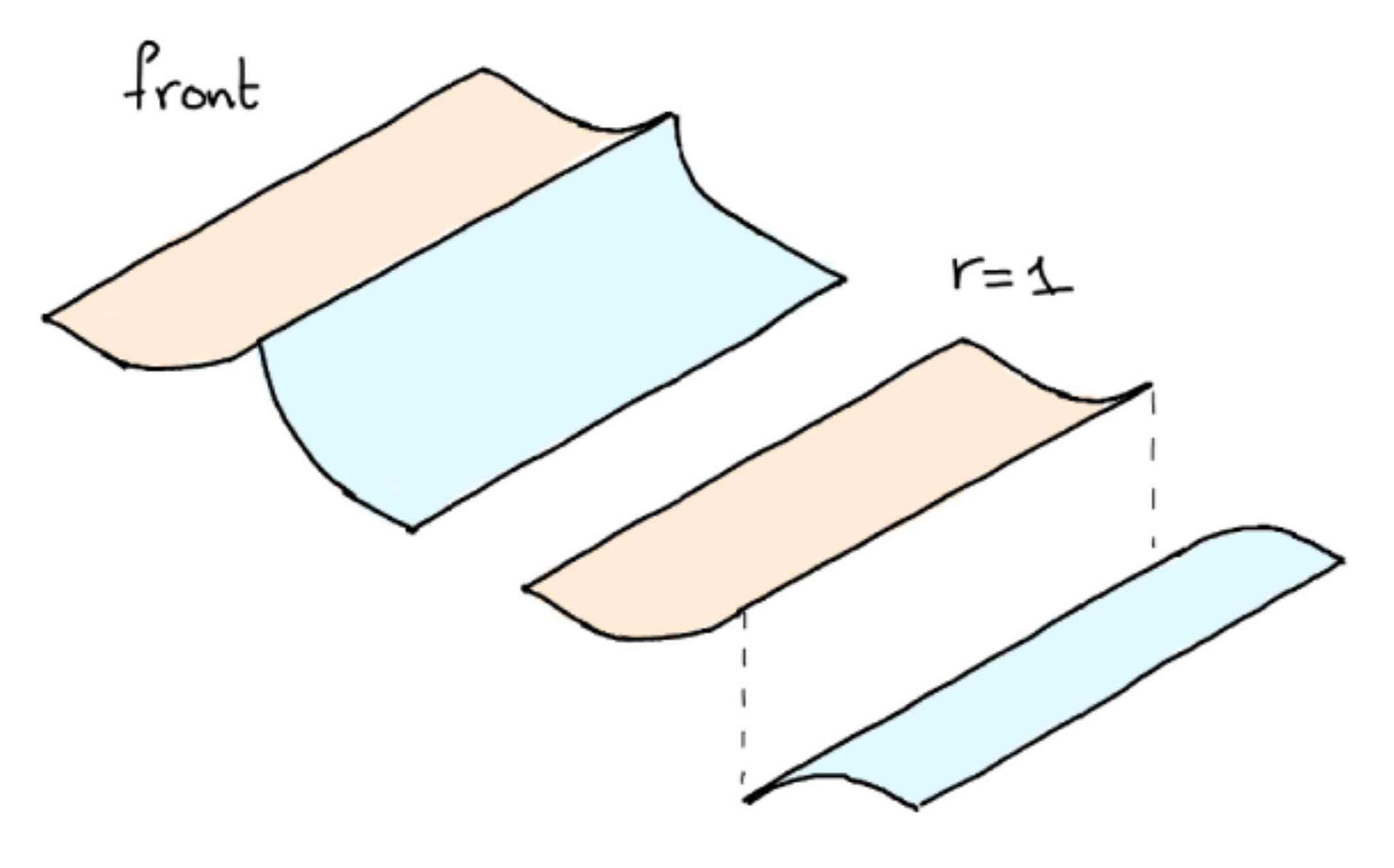}
\caption{A ridge singularity in the front projection (left) and its lift to $1$-jet space (right). In the front, the two branches meet each other with different slopes. Together they form a topological embedding. However, since the lift is discontinuous, they do not form a differentiable multi-section.}\label{fig:HolonomicApprox_ridge}
\end{figure}

\begin{remark}
Recall that our goal is the construction/classification of integral maps. A sufficiently rich collection of models (as established in Section \ref{sec:singularities}) allows us to deal with their singularities of tangency. Everywhere else, it is simpler to work using the front, as this boils down to manipulating sections of $Y$. This reduces the study of integral maps to the study of multi-sections. \hfill$\triangle$
\end{remark}

\subsubsection{Zig-zags} \label{sssec:zigZags}

The singularities of mapping of a general multi-section may be extremely complicated. This is already apparent in the study of frontal singularities of smooth legendrians. Due to this, we restrict ourselves to the following simple model:
\begin{definition} \label{def:zigzag}
A multi-section $f: M \to Y$ has a \textbf{zig-zag} along $A \subset M$ if its lift to $J^r(Y)$ has
\begin{itemize}
\item a double fold, if $r$ is even,
\item a closed double fold, if $r$ is odd,
\end{itemize}
along $A$. The base of a zig-zag is the base of the corresponding (closed) double fold. The same applies to the membrane $A$.
\end{definition}
Equivalently, a zig-zag consists of a pair of $A_{2r}$-cusps in the front projection in an open configuration. We note that this is indeed the simplest singularity lifting to $J^r(Y)$ that one may consider. For instance, the cusps $A_{2r'}$, $r'< r$, do not lift to continuous maps into $J^r(Y)$.

\begin{figure}[ht]
\centering
\includegraphics[width = \linewidth ]{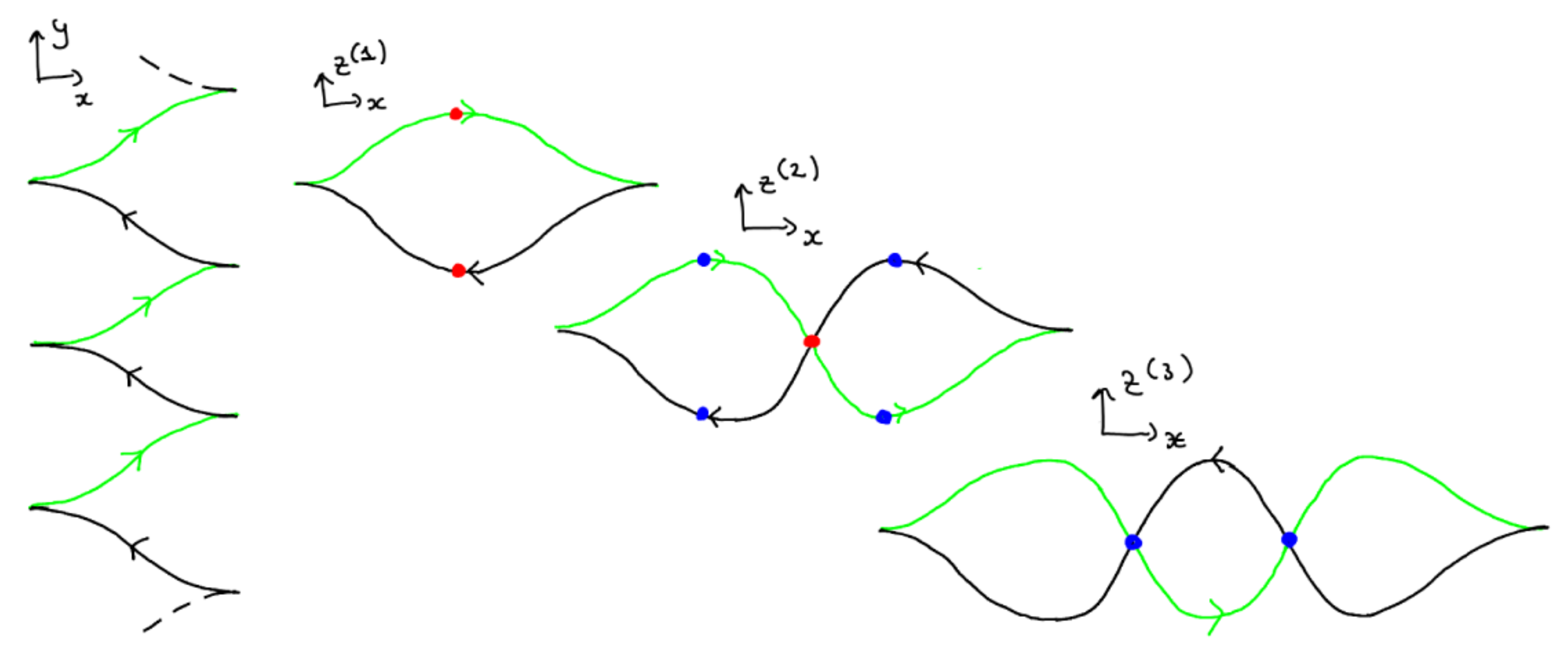}
\caption{On the left, a multi-section $f: \R \to J^0(\R,\R)$. It zig-zags upwards in a periodic manner. It follows that its derivatives are periodic; these are shown in the subsequent pictures. All the curves shown have cusps; these become milder as we pass to larger derivatives, becoming eventually smooth singularities of tangency in the metasymplectic projection. We see the switching phenomenon again: as we pass from one jet space to the next, we see that the cusp(s) on the right-hand side ``switch'', but the ones on the left do not. In particular, as we pass from $r=0$ to $r=1$, we see the zig-zag in the front, which is a sequence of cusps in an open configuration, become a closed configuration of two cusps.}\label{fig:Singularities_FigureEight}
\end{figure}

\begin{remark}
We used the name \emph{vertical cusp} to refer to the birth/death event associated to the stabilisation; see Subsection \ref{sssec:stabilisation}. This is a singularity for integral maps. We will henceforth use the name \emph{cusp} to refer to the $A_{2r}$-cusp (i.e. the front projection of the fold). This is a singularity for multi-sections. It should be clear from context which one is meant. \hfill$\triangle$
\end{remark}

\begin{definition} \label{def:sectionZigzag}
A \textbf{multi-section with zig-zags} is a pair $(f,\{A_i\})$ consisting of
\begin{itemize}
\item a multi-section $f: M \to Y$,
\item and a locally finite collection of disjoint annuli $A_i \subset M$
\end{itemize}
satisfying:
\begin{itemize}
\item $f$ is a topological embedding.
\item the restrictions $f|_{A_i}$ are zig-zags.
\item $f$ has no other singularities.
\end{itemize}
\end{definition}
Let us observe that a zig-zag is indeed a topological embedding. Furthermore:
\begin{lemma} \label{lem:liftZigzag}
Let $(f: M \to Y,\{A_i\})$ be a multi-section with zig-zags. Then the lift $j^rf: M \to J^r(Y)$ is an integral embedding.
\end{lemma}
\begin{proof}
Integrality follows from the fact that each cusp in a zig-zag lifts to a fold of $j^rf$, which has no singularities of mapping. Embeddedness follows from $f$ being a topological embedding.
\end{proof}

\begin{remark}
The reader may wonder why we require embeddedness, as this is irrelevant from a PDE perspective. The answer is that our results are meant to be used to construct and classify submanifolds tangent to the Cartan distribution. Embeddedness in the front projection allows us to ensure embeddedness in jet space upon lifting, as shown in Lemma \ref{lem:liftZigzag}. \hfill$\triangle$
\end{remark}

\subsection{The statement} \label{ssec:holonomicApproxZZ}

We now state the natural multi-section analogue of the holonomic approximation Theorem \ref{thm:holonomicApprox}. It reads:
\begin{theorem} \label{thm:holonomicApproxZZ}
Let $\sigma: X \to J^r(Y)$ a formal section. Then, for any $\varepsilon > 0$, there exists a multi-section with zig-zags $f: X \to Y$ satisfying $|j^rf - \sigma|_{C^0} < \varepsilon$.
\end{theorem}
In fact, by construction, all the zig-zags of $f$ will have spherical base. The rest of the section is dedicated to the proof of this statement. It will be immediate to the reader experienced in $h$-principles that a parametric and relative (in the domain and the parameter) version follows with minor adaptations. This is explained in detail in Section \ref{sec:holonomicApproxParam}. 

Instead of $r$-jets of sections, one can study $r$-jets of submanifolds and prove holonomic approximation for them. We will do this in Section \ref{sec:wrinkledEmbeddings}: Theorem \ref{thm:wrinkledEmbeddings} states that, using submanifolds with (potentially nested) zig-zags, one can approximate any formal homotopy of the $r$-jet of a submanifold. This generalises to higher jets the wrinkled embeddings of Eliashberg and Mishachev \cite{ElMiWrinEmb}. Theorem \ref{thm:holonomicApproxZZ} can be seen as a local version of Theorem \ref{thm:wrinkledEmbeddings}.

\subsubsection{Multi-solutions of PDRs}

Let $\SR \subset J^r(Y)$ be an open differential relation, not necessarily Diff-invariant. We will say that a multi-section $f: M \to Y$ is a \textbf{multi-solution} if $j^rf$ takes values in $\SR$.
\begin{corollary} \label{cor:holonomicApproxZZ}
Let $\SR \subset J^r(Y)$ be an open differential relation. Suppose $\SR$ admits a formal solution $\sigma: X \to \SR$. Then $\SR$ admits a multi-solution $f: X \to Y$. 
\end{corollary}
\begin{proof}
If $X$ is compact, there is some $\varepsilon > 0$, such that the $\varepsilon$-neighbourhood $U$ around $\sigma$ is contained in $\SR$. Then Theorem \ref{thm:holonomicApproxZZ} produces a multi-section with zig-zags $f$ whose lift $j^rf$ takes values in $U \subset \SR$, as desired. Otherwise, we apply the previous reasoning using an exhaustion by compacts.
\end{proof}
Observe that the constructed $f$ is topologically embedded, even though this is not important for the statement.

\begin{remark} \label{rem:DiffInvariance}
Given a differential relation $\SR_0$, we can always naturally associate to it a Diff-invariant counterpart $\SR_1$. Indeed, the solutions of $\SR_1$ should be integral mappings into $\SR_0$ lifting some diffeomorphism of the base manifold $X$. Diff-invariance is then automatically built-in. If $\SR_0$ was already Diff-invariant, the space of solutions of the latter is a trivial fibration over the space of solutions of the former with fibre $\Diff(X)$.

We can further consider the relation $\SR_2$ whose solutions are arbitrary integral mappings into $\SR_0$. Multi-solutions of $\SR_0$ are actual solutions of $\SR_2$. It is not difficult to produce examples where $\SR_2$ has solutions but $\SR_1$ and $\SR_0$ do not (in fact, they may not even have formal solutions). \hfill$\triangle$
\end{remark}

\subsubsection{The main ingredient of the proof} \label{sssec:zigzagBump}

We now introduce the simple observation that constitutes the basis of our work:
\begin{definition} \label{def:zigzagBump}
Let $I=[a,b]$ be an interval. A \textbf{zig-zag bump function} interpolating between $y_a$ and $y_b$ is a sequence of maps
\[ (\rho_N)_{N \in \N}: [a,b] \to J^0([a,b],\R) \]
satisfying:
\begin{itemize}
\item $\rho_N|_{\Op(a)}(t) = (x(t)=t, y(t)=y_a)$.
\item $\rho_N|_{\Op(b)}(t) = (x(t)=t, y(t)=y_b)$.
\item Each $\rho_N$ is a multi-section with zig-zags.
\item $|z^{(r')} \circ j^r\rho_N| = O(\frac{1}{N})$ for all $r'>0$.
\end{itemize}
\end{definition}

\begin{proposition} \label{prop:zigzagBumpConstruction}
A zig-zag bump function exists for any given $a$, $b$, $y_a$, $y_b$.
\end{proposition}
We postpone the proof to Subsection \ref{ssec:zigzagBumpConstruction}. It is based on the multi-section depicted in Figure \ref{fig:Singularities_FigureEight}.

We first explain (Subsection \ref{ssec:interpolating}) how the sequence $(\rho_N)_{N \in \N}$ may be used to interpolate between two actual sections. It is precisely their role in interpolating that motivates us to call them ``bump functions''. However, unlike normal bump functions, the $\rho_N$, for $N$ sufficiently large, allow us to interpolate in a controlled manner without introducing big derivatives.

\subsection{Interpolating using zig-zags} \label{ssec:interpolating}

Our statement about quantitatively controlled interpolation takes place in a chart. We work with a vector space $B$ of dimension $n$ as the base and we let $F$, a $k$-dimensional vector space, play the role of the fibre. We fix standard coordinates $(x,y,z)$ and use the resulting Euclidean metric in $J^r(B,F)$ to measure how close sections are. We write $\D_r$ for the $r$-ball; $\D$ denotes the unit ball.
\begin{proposition} \label{prop:zigzagBump}
Let $\varepsilon, \delta > 0$ be given. Consider sections 
\[ s_0,s_1: \D \subset B \to J^0(B,F) \]
satisfying $|j^rs_0(x) - j^rs_1(x)| < \varepsilon$ whenever $|x| \in [1-\delta,1]$.

Then, there exists a multi-section with zig-zags $f: \D \to F$ satisfying:
\begin{itemize}
\item $f|_{\Op(\partial\D)} = s_0$.
\item $f|_{\D_{1-\delta}} = s_1$.
\item $|j^rs_1 - j^rf|_{C^0} < 5\varepsilon$.
\end{itemize}
\end{proposition}

\begin{figure}[ht]
\centering
\includegraphics[width = \linewidth ]{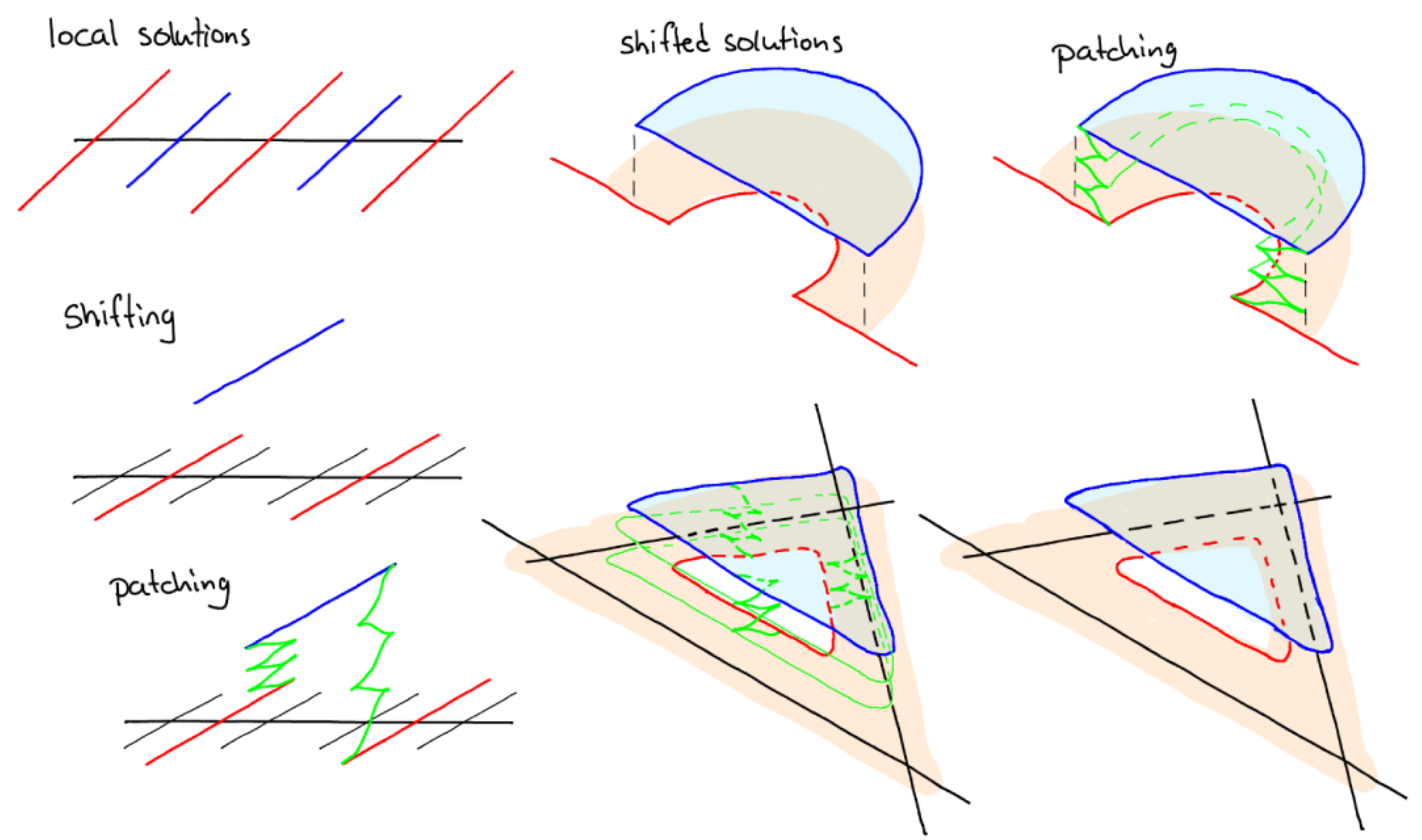}
\caption{On the left, from top to bottom, the steps in the proof of Proposition \ref{prop:zigzagBump}. First, we find solutions over each open cell (in blue) and on a neighbourhood of the codimension-one skeleton (in red). Then, the solutions on the cells are shifted so that their images are disjoint. Lastly, they are patched together using the zig-zag bump function. On the right, the shifted and patched solutions are drawn in higher dimensions, and in a triangulated space. The construction yields accordion-looking configurations made up of several concentric spheres of zig-zags.}\label{fig:HolonomicApprox_Mainproof}
\end{figure}

As long as we set $f|_{\partial\D_{1-\delta}} = s_1$, we can simply extend $f$ to the interior as $s_1$. Due to this, we henceforth restrict the domain of $s_0$ and $s_1$ to the region of interest $\NS^{n-1} \times [1-\delta,1]$. We break down the proof into steps. The argument is depicted in Figure \ref{fig:HolonomicApprox_Mainproof}.

\subsubsection{The pushing trick} We shift $s_1$ by adding a constant in $F$:
\[ \wtd s_1(x) := s_1(x) + (2\varepsilon,0,\dots,0). \]
Since $|s_0 - s_1|_{C^0} < \varepsilon$, replacing $s_1$ by $\wtd s_1$ guarantees that:
\[ \wtd s_1(x) \neq s_0(x), \qquad\textrm{ for every $x \in \NS^{n-1} \times [1-\delta,1]$,} \]
while retaining a bound $|\wtd s_1 - s_0|_{C^r} < 3\varepsilon$. 

\subsubsection{First simplification} We can simplify the setup by applying the fibrewise translation:
\begin{align*}
J^0(\NS^{n-1} \times [1-\delta,1],F) \quad\longrightarrow\quad & J^0(\NS^{n-1} \times [1-\delta,1],F) \\
p                                       \quad\mapsto\quad & p - s_0(\pi_b(p)),
\end{align*}
It preserves the $C^r$-distance and maps $s_0$ to the zero section. The section $\wtd s_1$ is mapped to $s := \wtd s_1 - s_0$. Consequently, we just need to explain how to interpolate between the zero section and some arbitrary section $s$ satisfying $|s|_{C^r} < 3\varepsilon$ and $s(x) \neq 0$ for all $x$, as the radius goes inwards from $1$ to $1-\delta$.

\subsubsection{Second simplification} A second symmetry allows us to put $s$ in normal form. Indeed, due to the nature of the shift we performed, we have that 
\[ \varepsilon < |y_1 \circ s(x)| < 3\varepsilon \]
for all $x$. This allows us to define a framing 
\begin{align*}
A: \NS^{n-1} \times [1-\delta,1] \quad\longrightarrow\quad & \GL(F) \\
x                             \quad \mapsto \quad & A(x) := (s,e_2,e_3,\dots,e_k),
\end{align*}
where $\{e_j\}_{j=1,\dots,k}$ is the framing dual to the coordinates $y_i$ in $F$. The framing $A$ defines a fibre-preserving transformation of the $F$-bundle by left multiplication. By construction $Ae_1 = s$.

\subsubsection{Main construction} We apply Proposition \ref{prop:zigzagBumpConstruction} to produce a zig-zag bump function 
\[ (\rho_N)_{N \in \N}: [1-\delta,1] \quad\to\quad J^0([1-\delta,1],\R) \]
interpolating between $1$ and $0$. The expression $\rho_N(t)e_1$ denotes then a sequence of multi-sections of $F$.

\begin{remark} \label{rem:multiplication}
In general, given a section $a: X \to Y$ and a multiply-valued function $b: M \to J^0(X,\R)$, we can define their multiplication to be multi-section $c: \quad M \to J^0(Y)$ given by 
\[c(t) := (x \circ b(t), y \circ b(t)a(x \circ b(t)).\]
This generalises the usual notion of multiplication of a section by a function. \hfill$\triangle$
\end{remark}

We can use $A$ to produce another sequence of multi-sections:
\begin{align*}
f_N: \NS^{n-1} \times [1-\delta,1] \quad\longrightarrow\quad & J^0(\NS^{n-1} \times [1-\delta,1],F) \\
(\tilde x, t)                      \quad\mapsto\quad & A(\rho_N(t)e_1)
\end{align*}
interpolating between $s$ and $0$. We claim that, for $N$ large enough, the mapping $f_N$ satisfies the properties given in the statement.

\subsubsection{Checking the claimed properties} We claim that $f_N$ is a multi-section with zig-zags. Indeed, $\rho_N$ is a multi-section with zig-zags and $\rho_Ne_1$ is its stabilisation. We conclude by noting that $j^rf_N$ is obtained from $j^r(\rho_Ne_1)$ by applying the point symmetry $j^rA$.

The second claim is that each $f_N$ agrees with $s$ in a neighbourhood of $\NS^{n-1} \times \{1-\delta\}$ and with $0$ in a neighbourhood of $\NS^{n-1} \times \{1\}$. This is clear by construction, since $Ae_1 = s$ and $\rho_N$ is identically $0$ when the radius is close to one and identically $1$ when the radius is close to $1-\delta$.

The final claim is that 
\[ |j^rf_N|_{C^0} < 4\varepsilon \]
if $N$ is large enough. From this bound the desired estimate will follow:
\[ |j^rs - j^rf_N|_{C^0} \leq |j^rs|_{C^0} + |j^rf_N|_{C^0} < 5\varepsilon. \]
Since $\rho_N$ is the graph of an actual function over a dense set, we can carry out our computations using said function; we abuse notation and still denote it by $\rho_N$. In this manner, $f_N$ itself can be regarded, over a dense set, to be the function $\rho_Ns$.

For each multi-index $I$ with $|I| \leq r$:
\[ |\partial^I(\rho_Ns)|^2 = \left| \sum_{I'+I'' = I} (\partial^{I'}\rho_N)(\partial^{I''}s)\right|^2 \leq 
                             \sum_{I'+I'' = I}|\partial^{I'}\rho_N|^2 |\partial^{I''}s|^2  \]
Now, each derivative $|\partial^{I'}\rho_N|$ is smaller than $1/N$, with the exception of $|\rho_N| = 1$.

Let $K_1$ be the maximum number of decompositions $I'+I'' = I$ that a multi-index $|I| \leq r$ in $n$ variables and $k$ outputs may have. Let $K_2$ be the number of multi-indices $|I| \leq r$. Then:
\begin{align*}
|\partial^I(\rho_Ns)|^2 < \quad&\quad |\partial^I s|^2 + \frac{9K_1}{N^2}\varepsilon^2 \\
|\rho_Ns|_{C^r}^2 <       \quad&\quad \sum_I \left(|\partial^I s|^2 + \frac{9K_1}{N^2}\varepsilon^2 \right) < |s|_{C^r}^2 + \frac{9K_1K_2}{N^2}\varepsilon^2.
\end{align*}
Therefore, by setting $N^2 > 9K_1K_2$, we conclude:
\[ |j^rf_N|_{C^0}^2 = |\rho_Ns|_{C^r}^2 < |s|_{C^r}^2 + \varepsilon^2 < 16\varepsilon^2, \]
as desired. \hfill $\Box$

\begin{remark} \label{rem:randomExtensionInterior}
Observe that the restriction $s_0|_{\D_{1-\delta}}$ plays absolutely no role in the statement or the proof. In practice (e.g. in the proof of Theorem \ref{thm:holonomicApproxZZ}) $s_0$ will only be given to us on a neighbourhood of $\partial\D$, so we will choose $\delta$ sufficiently small so that $s_0$ is defined for radii $|x| \geq 1-\delta$. An arbitrary extension of $s_0$ to the interior will then allow us to apply the proposition. \hfill$\triangle$
\end{remark}

\begin{remark}
An important feature of the proof is that the sections with zig-zags we construct are obtained from the ``standard'' sections with zig-zags $\rho_Ne_1$ by applying fibre-preserving diffeomorphisms. This implies that we do not need any form of stability for singularities in order to carry out our arguments (as we put the desired models ``by hand''). We do this intentionally, as we want to avoid stability discussions (particularly once we introduce birth/death events or wrinkles with more general bases). 

Indeed: the issue of stability for singularities that lift to jet space was first studied by V. Lychagin \cite{Ly80b} (for $1$-jet spaces in more than one variable). General jet spaces were later studied by A. Givental \cite{Giv}. In these works they already observe that even Whitney-type singularities may not be stable if the fibre has dimension at least $2$. \hfill$\triangle$
\end{remark}

\subsection{Constructing zig-zag bump functions}  \label{ssec:zigzagBumpConstruction}

We now provide a proof for Proposition \ref{prop:zigzagBumpConstruction}. Once again, we break it into steps, see Figure \ref{Fig:HolonomicApprox_Zigzag}.

\subsubsection{Reduction to a standard square}

We claim that it is sufficient to prove the claim for $I=[-1,1]$. Indeed, let $\Psi$ be a diffemorphism between two closed intervals. Due the compactness of the interval, the point symmetry $j^r\Psi$ distorts jet space in a bounded manner. It follows that this distortion can be absorbed by taking a larger $N$. We conclude that $\Psi$ maps zig-zag bump functions to zig-zag bump functions.

Furthermore, it is enough to prove the case $y_a=0$ and $y_b=1$. The general case follows by applying a point symmetry.

\subsubsection{One level in the infinite zig-zag}

Let
\begin{align*}
Z: \Op([1,3]) \quad\longrightarrow\quad J^0(\R,\R)
\end{align*}
be a germ of zig-zag with membrane $[1,3]$. This implies that $Z([1,3])$ is graphical over $\R$. Up to applying a point symmetry, we may assume that:
\begin{itemize}
\item The horizontal coordinate $x \circ Z$ is decreasing for $t \in (1,3)$ and $x \circ Z([1,3]) = [-1,1]$.
\item The vertical coordinate $y \circ Z$ is increasing.
\end{itemize}
We claim that we can extend $Z$ to a map with domain $\Op([-1,3])$, still denoted by $Z$, that satisfies:
\begin{itemize}
\item It is a topological embedding.
\item It is graphical away from $\{-1,1,3\}$.
\item $Z|_{\Op(0)}(t) = (t,0)$.
\item $Z(t) = Z(t-4)+4$ for $t \in \Op(3)$.
\end{itemize}
Indeed, we define $Z$ over $\Op(-1)$ using the last condition. Similarly we define it in $\Op(0)$ using the third property. Lastly, we extend over $(-1,1)$ in a graphical manner avoiding $Z([-1,1])$.

\begin{figure}[ht]
\centering
\includegraphics[width = \linewidth ]{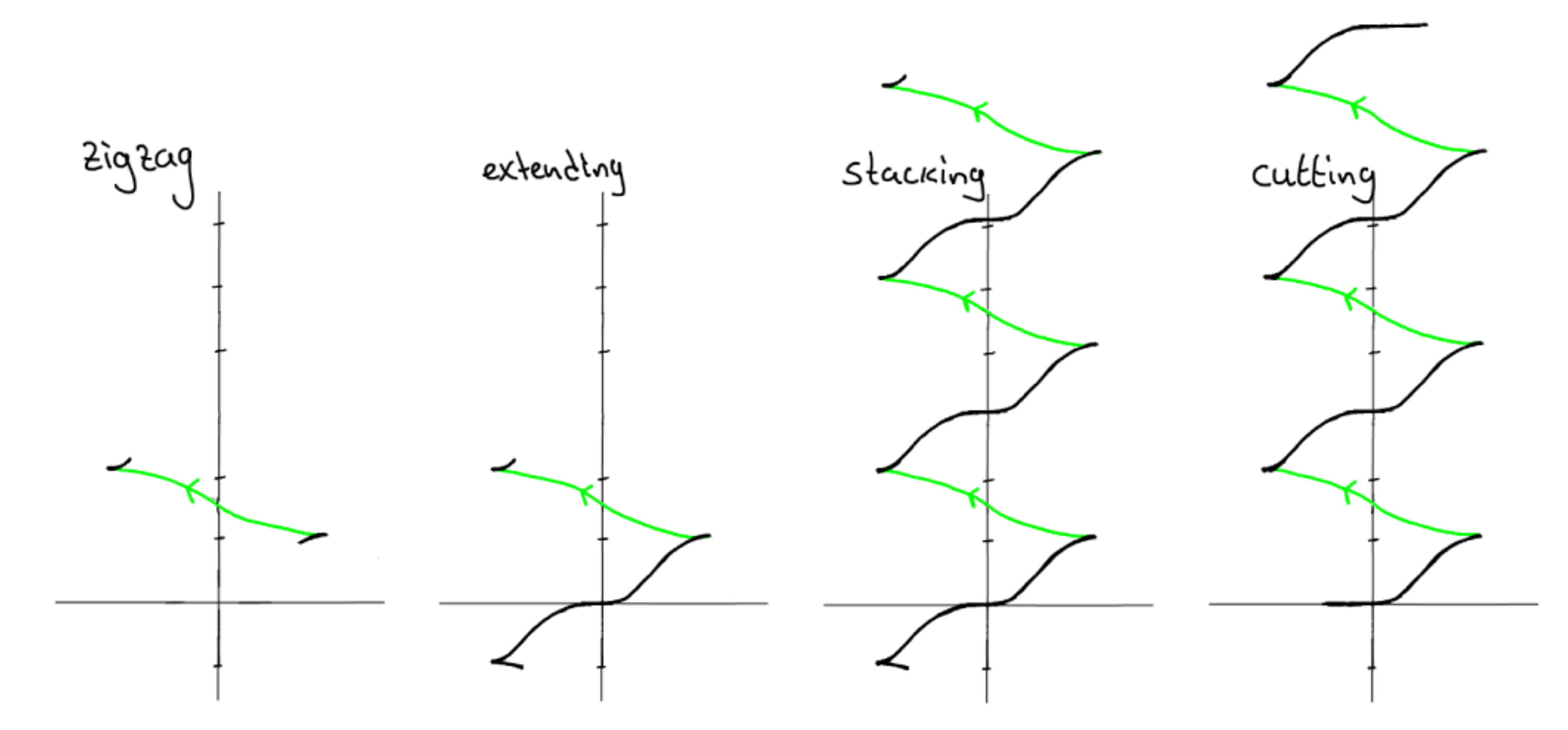}
\caption{From left to right, the various steps in the construction of the zig-zag bump function. We start with a zigzag (defined as a germ around its membrane). Then we extend it to obtain one full period of the infinite zig-zag. Stacking these levels on top of each other produces the infinite zig-zag. Finally, we cut out a finite piece and flatten the ends. The one step we do not picture is the rescaling, that takes the finite piece and compresses it so that it interpolates between $0$ and $1$.}\label{Fig:HolonomicApprox_Zigzag}
\end{figure}

\subsubsection{The infinite zig-zag}

Thanks to its defining properties, we can extend $Z$ to a topological embedding with domain $\R$ and satisfying the periodicity condition $Z(t) = Z(t-4)+4$. We call it the \textbf{infinite zig-zag bump}. It has a zig-zag in each interval $[4l+1,4l+3]$, $l \in \Z$.

Its lift $j^rZ: \R \to J^r(\R,\R)$ is an integral mapping with fold singularities. Since the front $Z$ is topologically embedded, $j^rZ$ is also embedded. With the exception of $r'= 0$, all other derivatives $z^{(r')} \circ j^rZ$ are truly periodic and thus bounded.

\begin{remark}
Depending on the choices made during the construction, the resulting infinite zig-zag (and its derivatives) will look slightly different. One possible choice is shown in Figure \ref{fig:Singularities_FigureEight}.\hfill$\triangle$
\end{remark}

\subsubsection{Rescaling}

Given a positive integer $N$, we let $\Psi_N$ be the point symmetry corresponding to rescaling $J^0(\R,\R)$ vertically by an amount $\frac{1}{4N}$. We write $Z_N(t) := \Psi_N \circ Z(t)$. By construction, $Z_N(0) = (0,0)$ and $Z_N(4N) = (0,1)$.

It follows from the nature of scaling that 
\[ |z^{(r')} \circ j^rZ_N| = O(1/N). \]
Do observe that these bounds do not depend on the parametrisation of an integral mapping, only on its image.

\subsubsection{A piece of the infinite zig-zag}

Lastly, we explain how to produce the desired $\rho_N$ from the map $Z_N$.

We pick a diffeomorphism $\phi: [-1,1] \to [-1,1+4N]$. The map $Z_N \circ \phi$ satisfies the desired bound on derivatives. In order to introduce the claimed boundary conditions we observe that $(Z_N \circ \phi)|_{\Op(\phi^{-1}(0))}$ maps to the zero section. This means that we can replace the map over the interval $[0,\phi^{-1}(0)]$ by a map into the zero section that is the zero section itself in $\Op(0)$. Similarly, $(Z_N \circ \phi)|_{\Op(\phi^{-1}(4N))}$ maps to the constant section with value one so we perform a similar replacement over $[\phi^{-1}(4N),1]$. This concludes the proof. \hfill$\Box$

\subsection{Proof of the Theorem} \label{ssec:holonomicApproxProof}

The proof of Theorem \ref{thm:holonomicApproxZZ} follows the standard structure of an $h$-principle argument. 

In Subsection \ref{sssec:holonomicApproxReduction} we prove the \emph{reduction step}. Its output is a holonomic section $g$, defined along the codimension-$1$ skeleton of $X$ and approximating the given formal section $\sigma$.

In Subsection \ref{sssec:holonomicApproxExtension} we provide the \emph{extension argument}: we extend $g$ to the interior of the top dimensional cells. In order to obtain a good approximation of $\sigma$, the extension to the interior must be a multi-section, as presented in Proposition \ref{prop:zigzagBump}.

\subsubsection{Preliminaries} \label{sssec:holonomicApproxPreliminaries}

We must fix some auxiliary data first. Depending on the constant $\varepsilon >0$ we fix a finite collection of pairs $\{(U_i,f_i)\}$ such that
\begin{itemize}
\item $\{U_i\}$ is a covering of $X$ by balls,
\item $f_i: U_i \to J^r(Y|_{U_i} \to U_i)$ is a holonomic section satisfying $|f_i - \sigma|_{U_i}|_{C^0} < \varepsilon/10$.
\end{itemize}
The existence of such a collection follows from the standard holonomic approximation Theorem \ref{thm:holonomicApprox} applied to each point in $X$. By compactness of $X$ we get a finite refinement.

We then triangulate $X$, yielding a triangulation $\ST$. We assume that this triangulation is fine enough to guarantee that each simplex is contained in one of the $U_i$. Given a top-simplex $\Delta \in \ST$, we choose a preferred $U_i$ and we denote the corresponding section $f_i$ by $f_\Delta$.

We remark that $Y|_{U_i}$ is trivial, so we can choose an identification $J^r(Y|_{U_i}) \cong J^r(\D,\F)$. We can then relate the $C^0$-distance in the former with the standard $C^0$-distance in the latter. By finiteness of the cover there is a constant bounding one in terms of the other. We assume this constant is $1$ to avoid cluttering the notation.

\subsubsection{Reduction} \label{sssec:holonomicApproxReduction}

The codimension-$1$ skeleton of $X$ is a CW-complex of positive codimension. Thus, according to Theorem \ref{thm:holonomicApprox}, there exists:
\begin{itemize}
\item a wiggled version $\wtd\ST$ of $\ST$,
\item a holonomic section $g: \Op(\wtd\ST) \to Y$ satisfying $|\sigma-j^rg|_{C^0} < \varepsilon/10$.
\end{itemize}
The wiggling can be assumed to be $C^0$-small, so each top-simplex $\Delta \in \wtd\ST$ is contained in the same $U_i$ as the original simplex. I.e., we have sections $g$ (defined over $\Op(\partial\Delta)$) and $f_\Delta$ (defined over the whole of $\Delta$), both of them $\varepsilon$-approximating $\sigma$.

\subsubsection{Extension} \label{sssec:holonomicApproxExtension}

We focus on a single top-simplex $\Delta \in \wtd\ST$; the argument is the same for all of them. We apply Proposition \ref{prop:zigzagBump} and Remark \ref{rem:randomExtensionInterior} to $g$ and $f_\Delta$ over an inner collar of $\partial\Delta$. This produces the desired multi-section extension $f$ of $g$ to the interior of $\Delta$. The Proposition guarantees that:
\[ |j^rf-\sigma|_{C^0} < |j^rf-j^rf_\Delta|_{C^0} + |j^rf_\Delta-\sigma|_{C^0} < \varepsilon \]
over each cell. This concludes the proof of Theorem \ref{thm:holonomicApproxZZ}. \hfill$\Box$

\begin{remark}
Here is an extremely biased comment about the proof of Theorem \ref{thm:holonomicApproxZZ}: The strategy presented (zig-zag bump functions together with the pushing trick) is simpler than the path followed in \cite{ElMiWrinEmb} (reducing to simple tangential homotopies along a single derivative and approximating them with wrinkles). Furthermore, the connection to classic holonomic approximation is more transparent with our implementation.

However, the idea of working one direction at a time translates nicely to the contact setting. This is precisely how \'Alvarez-Gavela proceeds in \cite{Gav2}. It is not clear to us whether our proof may be adapted to also reprove \cite{Gav2}, but it seems challenging. The reason is that our argument resembles classic holonomic approximation, but the statements in \cite{Gav2} require holonomic approximation relative to first order (as proven in \cite{Gav1}).  \hfill$\triangle$
\end{remark}

%% file: PaperI-HolonomicApproximationParam.tex
%
%
%
%
%
%
\section{Parametric holonomic approximation by multi-sections} \label{sec:holonomicApproxParam}

In this section we state and prove the parametric and relative version of Theorem \ref{thm:holonomicApproxZZ}. We denote our $k$-dimensional parameter space by $K$. We also use $k \in K$ to denote elements of $K$; this should not cause any major confusion. As in previous sections, we write $J^r(Y) \to Y \to X$ for the space of jets under study.

\subsection{Families of zig-zags} \label{ssec:zigzagParam}

In parametric families, double folds (Definition \ref{def:doubleFold}) can (dis)appear in Reidemeister I moves (Definition \ref{ssec:ReidemeisterI}). The same applies to closed double folds (Definition \ref{def:closedDoubleFold}), that disappear in stabilisation events (Definition \ref{def:stabilisation}). Now we need to extend the notion of a multi-section with zig-zags (Definition \ref{def:sectionZigzag}) to include the appropriate birth-death behavior. We will require that this happens in a wrinkly self-cancelling manner.

\subsubsection{Zig-zag wrinkles}

We recall notation from Subsections \ref{ssec:wrinkles} and \ref{ssec:closedWrinkles}. We fix a $(n-1)$-manifold $H$, a submanifold $D \subset K \times H$ with non-empty interior (and possibly with boundary), and a function $\rho$ that is negative in the interior of $D$, positive in the complement, and cuts $\partial D$ transversely. This data defines for us the sublevel set
\[ A := \{ t^2 + \rho(k,\wtd x) \leq 0 \} \subset K \times H \times \R. \]
\begin{definition} \label{def:zigZagParam}
The \textbf{model zig-zag wrinkle} with base $D$, height $\rho$, and membrane $A$
\[ \Op(\partial A) \longrightarrow K \times J^0(H \times \R,\R) \]
is the germ along $A$ of the front projection of:
\begin{itemize}
\item The model wrinkle with base $D$ and height $\rho$, if $r$ is even.
\item The model closed wrinkle with base $D$ and height $\rho$, if $r$ is odd.
\end{itemize}

A fibered-over-$K$ family of multi-sections
\[ g = (g_k)_{k \in K}: M \to J^0(Y) \]
has a \textbf{zig-zag wrinkle} along $A' \subset K \times M$ if it is point equivalent along $\partial A'$ to the model zig-zag wrinkle along $\partial A$. The region $A'$ is said to be the membrane of the wrinkle. The map $g|_{A'}$ is said to be a model zig-zag wrinkle if this equivalence can be extended to the interior of $A'$. 
\end{definition}
The equator of a zig-zag wrinkle is the corresponding object in the (closed) wrinkle lifting it. 

\subsubsection{Swallowtails and swallowtail moves}

Some simple instances of zig-zag wrinkles that are relevant for us are the following:
\begin{definition}\label{def:swallowtail}
A zig-zag wrinkle is said to be a \textbf{swallowtail} if it is (equivalent to) the front projection of:
\begin{itemize}
\item The pleat, if $r$ is even.
\item The closed pleat, if $r$ is odd.
\end{itemize}
\end{definition}

\begin{definition}
A zig-zag wrinkle is said to be a \textbf{swallowtail move} if it is (equivalent to) the front projection of:
\begin{itemize}
\item The Reidemeister I move, if $r$ is even.
\item The stabilisation, if $r$ is odd.
\end{itemize}
The projection of a (closed) double fold embryo (up to equivalence) will be called a zig-zag embryo.
\end{definition}
In particular, a zig-zag wrinkle whose base is a submanifold of $H$ (i.e. $K$ is just a point) has swallowtails along the equator. Similarly, a zig-zag wrinkle whose base is a submanifold of $K$ (i.e. $H$ is a point) has swallowtail moves along the equator.

\begin{definition}
A zig-zag wrinkle is said to have ball (resp. cylinder) base if it is the front projection of a (closed) wrinkle with ball (resp. cylinder) base.
\end{definition}

\subsubsection{Families of multi-sections with zig-zag wrinkles}

The parametric analogue of Definition \ref{def:sectionZigzag} reads:
\begin{definition}\label{def:sectionZigzagParam}
A \textbf{wrinkle family of multi-sections} is a pair $(f,\{A_i\})$ consisting of:
\begin{itemize}
\item A $K$-family of multi-sections $f = (f_k)_{k \in K}: M \to Y$.
\item A locally finite collection of disjoint submanifolds $A_i \subset K \times M$ (possibly with boundary).
\end{itemize}
Satisfying:
\begin{itemize}
\item Each $f_k$ is a topological embedding.
\item $f$ has a model zig-zag wrinkle with membrane $A_i$, for each $i$. 
\item $f$ has no other singularities.
\end{itemize}
\end{definition}
Do note that a model zig-zag wrinkle is indeed a topological embedding. We write $j^rf$ for the family of fibrewise lifts $(j^rf_k)_{k \in K}$. Then:
\begin{lemma} \label{lem:liftZigzagWrinkle}
Let $(f: K \times M \to K \times Y,\{A_i\})$ be a wrinkle family of multi-sections. Then:
\begin{itemize}
\item If $r$ is even, the maps $j^rf_k$ are embeddings.
\item If $r$ is odd, the maps $j^rf_k$ are topological embeddings.
\end{itemize}
\end{lemma}
\begin{proof}
The topological embedding condition for $j^rf$ follows from the corresponding assumption on $f$. Moreover, if $r$ is even, the model zig-zag wrinkle lifts to a model wrinkle, which it itself a smooth embedding. 
\end{proof}

\begin{remark}
Let us emphasise: For $r$ odd, a zig-zag wrinkle lifts to a family of integral maps with singularities of mapping (which may be closed pleats, stabilisations, or generalisations thereof). In Appendix \ref{sec:desingularisationOdd} we will address this issue. Namely, we will introduce a surgery procedure that gets rid of these singularities, yielding a smooth integral embedding (at the price of introducing self-intersections in the front projection).

Observe that such a surgery procedure cannot exist in contact topology, where we know that legendrians exhibit rigidity, but it will work in any other odd jet space. A more sophisticated version of this surgery is precisely the main ingredient behind the loose legendrians of E. Murphy \cite{Mur}. The analogue of Murphy's surgery, in the context of general jet spaces, will appear in the sequel paper \cite{PT2}, where it will be used to prescribe the singularities of tangency of integral submanifolds. \hfill $\triangle$
\end{remark}

\subsection{The statement} \label{ssec:holonomicApproxParam}

The parametric version of holonomic approximation for multi-sections states:
\begin{theorem}\label{thm:holonomicApproxParam}
Fix a constant $\varepsilon > 0$. Let 
\[ (\sigma_k)_{k \in K}: X \longrightarrow J^r(Y) \]
be a $K$-family of formal sections. Then, there exists a wrinkle family of multi-sections $(f_k)_{k \in K}: X \to Y$ satisfying 
\[ |j^rf_k - \sigma_k|_{C^0} < \varepsilon. \]

Moreover, if the family $(\sigma_k)_{k \in K}$ is holonomic on a neighborhood of a polyhedron $L \subset K \times M$, then we can arrange $j^rf = \sigma$ over $\Op(L)$.
\end{theorem}
Even further:
\begin{corollary} \label{cor:holonomicApproxCylinder}
The wrinkles of the family $(f_k)_{k \in K}: X \to Y$ produced by Theorem \ref{thm:holonomicApproxParam} may be assumed to have cylinder base.
\end{corollary}
We will prove this together with Theorem \ref{thm:holonomicApproxParam}. In Appendix \ref{sec:surgeries} we will introduce a surgery procedure to replace arbitrary zig-zag wrinkles by wrinkles with ball base (as often encountered in other papers on wrinkling), proving:
\begin{corollary}  \label{cor:holonomicApproxBall}
The wrinkles of the family $(f_k)_{k \in K}: X \to Y$ produced by Theorem \ref{thm:holonomicApproxParam} may be assumed to have ball base.
\end{corollary}

\subsubsection{h-Principle for multi-solutions}

Let us spell out some concrete consequences of the result: If $K=[0,1]$ and $\sigma_0$ and $\sigma_1$ are holonomic, the statement says that it is possible to approximate any formal homotopy between the two by a formal homotopy through multi-sections. Analogous statements for higher-dimensional families follow similarly.

From this, a parametric and relative version of Corollary \ref{cor:holonomicApproxZZ} may be proven:
\begin{corollary}
Any formally trivial family of solutions of an open differential relation is also trivial as a family of multi-solutions.
\end{corollary}
We leave the details of the proof to the reader.

\subsubsection{The main ingredient of the proof}

The proof of Theorem \ref{thm:holonomicApproxParam} will follow closely the proof of Theorem \ref{thm:holonomicApproxZZ}. The additional ingredient that we need is a concrete model for the birth/death of a zig-zag bump function.

We fix coordinates $(\ss,t)$ in the square $[c,d] \times [a,b] \subset \R^2$, where the first term is thought of as a parameter space. In order to distinguish source and target, we use coordinates $(x,y)$ in $J^0([a,b],\R)$.
\begin{definition}\label{def:zigzagBumpParam}
A \textbf{swallowtail family of bump functions} (over $[c,d] \times [a,b]$) is a sequence of fibered-over-$[c,d]$ maps
\[ (\nu_N)_{N \in \N}: [c,d] \times [a,b] \to [c,d] \times J^0([a,b],\R)\]
interpolating between the constant section
\[ \nu_N(c,t) = (x(t) = t, y(t)= 0) \in J^0([a,b],\R) \]
and a zig-zag bump function $(\nu_N(d,\cdot))_{N \in \N}$ that itself interpolates between $1$ and $0$.

Furthermore, we require that the singularity locus satisfies:
\begin{itemize}
\item Each $\nu_N$ is a wrinkle family of multi-sections.
\item All the singularities of the lift  
\[ j^r\nu_N: [c,d] \times [a,b] \to [c,d] \times J^r([a,b], \R) \]
are swallowtail moves.
\end{itemize}
Morever, the following boundary conditions should hold:
\begin{itemize}
\item $\nu_N(\ss,t) = (\ss,t,0)$ for all $t \in \Op(b)$ and all $\ss$. 
\item $\nu_N(\ss,t) = (\ss,t,h_N(\ss))$ for all $t \in \Op(a)$ and all $\ss$, for some function $h_N: [c,d] \to [0,1]$.
\end{itemize}
As well as the estimate:
\begin{itemize}
\item $|z^{r'} \circ j^r\nu_N| = O\left(\frac{1}{N}\right)$ for all $0 < r' \leq r$,
\end{itemize}
\end{definition}
We refer the reader to Figure \ref{fig:swallowtailBump} for an illustration. We remark that it is certainly possible to define families of bump functions with wrinkles subject to other boundary conditions (for instance, families that interpolate between values other than $1$ and $0$). However, the stated constraints are tailored to the proof of Theorem \ref{thm:holonomicApproxParam}.

\begin{figure}[ht]
\centering
\includegraphics[width =\linewidth ]{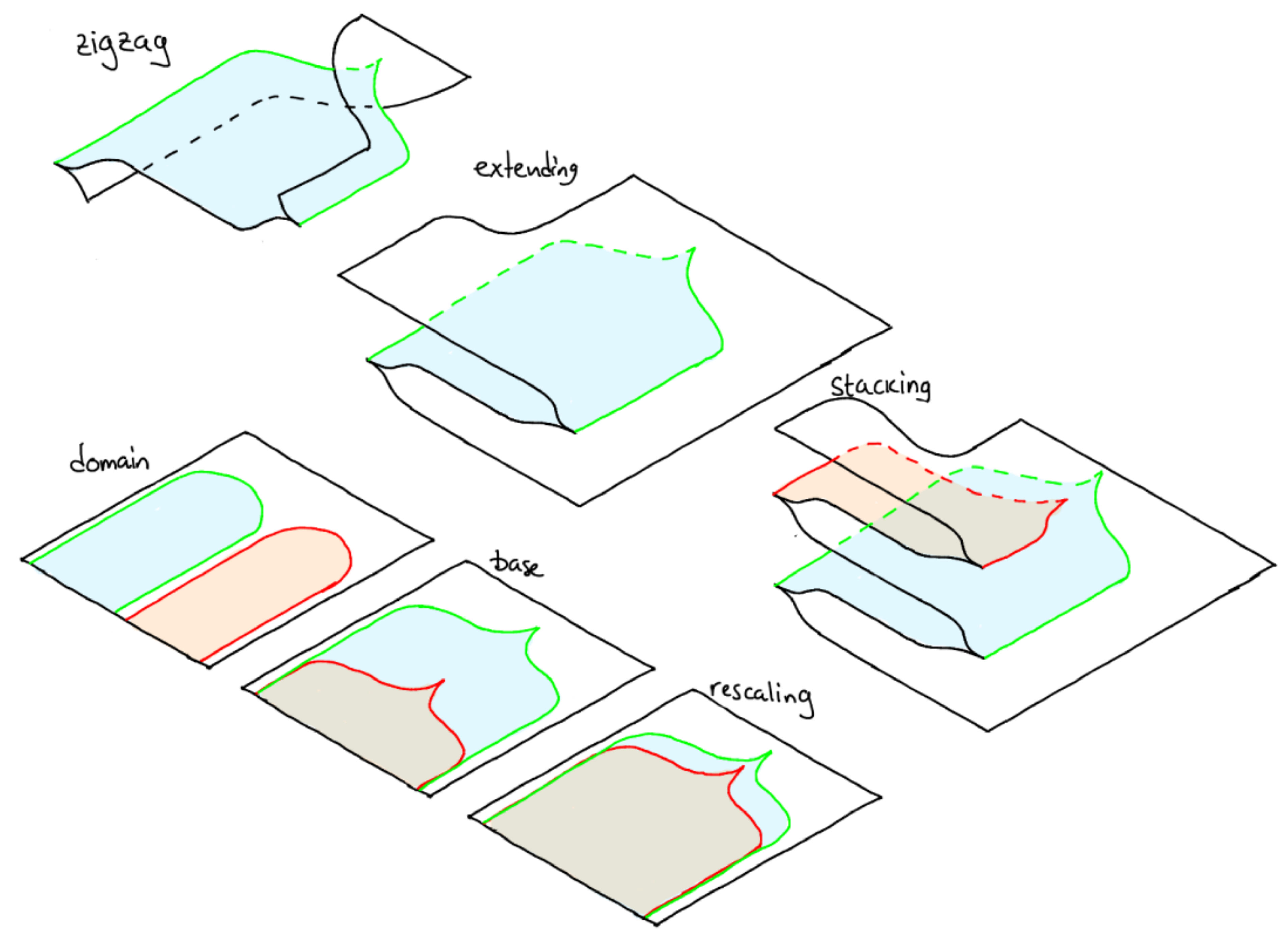}
\caption{Construction of the swallowtail family of bump functions; as explained in Proposition \ref{prop:swallowtailBump}. We start with a single swallowtail, given as a germ along the membrane. We then extend and flatten it. Thanks to this flattening, the resulting map can be stacked.} \label{fig:swallowtailBump}
\end{figure}

\begin{proposition}\label{prop:swallowtailBump}
A swallowtail family of bump functions exists on any square $[c,d] \times [a,b]$.
\end{proposition}

\subsection{Construction of swallowtail families}

We will present $\nu_N$ as a $1$-parameter family of maps 
\[ (\nu_{N,\ss})_{\ss \in [c,d]}: [a,b] \longrightarrow J^0([a,b], \R). \]
We will follow the strategy used in the non-parametric case, defining first a single stage and then introducing $\nu_N$ as a concatenation of multiple copies of this first stage.

\subsubsection{Passing to a standard square}

As in the non-parametric case, we can assume that $[a,b] = [-1,1]$ by applying a point symmetry and absorbing the resulting distortion by taking a larger $N$.

We claim that we may also assume $[c,d] = [-1,1]$. Indeed, any diffeomorphism $[c,d] \to [-1,1]$ does the trick. The reason is that none of the claimed properties are affected by reparametrising in the $\ss$-coordinate. More concretely, the only derivatives entering the definition are along $t$, not along the parameter direction $\ss$. This flexibility is crucial and will be exploited in the upcoming construction.
 
\subsubsection{The swallowtail move} 

Extend the $(\ss,t)$ coordinates to the plane $\R \times \R$. Similarly, extend the coordinates $(x,y)$ to the whole of $J^0(\R,\R)$.

Looking at the definitions of the Reidemeister I move (Subsection \ref{ssec:ReidemeisterI}) and the stabilisation (Subsection \ref{sssec:stabilisation}), we see that the swallowtail move is the germ along $A = \{t^2 \leq \ss\} \subset \R \times \R$ of some mapping
\[ f: \Op(A) \to \R \times J^0(\R,\R) \]
that is fibered over the $\ss$-variable and whose singularity locus is the parabola $\partial A = \{t^2 = \ss\}$. We write $f_\ss = f(\ss,-)$. Then, $f_0$ has a zig-zag embryo at the origin. If $\ss<0$, the map $f_\ss$ is graphical over $t$. If $\ss>0$, the map $f_\ss$ is a zig-zag with cusps at $\pm\sqrt{\ss}$.

We will modify the family $f$, using point symmetries and reparametrisations, in order to produce $\nu_1$. This will take several steps. Abusing notation, the result of each step will still be denoted by $f$. We restrict to the region $\ss \in [-1,1]$. 

\subsubsection{Construction of the first stage} 

Applying $\ss$-dependent point symmetries we may assume that:
\begin{itemize}
\item $x \circ f_\ss$ is strictly increasing if $\ss < 0$.
\item $x \circ f_\ss|_{[-\sqrt{\ss},\sqrt{\ss}]}$ is strictly decreasing if $\ss>0$.
\item $y \circ f_\ss$ is strictly decreasing in $\Op(\{\pm\sqrt{\ss}\})$ if $\ss>0$.
\item Each $f_\ss$ has image in $[-1,1] \times (0,1)$.
\item $x \circ f_1([-1,1]) = [-1,1]$.
\end{itemize}
Namely: The first two items are equivalent and follow by possibly reverting the orientation of $x$. The third one by reverting the orientation of $y$. The third one by translating and rescaling first in $x$ and then in $y$. The last one by scaling in $x$ (only for $\ss$ close to $1$) and recalling that $[-1,1]$ is the membrane of $f_1$.

We reparametrise the $f_s$ so that their membranes are contained in $[-1/4,1/4]$. We still use $A \subset [0,1] \times [-1/4,1/4] \subset [-1,1]^2$ to denote the membrane of the resulting map.

In the complement of $A$, the maps $f_s$ are graphical over $x$. We can then extend $f$ to the square $[-1,1]^2$ so that it verifies the boundary conditions:
\begin{itemize}
\item $f_s|_{[1/2,1])}(t) = (t,0)$.
\item $f_{-1}(t) = (t,0)$.
\item $f_1(-3/4) = (3/4,1)$ and $y \circ f_1|_{[-1,-3/4]}(t) = 1$.
\end{itemize}
This can be done while keeping $f_s$ a topological embedding, with image in $[-1,1] \times [0,1]$, and graphical in the complement of the membrane $A$ of the swallowtail move.

Let $\chi: [-1,1] \to [-1,1]$ be a map that restricts to an orientation-preserving diffeomorphism $(-1/2,1/2) \to (-1,1)$ and is constant in the complement. Then, the claimed family 
\[ \nu_1: [-1,1]^2 \longrightarrow [-1,1] \times J^0([-1,1], \R) \]
is given by the reparametrisation $\nu_1(\ss,t) := (\ss,f_{\chi(\ss)}(t))$.

\subsubsection{Stacking}

The boundary conditions of $\nu_1$ were designed so that multiple copies of $\nu_1$ may be stacked on top of one other, as we explain next. We will construct a sequence of maps $F_N$ and then obtain $\nu_N$ by scaling down (the reader may think of the $F_N$ as the parametric analogue of the infinite zig-zag). The argument is inductive in nature.

We set $F_1 = \rho_1$. Our induction hypothesis says that we have constructed $F_N$ satisfying:
\begin{itemize}
\item $(F_N)|_{[-1,-1/2] \times [-1,1] \cup [-1,1] \times [1/2,1]}(\ss,t) = (\ss,t,0)$, i.e. it is the zero section.
\item $(F_N)|_{[1/2,1] \times [-1,-3/4]}$ is a reparametrisation of the constant section
\[ (\ss,x) \in [1/2,1] \times [-1,3/4] \quad\mapsto\quad (\ss,x,N) \in [1/2,1] \times J^0(\R,\R). \]
\end{itemize}
$F_1$ does satisfy these properties, so the base case holds. The idea for the inductive step is to replace $(F_N)|_{[1/2,1] \times [-1,-3/4]}$ with a vertically shifted copy of $F_1$, yielding $F_{N+1}$.

Let $\alpha: [1/2,1] \to [-1,1]$ be an orientation-preserving diffeomorphism. Then we set
\[ G(\ss,t) := (\ss,F_1(\alpha(\ss),t)). \]
This map can then be vertically translated by an amount $N$ (which is a point symmetry); we still denote the result by $G$. We can then restrict its domain to $[1/2,1] \times [-1,3/4]$ and apply a diffeomorphism $[-1,-3/4] \to [-1,3/4]$ in $t$. Our induction hypotheses say that (up to matching the two parametrisations in a neighbourhood of the glueing region) we can replace $(F_N)|_{[1/2,1] \times [-1,-3/4]}$ by $G$.

The resulting map $H$ satisfies the first property in the induction hypothesis. We can then choose a diffeomorphism $\beta: [-1,1] \to [-1,1]$ mapping $[1/2,1]$ to $\Op(1)$. This dilates (in $\ss$) the region in which the ``upper branch'' of $H$ has value $N+1$. It follows that
\[ F_{N+1}(\ss,t) := (\ss,H(\beta(\ss),t)) \]
satisfies both inductive properties, concluding the construction.

\subsubsection{Size estimates and the end of the proof}

We built $F_N$ out of copies of $F_1$. In the $t$-coordinate we applied translations and reparametrisations to each copy of $F_1$, but not scalings. It follows that the derivatives $|z^{r'} \circ F_N|$ are bounded independently of $N$, for all $0 < r' \leq r$.

Let $\Psi: J^0(\R,\R) \to J^0(\R,\R)$ be the point symmetry associated to scaling $y$ by an amount $1/N$. We can then set $\nu_N := j^r\Psi(F_N)$ and it immediately follows that
\[ |z^{r'} \circ \nu_N| =  O\left(\dfrac{1}{N}\right), \quad \text{for all $0 < r' \leq r$}. \]
The rest of the claimed properties have to do with the boundary values of $\nu_N$, which hold by construction. This concludes the proof of Proposition \ref{prop:swallowtailBump}.\hfill$\Box$

\begin{remark}
We were rather careful in our construction and provided a ``periodic'' model for $\nu_N$, for which the desired estimates followed easily. However, we could have gotten away with a sloppier construction using the following observation:

Suppose we have constructed a sequence $(\beta_N)_{N \in \N}$ satisfying all the claimed properties except for the key estimate. Observe that $\beta_N(1,-)$ is still a zig-zag, so the estimate holds nonetheless for $\ss=1$. Then, for each $N$, we can find an $\ss$-dependent family of point symmetries $\Psi_\ss$ that are the identity close to $\ss=1$ but that outside of $\{\ss \in \Op(1)\}$ scale down the $y$-coordinate by a constant $C_N$. If $C_N$ is chosen to be sufficiently large, the key estimate will hold for $\Psi_\ss(\beta_N)$, which will be a swallowtail family. \hfill$\triangle$
\end{remark}

\subsection{Interpolation using swallowtail families} \label{ssec:interpolationParam}

We now state and prove the analogue of Proposition \ref{prop:zigzagBump} in the parametric setting. We work locally in parameter and domain. The role of the parameter space is played by a $k$-dimensional vector space still denoted by $K$. The base manifold will be a vector space $B$ of dimension $n$. The fibre will be a vector space $F$, its dimension $\ell$ is not relevant for the argument. We write $\D_r$ for the $r$-ball; $\D$ denotes the unit ball.

We fix standard coordinates $(x,y,z)$ and use the resulting Euclidean metric in $J^r(B,F)$ to measure how close sections are. In particular, whenever we use the $C^r$-distance, this takes into account only derivatives with respect to $B$ and not $K$.

\begin{proposition} \label{prop:interpolationParam}
Let $\varepsilon,\delta>0$ be given. Fix two $K$-families of sections 
\[ g_0,g_1 : \D^k \times \D^n \subset K \times B \to K \times J^0(B,F) \]
that satisfy $|j^rg_0(k,x) - j^rg_1(k,x)| < \varepsilon$ in the complement $A$ of $\D_{1-\delta}^k \times \D_{1-\delta}^n$.

Then, there is a wrinkle family of multi-sections
\[ f: \D^k \times \D^n \subset K \times B \to K \times J^0(B,F) \]
satisfying:
\begin{itemize}
\item $f|_{\Op(\partial(\D^k \times \D^n)} = g_0$;
\item $f|_{\D_{1-\delta}^k \times \D_{1-\delta}^n} = g_1$;
\item $|j^rf - j^rg_1|_{C^0} < 4\varepsilon$.
\end{itemize}
\end{proposition}
\begin{proof}
The proof follows closely the argument of Proposition \ref{prop:zigzagBump}, but replacing the zig-zag bump function by a swallowtail family. The first steps in the proof (the ``pushing trick" and the ``first and second simplifications") go through word for word. After those, we have reduced the statement to the following special case:
\begin{itemize}
\item $g_0$ is the zero section.
\item $g_1$ is $\varepsilon$-disjoint from the zero section and $|g_1|_{C^r} < 3\varepsilon$ holds.
\end{itemize}

We restrict our attention to the collar $A$. We fix a swallowtail family of bump functions $(\nu_N)_{N \in \N}$ with domain $[1-\delta,1]^2$; recall the associated family of functions $h_N: [1-\delta,1] \to [0,1]$ appearing in Definition \ref{def:zigzagBumpParam}. We also fix a bump function $\chi:  [0,1] \to [-1,1]$ that is $1$ in $\Op([0,1-\delta])$ and $-1$ in $\Op(1)$. 

Replace the coordinate $x$ by $(\theta,r)$ in the region $C = \{|x| \in [1-\delta,1]\} = \NS^{n-1} \times [1-\delta,1]$. Introduce the multiply-valued function
\[ F_N(k,\theta,r) := (k,\theta,\nu_N(\chi(|k|),r)): \D^k \times C \to K \times J^0(B,\R). \]
It allows us define:
\begin{itemize}
\item $f_N(k,x) = g_0(k,x) = 0$ in $\Op(\partial(\D^k \times \D^n))$.
\item $f_N(k,x) = g_1(k,x)h_N(\chi(|k|))$ in $\D^k \times \D_{1-\delta}^n$.
\item $f_N = g_1F_N$ in $\D^k \times C$. Recall that the multiplication of a section by a multi-valued function was defined in Remark \ref{rem:multiplication}.
\end{itemize}
The three defining regions overlap with each other but, when they do, the definitions of $f_N$ agree. Indeed, for the last two items, this follows from the fact that $h_N$ is the boundary condition of $F_N$ at $r=1-\theta$. For the first two items, we recall that $h_N$ is zero close to $-1$. Lastly, for the first and last, we note that $g_1F_N$ is identically zero for $|k|$ close to one, due to the definition of a swallowtail family. These verifications also show that $f_N$ interpolates between $g_1$ (in the interior) and $g_0$ (in the boundary).

We observe that non-graphicality for $f_N$ happens only in the region $\D^k \times C$ and there it follows from the use of $\nu_N$, which is a topological embedding (as is $F_N$). The singularities of $f_N$ are also located in $\D^k \times C$. We claim that they are zig-zag wrinkles. Indeed, the multi-function $\nu_N(\chi(|k|),r)$ by itself is a zig-zag wrinkle whose base is a ball in parameter space. Introducing the $\theta$-coordinates stabilises it by $\NS^{n-1}$. 

The proof concludes if we prove that $f_N$, for $N$ large enough, satisfies the claimed bound on size. This is proven as in Proposition \ref{prop:zigzagBump} (see the step ``Checking the claimed properties''), relying on the corresponding estimate for $\nu_N$.
\end{proof}

\begin{figure}[ht]
\centering
\includegraphics[width =\linewidth ]{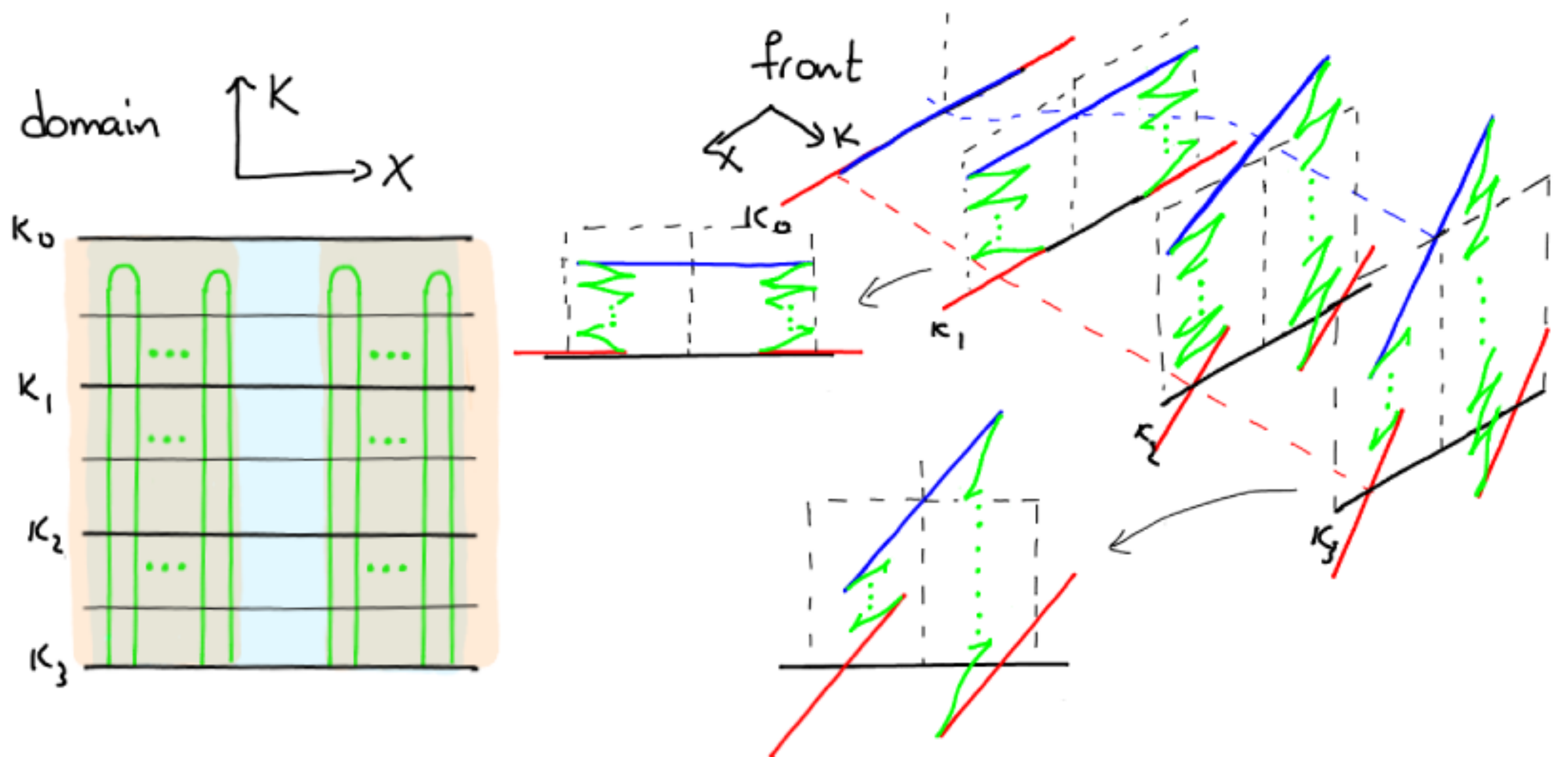}
\caption{The parametric version of Figure \ref{Fig:HolonomicApprox_Zigzag}. We are given a homotopy of formal data in which the zero jet stays fixed, but the formal first derivative is increasing over time (parametrised by $k \in K$). We approximate this using a wrinkled family of multi-sections, all whose zig-zag wrinkles have cylinder base. We see this as a movie in which various zig-zags are born. The singularity locus, on the product $K \times X$, is depicted on the left.} \label{Fig:HolonomicApprox_ZigzagParam}
\end{figure}

\subsection{Proof of Theorem \ref{thm:holonomicApproxParam}}

Much like for Theorem \ref{thm:holonomicApproxZZ}, the proof consists of a reduction step and a extension step, now taking place in the product manifold $K \times X$. We work directly in the relative setting, letting $L \subset K \times X$ be a polyhedron such that $\sigma$ is holonomic on a neighbourhood $U \supset L$.

We choose a triangulation $\wtd \ST$ of $K \times X \setminus U'$, where $U' \subset U$ is a smaller neighbourhood of $L$. In order to deal with the parametric setting we need a triangulation that is in general position with respect to the fibres of $K \times X \to K$. That this is possible follows from Theorem \ref{thm:Thurston}: the desired triangulation $\ST$ is obtained from $\wtd \ST$ by subdivision and jiggling.

We require $\ST$ to be fine enough, so that a $\varepsilon$-holonomic approximation $s_\Delta$ of $(\sigma_k)_{k \in K}$ can be chosen over each top cell $\Delta \in \ST$. 

For the reduction step: An application of Theorem \ref{thm:holonomicApprox} yields a holonomic approximation $s_\ST$ of $(\sigma_k)_{k \in K}$ along the codimension-one skeleton of some wiggled version $\ST'$ of $\ST$. The triangulation $\ST'$ can be guaranteed to remain in general position with respect to the fibres if we wiggle along the $X$-direction. Furthermore, making the wiggling sufficiently $C^0$-small ensures that $s_\Delta$ is still defined over the wiggled copy $\Delta' \in \ST'$ of $\Delta$.

For the extension step we use Proposition \ref{prop:interpolationParam} in each top-cell $\Delta' \in \ST'$. Namely, note that $\Delta' \setminus \partial \Delta'$ is indeed a fibered open polydisc in $K \times X$, which implies that we can parametrise $\Delta'$ (up to an arbitrary small neighbourhood of the boundary) by $\D^k \times \D^n$ in a fibered manner. The proposition then allows us to interpolate between $s_\Delta$ and $s_\ST$ in the region $\Op(\partial\Delta')$, yielding the desired family $f = (f_k)_{k \in K}$ of multi-sections approximating $(\sigma_k)_{k \in K}$. 

By construction, the zig-zag wrinkles of the family $f$ are all wrinkles with cylinder base. Since the triangulation was chosen relative to $L$, the family $j^rf$ agrees with $\sigma$ close to $L$. This concludes the proof and proves additionally Corollary \ref{cor:holonomicApproxCylinder}. \hfill$\Box$

%% file: PaperI-WrinkledEmbeddings.tex
%
%
%
%
%
%
\section{Holonomic approximation for submanifolds}\label{sec:wrinkledEmbeddings}

In this Section we will prove the analogue of Theorems \ref{thm:holonomicApproxZZ} and \ref{thm:holonomicApproxParam} in the setting of jet spaces of submanifolds. Our main result, Theorem \ref{thm:wrinkledEmbeddings}, is a generalisation to jet spaces of arbitrary order of the wrinkled embeddings h-principle due to Eliashberg and Mishachev \cite{ElMiWrinEmb}.

We state Theorem \ref{thm:wrinkledEmbeddings} in Subsection \ref{ssec:wrinkledEmbeddings}. Before we get there, we need to introduce the singularities that we allow our submanifolds to develop in order to prove holonomic approximation; this is done in Subsection \ref{ssec:submanifoldsZigzag}.

Recall the notation introduced in Subsection \ref{ssec:jetsSubmanifolds}. We fix an ambient manifold $Y$. We write $J^r(Y,n)$ for the space of $r$-jets of $n$-dimensional submanifolds of $Y$. Given an $n$-submanifold $X \subset Y$, its holonomic lift to $J^r(Y,n)$ is denoted by $j^rX$.

\subsection{Wrinkled submanifolds} \label{ssec:submanifoldsZigzag}

\begin{definition}
A ($r$-times differentiable) \textbf{singular submanifold} of $Y$ is a subset $X \subset Y$ satisfying:
\begin{itemize}
\item[a.] There is a dense subset $U \subset X$ that is an $n$-dimensional submanifold of $Y$.
\item[b.] There is a subset $j^rX \subset J^r(Y,n)$ lifting $X$ and extending the holonomic lift $j^rU$. By lifting we mean that $\pi_f: j^rX \to X$ is a homeomorphism.
\end{itemize}
\end{definition}
As for multi-sections, we note that the set $U$ in condition (a.) may be assumed to be open. Furthermore, due to density, the lift $j^rX$ is unique.

We remark that this definition is extremely general and not particularly useful by itself. Our next goal is to single out particular classes of singular submanifolds with controlled singularities. Recall that $J^r(Y,n)$ is locally modelled on a jet space of sections. Namely, given a submanifold $X \subset Y$ with normal bundle $\nu(X)$, we can find an embedding of $J^r(\nu(X))$ into $J^r(Y,n)$ preserving the Cartan distribution. This motivates us to consider singular submanifolds of $Y$ whose singularities are modelled on zig-zag wrinkles, mimicking the case of multi-sections.

\subsubsection{Singular embeddings}

We will go back and forth between (nicely behaved) singular submanifolds and the mappings parametrising them. Namely:
\begin{definition}
A ($r$-times differentiable) \textbf{singular embedding} is a topological embedding $f: M \to Y$ satisfying:
\begin{itemize}
\item[a.] There is a dense subset $U \subset M$ such that $f|_U$ is non-singular.
\item[b.] There is a smooth integral lift $j^rf: M \to J^r(Y,n)$.
\end{itemize}
\end{definition}
As above, the set $U$ may be assumed to be open and the lift is unique due to density.

\subsubsection{Wrinkled embeddings and submanifolds}

Let us observe that a multi-section is already a non-graphical map into the total space of the bundle. By considering general diffeomorphisms of the total space (not necessarily preserving the fibration structure), we can talk about maps that are locally modelled on a multi-section singularity:
\begin{definition}
Let $f = (f_k)_{k \in K}: M \to Y$ be a family of singular embeddings. Let $g$ be a multi-section zig-zag wrinkle (Definition \ref{def:sectionZigzagParam}) with membrane $A$.

We say that $f$ has a \textbf{zig-zag wrinkle} along $A' \subset K \times M$ if $f$ and $g$ are equivalent along $\partial A'$ and $\partial A$. If the equivalence extends to the interior, we say that $f|_{A'}$ is a model zig-zag wrinkle.
\end{definition}
We say that $A'$ is the membrane of $f$. We can similarly talk about the base, height, and equator of $f$.

\begin{definition}
A \textbf{wrinkled family of singular embeddings} is a pair $(f = (f_k)_{k \in K},\{A_i\})$ consisting of:
\begin{itemize}
\item a family of singular embeddings $f: K \times M \to K \times Y$,
\item a collection of disjoint submanifolds with boundary $A_i \subset K \times M$.
\end{itemize}
such that all the singularities of $f$ are zig-zag wrinkles with membranes $A_i$.
\end{definition}

In order to talk about submanifolds and not about embeddings, we introduce:
\begin{definition}
A \textbf{wrinkled family of singular submanifolds} of $Y$ is a pair $(X,\{f_i\})$ satisfying:
\begin{itemize}
\item $X = (X_k \subset Y)_{k \in K}$ is a family of singular submanifolds.
\item Each $f_i: \Op(A_i) \subset K \times M \to Y$ is a zig-zag wrinkle with image contained in $X$.
\item The images of the $f_i$ are disjoint.
\item The complement in $X$ of the images of the $f_i$ is a family of smooth submanifolds of $Y$.
\end{itemize}
\end{definition}

\subsection{Statement of the theorem} \label{ssec:wrinkledEmbeddings}

Suppose we are given a smooth manifold $Y$, an integer $n$, and an $n$-dimensional submanifold $X \subset Y$. An \textbf{$r$-jet homotopy} of $X$ is a homotopy $(F_s)_{s \in [0,1]}: X \to J^r(Y,n)$ satisfying $F_0 = j^rX$ and lifting the inclusion $X \to Y$. In the case $r=1$ these were introduced by Eliashberg and Mishachev under the name of \textbf{tangential homotopies}, since $J^1(Y,n)$ is simply the grassmannian of $N$-planes in $TY$.

The holonomic approximation question in $J^r(Y,n)$ asks: Is it possible to follow a given $r$-jet homotopy $(F_s)_{s \in [0,1]}$ by a homotopy of the submanifold $X$ itself? That is, given $\varepsilon > 0$, is there a family of embeddings $(f_s)_{s \in [0,1]}: X \to Y$ satisfying $|j^rf_s - F_s| < \varepsilon$ and with $f_0$ equal to the inclusion? 

If $X$ is closed, the answer to this question is, in general, negative (much like for sections). The main result of this section says that the statement can be salvaged if we work with singular submanifolds instead.

\subsubsection{Statement without parameters}

For simplicity, and in order to make the ideas transparent, we will address the non-parametric case first:
\begin{theorem} \label{thm:wrinkledEmbeddings}
Fix a smooth manifold $Y$, an integer $n$, a submanifold $X \subset Y$ and a $r$-jet homotopy $(F_s)_{s \in [0,1]}: X \to J^r(Y,n)$.

Then, there is a wrinkled family of singular embeddings 
\[ (f = (f_s: X \to Y)_{s \in [0,1]},\{A_i \subset [0,1] \times X\}) \]
satisfying:
\begin{itemize}
\item $|j^rf_s - F_s| < \varepsilon$.
\item $f_0: X \to Y$ is the inclusion.
\item All the singularities $f|_{A_i}$ are zig-zag wrinkles with cylinder base.
\end{itemize}

Additionally: If $F$ is already holonomic in a neighbourhood of a polyhedron $L \subset [0,1] \times X$, the singularities of $f$ may be taken to be disjoint from $L$.
\end{theorem}
The idea of the proof is the following: We can reduce to the case that $F_s$ is graphical over $F_0$ using a sufficiently fine partition of the $s$-interval. Seeing $F_0$ as the holonomic lift of the zero section and $F_s$ as a homotopy of formal sections allows us to invoke Theorem \ref{thm:holonomicApproxZZ}. This effectively introduces spheres of zig-zags, that are born in spherical embryos.

It is apparent that, after the first application of Theorem \ref{thm:holonomicApproxZZ}, we will not be dealing anymore with submanifolds, stopping us from reducing to the case of sections. We address this by applying (classic) holonomic approximation close to the wrinkle locus and arguing as above in the complement (which is indeed smooth).

We illustrate the argument in Figure \ref{fig:wrinkledEmbeddings}. This strategy to pass from graphical to non-graphical is standard and was already used in \cite{ElMiWrinEmb}.

\begin{figure}[ht]
\centering
\includegraphics[width =\linewidth ]{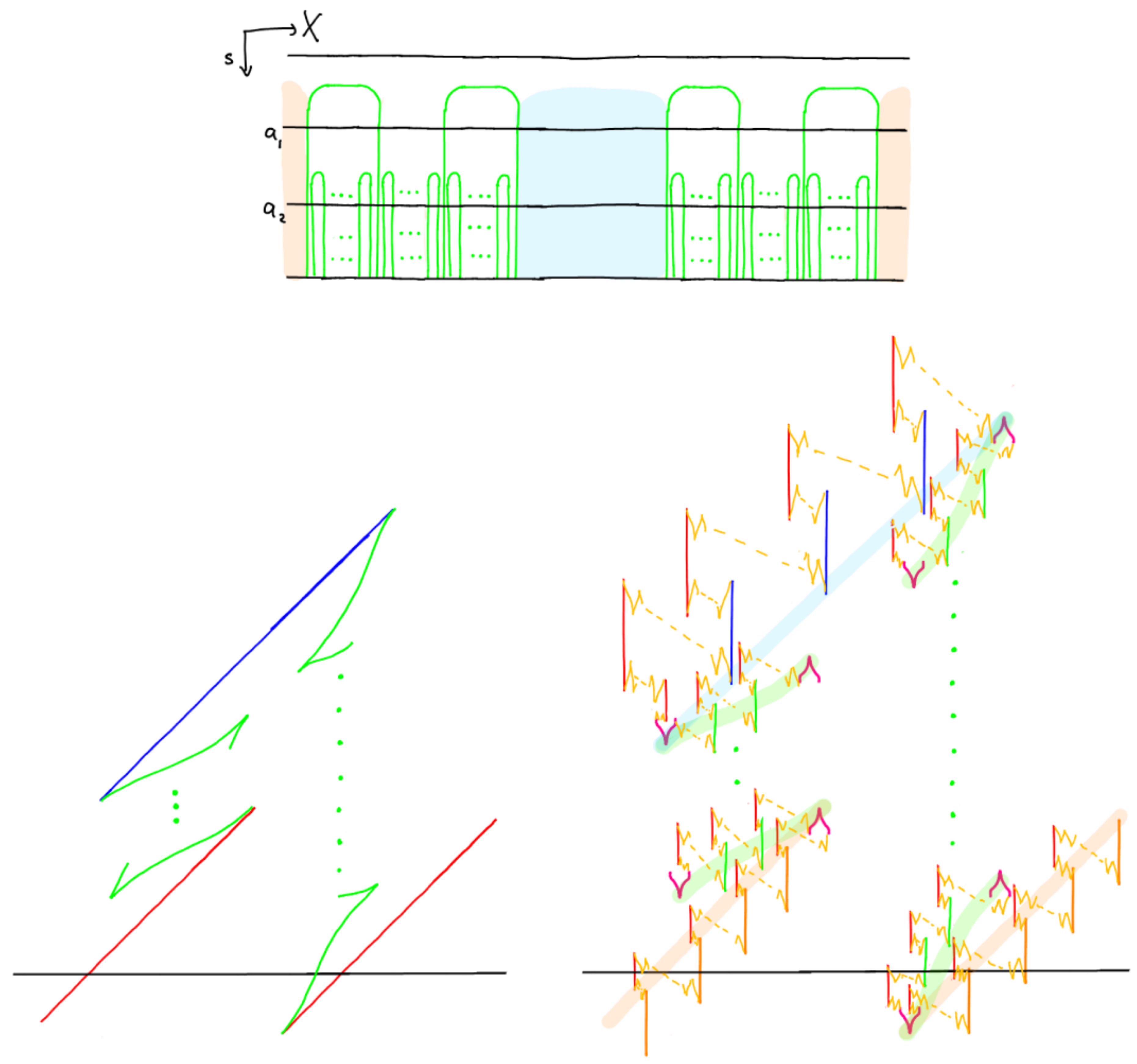}
\caption{The wrinkled submanifold version of Figures \ref{Fig:HolonomicApprox_Zigzag} and \ref{Fig:HolonomicApprox_ZigzagParam}. We are given a $1$-jet homotopy in which the zero jet (the horizontal line) stays fixed, but the Gauss map turns counterclockwise over time. The top picture represents the domain; the wrinkle singularity locus is shown in green. A time $a_1$ is shown in the bottom left; the formal Gauss map is still graphical and the wrinkled submanifolds constructed up to this point do not exhibit nesting. A later time $a_2$ is shown in the bottom right; the desired Gauss map is now vertical. This is graphical over the Gauss map at time $a_1$, but not graphical over the Gauss map we started with. It follows that new zig-zags have to be added within the membranes of the previous ones; this nesting is shown in the top picture as well.} \label{fig:wrinkledEmbeddings}
\end{figure}

\begin{remark}
Since the singularities of $f$ have cylinder base, the singularities of $f_s$, for most $s$, are zig-zags. At discrete times $s_i$, the map $f_{s_i}$ will also have spherical embryos.

Moreover, even though there may be nested singularities within the membrane of a given wrinkle, it is the case that all the zig-zags of a given $f_s$ come in pairs. Indeed, the membranes tell us what the pairings are. See Figure \ref{fig:wrinkledEmbeddings}.\hfill$\triangle$
\end{remark}

\subsubsection{Parametric statement}

We now consider a smooth compact manifold $K$ serving as parameter space. The result will deal with a fibered-over-$K$ family $X = (X_k)_{k \in K}$ of submanifolds of $Y$. As a family of abstract manifolds, this is simply a locally trivial fibre bundle over $K$. However, it need not be globally trivial.

\begin{theorem} \label{thm:wrinkledEmbeddingsParam}
Fix a smooth manifold $Y$, an integer $n$, a $K$-family of submanifolds $X = (X_k \subset Y)_{k \in K}$ and a family of $r$-jet homotopies 
\[ F = ((F_{k,s})_{s \in [0,1]}: X_k \to J^r(Y,n))_{k \in K}. \]

Then, there is a wrinkled family of singular embeddings
\[ (f=(f_{k,s}: X_k \to Y)_{k \in K, s \in [0,1]},\{A_i \subset [0,1] \times X\}) \]
satisfying:
\begin{itemize}
\item $|j^rf_{k,s} - F_{k,s}| < \varepsilon$.
\item $f_{k,0}: X_k \to Y$ is the inclusion.
\end{itemize}

Additionally: If $F$ is already holonomic in a neighbourhood of a polyhedron $L \subset [0,1] \times X$, the singularities of $f$ may be taken to be disjoint from $L$.
\end{theorem}

As in the case of multi-sections, we will prove:
\begin{corollary} \label{cor:wrinkledEmbeddingsCylinder}
The wrinkles of the family $f$ produced by Theorem \ref{thm:wrinkledEmbeddingsParam} may be assumed to have cylinder base.
\end{corollary}

The surgery methods from appendix \ref{sec:surgeries} will also show:
\begin{corollary} \label{cor:wrinkledEmbeddingsBall}
The wrinkles of the family $f$ produced by Theorem \ref{thm:wrinkledEmbeddingsParam} may be assumed to have ball base.
\end{corollary}

\subsection{Proof of the non-parametric case} \label{ssec:wrinkledEmbeddingsProof}

As stated above, the core of the proof consists in reducing to the case of sections. We first explain what a graphical $r$-jet homotopy is and how one subdivides the time interval (Subsection \ref{sssec:wrinkledEmbeddingsProof1}). This turns the proof into an induction argument on the number of time subdivisions (Subsection \ref{sssec:wrinkledEmbeddingsProof2}). In each step (except for the base case), we deal with the singularity locus $\Sigma$ first (Subsection \ref{sssec:wrinkledEmbeddingsProof4}) and then with the complement (Subsection \ref{sssec:wrinkledEmbeddingsProof5}). In order to work around $\Sigma$, we need to produce an auxiliary submanifold $\SS$ close to it to serve as a reference and play the role of the zero section (Subsection \ref{sssec:wrinkledEmbeddingsProof3}).

Fix a metric in $Y$ and a metric in $J^r(Y,n)$. During the proof we will provide size bounds using constants $\varepsilon_l>0$. These will depend only on $\varepsilon$, the $r$-jet homotopy $(F_s)_{s \in [0,1]}: X \to J^r(Y,n)$, and the chosen metrics. Instead of giving a precise value for them, we will just indicate what is the reasoning behind their choice.

\subsubsection{Chopping the time interval} \label{sssec:wrinkledEmbeddingsProof1}

Given $(F_s)_{s \in [0,1]}$, we can consider the associated homotopy of Gauss maps given by projection to the $1$-jet 
\[ (\pi_{r,1} \circ F_s)_{s \in [0,1]}: X \to J^1(Y,n). \]
Using the metric we obtain a homotopy of formal normal bundles $\nu(F_s) := (\pi_{r,1} \circ F_s)^\bot$.

We will say that $(F_s)_{s \in [a,b]}$ is a \textbf{graphical} homotopy if for all $s \in [a,b]$ it holds that $\pi_{r,1} \circ F_s$ is transverse to $\nu(F_s)$. This depends on the choice of metric.

We fix a collection of closed consecutive intervals $\{[a_j,a_{j+1}]\}$ covering $[0,1]$ such that each $(F_s)_{s \in [a_j,a_{j+1}]}$ is graphical. This is possible if the submanifold $X$ is compact. In the non-compact case one needs to carry out the upcoming arguments using a exhaustion by compacts, invoking the relative nature of the statement.

\subsubsection{Setup for induction} \label{sssec:wrinkledEmbeddingsProof2}

We will argue inductively on the number of time subintervals. The inductive hypothesis at step $j$ reads: There is a wrinkled family of singular embeddings 
\[ (g: [0,a_j] \times X \to [0,a_j] \times Y,\{D_i \subset [0,a_j] \times X\}) \]
satisfying
\begin{itemize}
\item $g$ is a holonomic approximation. I.e. $|j^rg_s - F_s| < \varepsilon_0$ for each $g_s = g|_{s \times X}$.
\item $g_0$ is the inclusion of $X$.
\item All the singularities of $g$ are zig-zag wrinkles with cylinder base $D_i$.
\item Each $\partial D_i$ is transverse to $\{a_j\} \times X$. The singularity locus $\Sigma$ of $g_{a_j}$ is a disjoint collection of $\NS^{n-1}$-stabilised cusps.
\end{itemize}
See Figure \ref{fig:wrinkledEmbeddings}.

The base case of the induction $j=0$ holds by letting $[0,a_j]$ be the degenerate interval $0$ and $g$ be the inclusion $X \to Y$. Then, there are no singularities present.

If $j \neq 0$, we shift our focus from $(F_s)_{[a_j,a_{j+1}]}$ to a homotopy:
\[ (G_s)_{[a_j,a_{j+1}]}: [0,1] \times X \to [0,1] \times J^r(Y,n) \]
with better properties. We ask that $G_s$ lifts $g_{a_j}$ and satisfies $|G_s-F_s| < \varepsilon_1$. Indeed, each $G_{s_0}$ may be constructed from $F_{s_0}$ by lifting the homotopy $(g_s)_{[0,a_j]}$ using parallel transport in $J^r(Y,n)$. Note that can estimate a priori the distortion introduced by parallel transport in terms of the length of the paths involved. We can then choose $\varepsilon_0$ so that $\varepsilon_1$ is sufficiently small (much smaller than $\varepsilon$ times the number of time intervals).

Moreover, we can homotope $(G_s)_{s \in \Op(a_j)}$ in order to impose $G_{a_j} = j^rg_{a_j}$. We can do this while keeping a bound of $|G_s-F_s| < 2\varepsilon_1$. We can henceforth focus on approximating $G_s$ instead of $F_s$. We are now effectively attempting to achieve holonomic approximation relative to a singular embedding instead of relative to an embedding. We leave it to the reader to restate Theorem \ref{thm:wrinkledEmbeddings} in such generality.

\subsubsection{The cusp locus} \label{sssec:wrinkledEmbeddingsProof3}

Our inductive assumption is that the singular locus $\Sigma$ of $g_{a_j}$ consists of a disjoint union of codimension-one spheres mapping to cusps. Note that $g_{a_j}|_\Sigma$ itself is a smooth embedding.

A zig-zag wrinkle with cylinder base is given by a model. Particularising this model to the zig-zags in the time slice $a_j \times X$, we deduce that there are:
\begin{itemize}
\item a smooth manifold $B \simeq (-1,1) \times \Sigma$,
\item an embedding of $B$ into $Y$,
\item lifting to a front symmetry embedding $J^r(\nu(B))$ into $J^r(Y,n)$.
\end{itemize}
such that $g|_\Sigma$ maps into $\nu(B)$, $j^rg_{a_j}|_\Sigma$ maps into $J^r(\nu(B))$, and $g: \Op(\Sigma) \to \nu(B)$ is described by a collection of multi-section cusp models.

It is then possible to take a thickening of the singular locus
\[\SS := \Op(\Sigma) \simeq (-1,1) \times \Sigma \subset X \]
and produce an embedding $\Psi:\SS \hookrightarrow Y$, such that $j^r\Psi = j^rg_{a_j}$ along $\Sigma$. Indeed: close to $\Sigma$, the lift $j^rg_{a_j}$ is a fold mapping into $J^r(\nu(B))$. It is the lift of a map $\psi$ into a principal projection. We can then take $\psi|_\Sigma$, extend it to an embedding of $\SS$ that is graphical over the base $B$, and then lift to a non-singular integral map using $j^rg_{a_j}|_\Sigma$ as initial datum. The front projection of this map is the claimed $\Psi$.

Lastly, we extend the embedding $\Psi:\SS \hookrightarrow Y$ to an embedding of the normal bundle $\nu(\SS)$. This provides us with an front symmetry embedding $J^r(\nu(\SS)) \to J^r(Y,n)$. In these new coordinates, $g_{a_j}$ is the $\NS^{n-1}$-stabilisation of a cusp tangent to the zero section along $\Sigma$.

\subsubsection{Holonomic approximation around the cusps} \label{sssec:wrinkledEmbeddingsProof4}

The goal now is to apply holonomic approximation to the submanifold $\SS$ and use this to achieve holonomic approximation for $g_{a_j}$ close to $\Sigma$. We work in $J^r(\nu(\SS))$.

We take the family $(G_s|_\Sigma)_{[a_j,a_{j+1}]}$ and extend it arbitrarily to $\SS$; denote the resulting homotopy by 
\[ (H_s: \SS \to J^r(Y,n))_{[a_j,a_{j+1}]}. \]
Our assumption on the graphicality of $(F_s)_{s \in [a_j,a_{j+1}]}$ translates now into the fact that $H_s$ takes values in $J^r(\nu(\SS))$. We can think of it as a homotopy of formal sections.

A further homotopy allows us to impose $H_{a_j} = 0$. By possibly shrinking $\SS$ we can ensure $|H_s-F_s| < 4\varepsilon_1$. This follows from the fact that $G_{a_j} = j^rg_{a_j}$ was already holonomic and $\Psi$ approximates $g_{a_j}|_\Sigma$. 

We now apply the classic holonomic approximation Theorem \ref{thm:holonomicApprox} to find:
\begin{itemize}
\item an isotopy $(\rho_s)_{s \in [a_j,a_{j+1}]}: \SS \to \SS$,
\item a homotopy of sections $(h_s)_{s \in [a_j,a_{j+1}]}: \SS \to J^r(\nu(\SS)) \subset J^r(Y,n)$,
\end{itemize}
satisfying:
\begin{itemize}
\item $h_{a_j}$ is the zero section.
\item $\rho_{a_j}$ is the identity.
\item Holonomic approximation $|j^rh_s - H_s|_{C^0} < \varepsilon_1$ holds on $\Op(\wtd{\Sigma}_s)$ where $\wtd{\Sigma}_s := \rho_s(\Sigma)$.
\end{itemize} 
This data defines a homotopy of point symmetries $(j^r\Lambda_s)_{s \in [a_j,a_{j+1}]}$ of $J^r(\nu(\SS))$. Namely, $\Lambda_s: \nu(\SS) \to \nu(\SS)$ is the unique lift of $\rho_s$ taking the zero section to $h_s$.

On a neighborhood of $\Sigma$ we can define $(g_s)|_{s \in [a_j,a_{j+1}]}$ to be given by $\Lambda_s(g_{a_j})$. We then observe that $j^rg_s = j^rh_s$ along $\Sigma$. It follows that, by making the neighbourhood is sufficiently thin and combining our previous bounds, we can obtain a bound $|j^rg_s - F_s| < 8\varepsilon_1$ for all $s \in [a_j,a_{j+1}]$ over $\Op(\Sigma)$.

\subsubsection{Conclusion of the inductive step} \label{sssec:wrinkledEmbeddingsProof5}

We now observe that $g_{a_j}$ is an honest embedding in the complement of $\Sigma$. We let $L$ be its image and $\nu(L)$ be the normal bundle. As above, we obtain a local point symmetry embedding $J^r(\nu(L))$ into $J^r(Y,n)$. By construction, $j^rg_{a_j}$ takes values in the image and is in fact the zero section. Furthermore, graphicality implies that $(G_s)_{s \in [a_j,a_{j+1}]}$ takes values in $J^r(\nu(L))$.

Since Theorem \ref{thm:holonomicApproxParam} is relative in the domain and the parameter, it can be applied to $(G_s)_{s \in [a_j,a_{j+1}]}$ relative to 
\[ \{a_j\} \times X \quad\bigcup\quad [a_j,a_{j+1}] \times \Op(\Sigma). \]
This yields a wrinkled family of singular embeddings $4\varepsilon_1$-approximating $G_s$. By construction, some wrinkles are nested in the membranes of the wrinkles we already had. 

A suitable choice of $\varepsilon_1$, roughly of size $\varepsilon$ over $20$ times the number of time intervals, completes the inductive step and thus the proof of Theorem \ref{thm:wrinkledEmbeddings}. \hfill$\Box$

\subsection{Proof of the parametric case} \label{ssec:wrinkledEmbeddingsProofParam}

The proof follows the strategy explained in the non-parametric setting, Subsection \ref{ssec:wrinkledEmbeddingsProof}. The only difference is that the singularities appearing now are slightly more complicated. Namely, the intersection of the singularity locus with a given time slice has now a non-empty equator.

In order to deal with this, we first produce a holonomic approximation along the equator, then along the cusps, and lastly along the smooth part\footnote{In arguments involving only $\Sigma^1$ singularities it is standard to proceed inductively from the worse singularities to the best, as we do here.}.

\subsubsection{Setup}

We work with a family of manifolds $X \to K$ whose fibre over $k \in K$ is denoted by $X_k$. They are presented to us as submanifolds of an ambient manifold $Y$. Furthermore, we are given a family of $r$-jet homotopies
\[ F = (F_s := (F_{k,s}: X_k \to J^r(Y,n))_{k \in K}))_{s \in [0,1]} \]
lifting the inclusions.

As before, we can produce a collection of intervals $\{[a_j,a_{j+1}]\}$ covering $[0,1]$ such that the family $F$ is graphical over each $[a_j,a_{j+1}]$, meaning that it is graphical for each individual $k \in K$. This relies on compactness of $X$; otherwise we argue using an exhaustion by compacts.

We consider then the following inductive statement at step $j$: A family 
\[ (g = (g_s = (g_{k,s})_{k \in K})_{s \in [0,a_j]}, \{B_i \subset [0,a_j] \times X\}) \]
 of holonomic approximations of $F$ has been produced, additionally satisfying:
\begin{itemize}
\item The singularities of a given $g_s$ are zig-zag wrinkles with cylinder base or embryos thereof.
\item Each $\partial B_i$ is transverse to $\{a_j\} \times X$. In particular, the singularities of $g_{a_j}$ are zig-zag wrinkles with cylinder base.
\end{itemize}
The base case $j=0$ holds by letting $g$ be the family of smooth inclusions $(X_k \to Y)_{k \in K}$.

We write $\Sigma^{1,0}$ for the cusp locus of $g_{a_j}$ and $\Sigma^{1,1}$ for the equator. It follows that each component of $\Sigma^{1,1}$ is a smooth submanifold diffeomorphic, in a fibered manner over $K$, to $\NS^{k-1} \times \NS^{n-1}$. Similarly, each component of $\Sigma^{1,0}$ is diffeomorphic to $\overset{\circ}{\D^k} \times \NS^{n-1}$.

We can moreover assume, using appropriate homotopies of the formal data, that instead of $(F_s)_{s \in [a_j,a_{j+1}]}$ we are interested in approximating 
\[ (G_s)_{s \in [a_j,a_{j+1}]}: X \to J^r(Y,n) \]
lifting $g_{a_j}$ and satisfying $G_{a_j} = j^rg_{a_j}$.

\subsubsection{Holonomic approximation around the equator}

As stated above, each component of $\Sigma^{1,1}$ is diffeomorphic to $\NS^{k-1} \times \NS^{n-1}$. Furthermore, since it is part of a wrinkle, its normal bundle is trivial. It follows that we can restrict the model around the wrinkle to the time slice $a_j \times X$ to yield:
\begin{itemize}
\item a smooth manifold 
\[ B \simeq (-1,1) \times \Sigma^{1,1} \times (-1,1) \simeq [(-1,1) \times \NS^{k-1}] \times [\NS^{n-1} \times (-1,1)], \]
where the first component is regarded as a parameter space,
\item which embeds into $K \times Y$ in a fibered manner,
\item and a extension to a front symmetry embedding $J^r(\nu(B))$ into $K \times J^r(Y,n)$, where $\nu(B)$ is the fibrewise normal bundle along the $Y$ component,
\end{itemize}
such that $j^rg_{a_j}|_{\Sigma^{1,1}}$ maps into $J^r(\nu(B))$.

It is then possible to take a thickening of $\Sigma^{1,1}$:
\[\SS := \Op(\Sigma^{1,1}) \simeq B \subset X \]
and produce a fibered embedding $\Psi: \SS \hookrightarrow K \times Y$, such that $j^r\Psi = j^rg_{a_j}$ along $\Sigma^{1,1}$. We do this as in Subsection \ref{sssec:wrinkledEmbeddingsProof3}. There is a fibered map $\psi$ into a principal projection whose lift to jet space yields $j^rg_{a_j}|_{\Op(\Sigma^{1,1})}$. We can then extend $\psi|_{\Sigma^{1,1}}$ to $\SS$ as a map graphical over $B$. Lifting to $J^r(\nu(B))$ and taking the front projection yields $\Psi$.

We can extend $(G_s|_{\Sigma^{1,1}})_{[a_j,a_{j+1}]}$ to a homotopy $(H_s)_{[a_j,a_{j+1}]}$ of formal sections over $B$ with values in $J^r(\nu(B))$. As such, we can apply classic holonomic approximation to follow $H_s$ in a holonomic manner in a neighbourhood of a wiggled version of $\Sigma^{1,1}$. This defines a homotopy of point symmetries that we apply to $g_{a_j}$ in a neighbourhood $\Sigma^{1,1}$. The resulting maps are precisely $(g_s|_{\Op(\Sigma^{1,1})})_{[a_j,a_{j+1}]}$.

\subsubsection{Conclusion of the inductive step}

The same reasoning can then be applied in a neighbourhood of $\Sigma^{1,0}$, relative to the process carried out close to $\Sigma^{1,1}$. This still relies on classic holonomic approximation. Lastly, we apply holonomic approximation for multi-sections (Theorem \ref{thm:holonomicApproxParam}) in the smooth locus, relative to $\Sigma^1$. This completes the inductive step and thus the proof of Theorem \ref{thm:wrinkledEmbeddingsParam}. \hfill$\Box$

%% file: PaperI-AppendixOdd.tex
%
%
%
%
%
%
\section{Desingularisation in odd jet spaces} \label{sec:desingularisationOdd}


The results in this section apply to jet spaces both of sections and of submanifolds. To keep the discussion grounded, we deal with sections first. The case of submanifolds is addressed in Subsection \ref{sssec:fishSurgerySubm}.

Consider a submersion $Y \longrightarrow X$ with $n$-dimensional base and $\ell$-dimensional fibres. Holonomic approximation for multi-sections, as stated in Theorem \ref{thm:holonomicApproxZZ}, produces topologically embedded multi-sections $X \to Y$ that lift to integral embeddings $X \to J^r(Y)$. Unfortunately, this is not true anymore in the parametric setting (Theorem \ref{thm:holonomicApproxParam}). More precisely, when $r$ is odd, the lift of a wrinkled family of multi-sections (Definition \ref{def:zigZagParam}) will have birth/death events modelled on a stabilisation-type event. When $r$ is even this issue does not arise since the Reidemeister I move is non-singular. 

In this section we explain a surgery procedure for zig-zag wrinkles (and hence for wrinkled families of multi-sections) that addresses this problem. The upshot is that, even in the presence of parameters, we can produce multi-sections whose integral lifts are smooth embeddings; see Figure \ref{fig:WrinkledSolutions_OddSurgery}. However, there are two caveats. First: in the contact case ($r=1$, $\ell=1$), the procedure does not apply. Secondly, for odd $r>1$ but $\ell=1$, the resulting multi-sections are not topological embeddings themselves (but their self-intersections are controlled and given by a concrete model).

The required singularity models are explained in Subsection \ref{ssec:fish}. The key construction behind them is provided in Subsection \ref{ssec:wrinkleObjects}. The main results of this section, including the surgery statement, are given in Subsection \ref{ssec:fishSurgery}. Their proofs are presented in Subsection \ref{ssec:fishSurgeryProof}.

\begin{figure}[ht]
\centering
\includegraphics[width = \linewidth ]{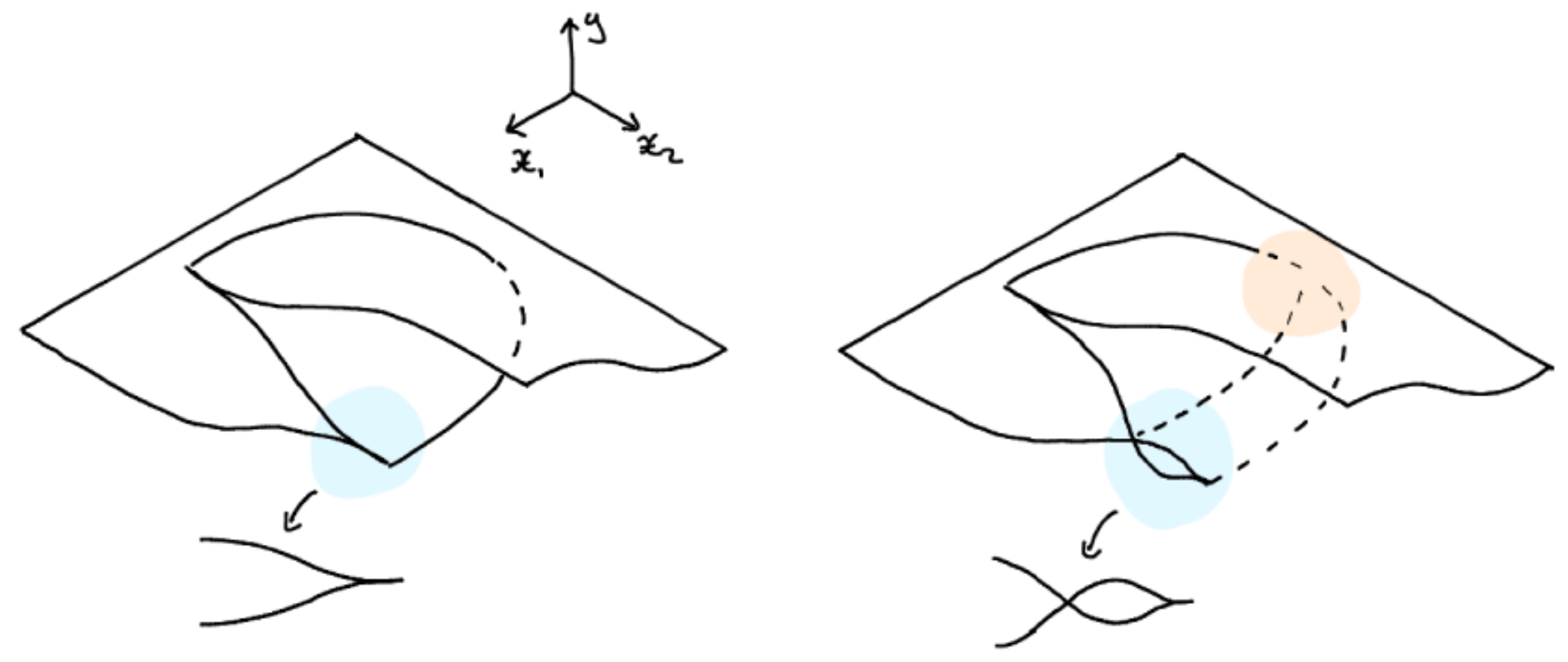}
\caption{On the left, a zig-zag wrinkle with $r$ odd. On the right, the fish wrinkle that replaces it thanks to the surgery procedure given in Proposition \ref{prop:fishSurgery}. Close to the equator, the swallowtail gets replaced by a more complicated birth/death phenomenon that we call the fish swallowtail; see Definition \ref{def:fishWrinkle}. We mark it in orange; it is drawn in detail in Figure \ref{fig:Surgery_FishSurgery}. It is based on the Reidemeister I move, yielding thus an embedded integral lift.}\label{fig:WrinkledSolutions_OddSurgery}
\end{figure}

\subsection{Wrinkle-type objects} \label{ssec:wrinkleObjects}

As explained in Subsection \ref{ssec:wrinkles}, the wrinkle is fibered in nature. We think of it as a self-cancelling configuration naturally arising from the Reidemeister I move. Similarly, the closed wrinkle can be thought as a self-cancelling family defined from the stabilisation. In this section we explain how a wrinkle-type object may be constructed from a given $1$-parameter family of integral curves in $J^r(\R,F)$. Here $F$ is a vector space locally modelling the fibres of $Y$.

\subsubsection{Setup}

We are given a pair $(f,A)$. Here
\[ f: \R^2 \to \R \times J^r(\R,F) \]
is an $\R$-family of integral curves. We use coordinates $(\ss,t)$ in the source and coordinates $(\ss,x)$ in the base of the target. The $\ss$-coordinate is the parameter variable. The fibre $F$ has coordinates $y$. These determine a set of standard coordinates $(x,y,z)$ in $J^r(\R,F)$.

$A \subset [-1,1]^2$ is a closed subset, playing the role of the membrane of $f$. Only the germ of $f$ along $A$ is relevant to us, and it will determine what the membrane of the corresponding wrinkle is. We assume that:
\begin{itemize}
\item $A$ is disjoint from $\{\ss < 0\}$.
\item $A$ contains the singularity locus of $f$.
\end{itemize}
We think of $f$ somewhat distinctly compared to previously introduced birth/death events (e.g. the Reidemeister I move or the stabilisation): Whereas the latter consist of a single embryo event in which the relevant singularity appears immediately, the former is thought of as a whole time interval providing a transition between a holonomic section (at $\ss < 0$) and the lift of a non-trivial multi-section (at $\ss=1$). This transition may include multiple inequivalent singularities.

In particular, the boundary of $A$ need not be the singularity locus of $f$, as was the case in our previous constructions.

\subsubsection{The construction}

Fix a compact manifold $K$ to serve as parameter space and an $(n-1)$-dimensional manifold $H$. We use coordinates $(k,\widetilde x,t)$ in $K \times H \times \R$ when regarded as the source manifold. When seen as the target of our maps, we use instead $(k,x) = (k,\widetilde x,x_n)$. The wrinkle will be fibered over $(k,\widetilde x)$.

Fix a full-dimensional submanifold $D \subset K \times H$ with boundary and let $C = (-1,0] \times \partial D$ be an inner collar of $\partial D$. We write $a$ for the radial coordinate in the collar. This allows us to define a function $\rho: \Op(D) \to [-1,1]$ that is identically $1$ in $D \setminus C$, zero on $\partial D$, negative outside of $D$, and has each level $\{a\} \times \partial D$ as a regular level set.

We can now consider the principal projection of $f$ with respect to the variable $x$. It is a fibered-over-$\R$ map:
\[ \psi: \R^2 \quad\longrightarrow\quad \R \times [\R \times \Sym^r(\R,F)]. \]
Applying the lifting Proposition \ref{prop:principalProjectionGeneral} to $\psi$, using $f|_{t=0}$ as initial datum, recovers $f$. We can then define an integral mapping:
\[ g: K \times H \times \R \quad\longrightarrow\quad K \times J^r(H \times \R,F) \]
as the integral lift of the map
\[  (k,\widetilde x,t) \quad\mapsto\quad (k,\widetilde x, \psi(\rho(k,\widetilde x),t),0) \in K \times [H \times \R \times \Sym^r(\R \times H,F)] \]
with initial datum $(f|_{t=0},0)$. We write $A' := \{(\rho(k,\widetilde x),t) \,\in\, A \}$. 
\begin{definition}
The \textbf{model $(f,A)$--wrinkle} with membrane $A'$, base $D$, and height $\rho$ is the germ of $g$ along $A'$.
\end{definition}
A concrete case of interest is when $K$ is just a point, $H$ is $\R$, and $D$ is $\R^\geq$. The resulting $(f,A)$-wrinkle is then said to be an $(f,A)$-pleat.

\begin{remark}
If we let $f$ be the Reidemeister I move and $A$ be its membrane, the resulting $(f,A)$-wrinkle is a model wrinkle. Keeping the same $f$ but letting $A$ be the singularity locus we obtain a wrinkle, allowing therefore for nesting.

Similarly, if $f$ is the stabilisation and $A$ is the membrane, an $(f,A)$-wrinkle is a model closed wrinkle. If we take $A$ to be the singularity locus, we obtain a closed wrinkle. \hfill$\triangle$
\end{remark}

\subsection{The fish zig-zag} \label{ssec:fish}

Fix a pair $(r,\ell)$ with $r$ odd and at least one of them different from $1$.

In this subsection we provide two pairs $(f,A_m)$ and $(f,A_n)$ as in Subsection \ref{ssec:wrinkleObjects}. The map $f: \R^2 \to \R \times J^r(\R,F)$ is the same for both. It is obtained from the standard Reidemeister I move and is meant to interpolate between a holonomic section and the lift of a zig-zag with a loop close to its right cusp. We call this multi-section the fish zig-zag. See the last image in Figure \ref{fig:fishL1}.

The loci $A_m \subset [-1,1]^2$ and $A_n \subset [-1,1]^2$ are different. The former has a single boundary component, bounding a ``full membrane''. The resulting $(f,A_m)$-wrinkle model is supposed to be a ``model'' wrinkle object. In contrast, the set $A_n$ has both an outer boundary and an inner boundary; we think of $A_n$ as a neighbourhood of the singularity locus. It follows that $(f,A_n)$-wrinkles can be nested.

\subsubsection{Constructing the model for $r>1$ and $\ell=1$} \label{sssec:fishL1}

We begin with a Reidemeister I move $(g_\ss)_{\ss \in \R}$ with membrane $M = \{t^2 \leq \ss\}$. It is given as the lift of the map 
\[ \phi_\ss(t) = (x(t)=t^3/3-t\ss, z_1^{(r)}(t) = t) \in \R \times \Sym^r(\R,\R). \]
Since both entries are odd on $t$, the front projection
\[ \pi_f \circ g_\ss(t) = (x(t)=t^3/3-t\ss, y(t)). \]
has $y(t)$ even. In particular, the points $t=\pm\sqrt{3\ss}$ are mapped to a self-intersection of the front. The front is otherwise topologically embedded. The desired $A_m \subset \R^2$ will be $\{t^2 \leq 3s\} \cap [-1,1]^2$.

Let $\chi: \R \to (-\infty,1/4]$ be a smooth map that is the identity over $(-\infty,1/8]$ and restricts to an orientation-preserving diffeomorphism $[-\infty,1) \to [-\infty,1/4)$. We define a new family $(h_\ss)_{\ss \in \R}$ by setting $h_\ss = g_{\chi(\ss)}$. It follows that each $h_\ss$, $\ss \in [1/4,\infty)$, is a double fold that does not depend on $\ss$. Furthermore, the self-intersections of $(\pi_f \circ h_\ss)_{\ss \in \R}$ take place within $A_m$.

\begin{figure}[ht]
\centering
\includegraphics[width = \linewidth]{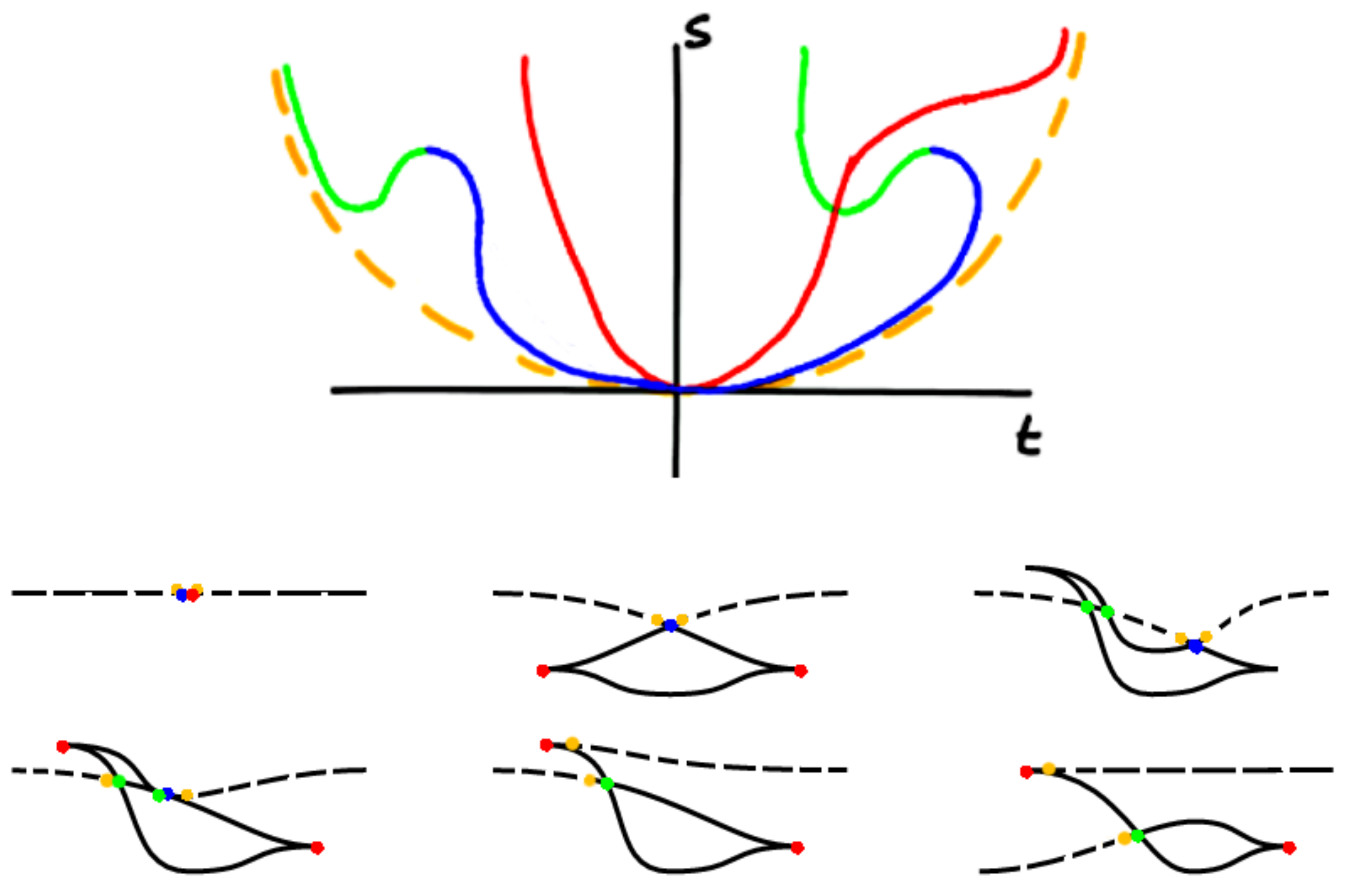}
\caption{The $1$-parametric family of integral curves used in Subsection \ref{sssec:fishL1}, seen in the front projection. The colour-coding identifies the singular loci in the domain and the corresponding singularities in the target. In blue, the front self-intersection appearing at the Reidemeister I event. In red, the two cusps born with the Reidemeister I. In green, the front self-intersections appearing later in the homotopy. The dotted orange line represents the boundary of the membrane $A_m$. For each instant, the image of the membrane is depicted as a solid black line and the complement of the membrane as dotted black.}\label{fig:fishL1}
\end{figure}

We now use Figure \ref{fig:fishL1} to replace $(\pi_f \circ h_\ss)_{\ss \in [1/2,1]}$. The first picture corresponds to $\ss = 1/2$ and the last one to $\ss \in \Op(1)$. This homotopy is meant to push the left-most cusp across the upper branch, as shown. It is possible to define it as the result of applying a fibrewise isotopy to the piece $\pi_f \circ h_\ss|_{t \in \Op([0,\sqrt{3\ss}])}$, while keeping the rest of the curve fixed. In particular, we do not change the singularity type of the left-most cusps.

Two properties of the homotopy are crucial: First, the self-intersections of the front are all transverse, with the exception of a single time $\ss_0$ (depicted in the 4th picture) in which a tangency appears. Secondly, the curvatures of the two branches meeting at this tangency point are different. 

\begin{remark}
As long as $r>1$, any generic family interpolating between the second and last pictures in Figure \ref{fig:fishL1} will have transverse intersections except for a finite collection of times in which a tangency will appear. Generically, the two curvatures at such a tangency will be distinct. For our purposes, any homotopy satisfying these properties would work equally well. \hfill$\triangle$
\end{remark}

One last reparametrisation of $t$, depending on $\ss \in \Op(1)$, allows us to assume that the point $(\ss,\sqrt{3\ss})$ maps to the right-most cusp appearing in the last picture and that the point $(\ss,-\sqrt{3\ss})$ is the first point mapping to the self-intersection. The result of this reparametrisation is the desired $f$.

We let $A_n$ be $A_m \setminus O$, where $O$ is a small ball, centered at the point $(1,0)$ and disjoint from the singularity locus.

\subsubsection{Constructing the model for $\ell>1$} \label{sssec:fishLnot1}

We now use the extra fiber dimensions to avoid self-intersections. Our starting point is a deformation of the principal projection of the Reidemeister I move:
\[ \phi_\ss(t) = (x(t)=t^3/3-t\ss ,z_1^{(r)}(t) = t, z_2^{(r)}(t),0,\cdots,0). \]
We denote its integral lift by $(g_\ss)_{\ss \in [-1,1]}$. A generic choice of $z_2^{(r)}$ will remove the front self-intersections along $\{\ss=t^2/3\}$ that are otherwise present. The same is true if we pick $z_2^{(r)}$ identically zero for $t \geq 0$ and strictly positive for $t$ negative.

The rest of the construction follows the steps given above with minor adjustments. The partial front projection $(x \circ f, y_1 \circ f)$ will be exactly as in the previous case but the $y_2$-coordinate will separate the self-intersections seen in $(x,y_1)$.

Namely: We use $\chi: \R \to (-\infty,1/4]$ to reparametrise in $\ss$ and yield $(h_\ss)_{\ss \in \R}$, as in the previous case. We then consider the homotopy depicted in Figure \ref{fig:fishLnot1}. It is given to us in the partial front projection $(x,y_1)$ but we can take derivatives in order to obtain a map into the principal coordinates $(x,z^{(r)}_1)$. Together with our previous choice for $z_2^{(r)}$, this yields a map into the principal projection. It can be lifted to a homotopy of integral curves that replaces $(h_\ss)_{\ss \in [1/2,1]}$. Their fronts have no self-intersections by construction. We then reparametrise in $t$, in an $\ss$-dependent manner, so that $(\ss,\sqrt{3\ss})$ maps to the right-most cusp if $\ss \in \Op(1)$. Similarly, we make $(\ss,-\sqrt{3\ss})$ be the first point mapping to the $(x,y_1)$-self-intersection.

There is now an additional step compared to the previous subsection. We reparametrise in $\ss$ by setting $f(\ss,t) = h_{\ss-\delta}(t)$, where $\delta>0$ is a small constant. If $\delta$ is sufficiently small, it follows that 
\[ x \circ f(1,\sqrt{3}) < x \circ f(1,-\sqrt{3}). \]
As before, we set $A_m = \{t^2 \leq 3\ss\} \cap [-1,1]^2$ and $A_n = A_m \setminus O$, where $O$ is a small ball, centered at the point $(1,0)$. 

Lastly, we impose $y_2 \circ f|_{\Op(\partial A_n)} = 0$ using a point symmetry supported close to $\partial A_n$. No self-intersections of the front will be introduced as long as $(x \circ f, y_1 \circ f)|_{\Op(\partial A_n)}$ is an embedding graphical over $x$. This is the case as long as $\delta$ and $O$ are sufficiently small. See Figure \ref{fig:fishLnot1}.

\begin{figure}[ht]
\centering
\includegraphics[width = \linewidth ]{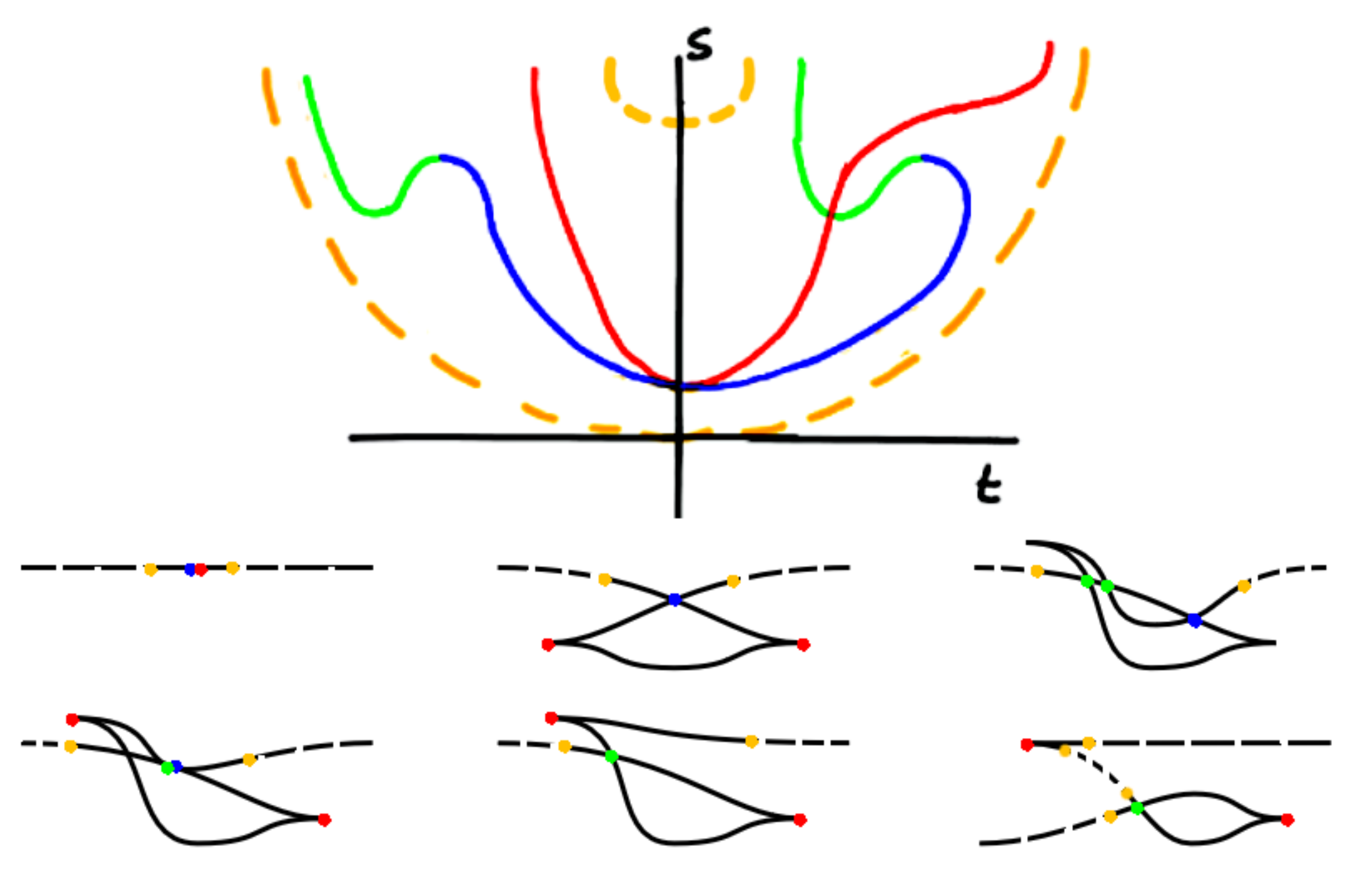}
\caption{The $1$-parametric family of multi-sections used in Subsection \ref{sssec:fishLnot1}. The dotted orange line represents the two boundary components of the membrane $A_n$. It is disjoint from the other curves, which represent various other singularities. The singularities in blue and green are not actual front singularities; they are self-intersections of the partial front associated to the $y_1$ direction.}\label{fig:fishLnot1}
\end{figure}

\subsubsection{The fish wrinkle}

We now associate to $(f,A_m)$ and $(f,A_n)$ various maps with values in the front projection. One could name the integral maps that lift them as well, but this is unnecessary for our purposes.
\begin{definition}\label{def:fishWrinkle}
The \textbf{fish swallowtail move} is the germ of $\pi_f \circ f$ along $A_m$.

The \textbf{model fish wrinkle} is the front projection of the model $(f,A_m)$-wrinkle.

The \textbf{fish wrinkle} is the front projection of the model $(f,A_n)$-wrinkle.

The \textbf{fish swallowtail} is the front projection of the $(f,A_m)$-pleat.
\end{definition}
Do note that these singularity models depend on our choice of height function $\rho$, which in turn depends on our choice of collar along the boundary of the base $D$ of the wrinkle. We leave these implicit unless it is important to our arguments.


\begin{lemma}
The singularities presented in Definition \ref{def:fishWrinkle} are (families of) multi-sections. They additionally satisfy:
\begin{itemize}
\item Their lifts are (families of) integral embeddings.
\item If $\ell \neq 1$, they themselves are topological embeddings.
\end{itemize}
\end{lemma}
\begin{proof}
Recall that $f$ was constructed as the integral lift of an immersion with values in the principal projection and with discrete singularities of tangency. It follows that the given singularities are multi-sections lifting to integral immersions. Since $f$ itself was embedded, the first claim holds. The second claim follows from the analogous property for $\pi_f \circ f$.
\end{proof}

\subsection{The surgery result and its corollaries} \label{ssec:fishSurgery}

The following surgery procedure allows us to replace a wrinkled family of multi-sections by a family with fish wrinkles. We let $K$ be the parameter space.
\begin{proposition} \label{prop:fishSurgery}
Suppose $r$ is odd. Fix a constant $\varepsilon > 0$. Let $g$ be the model zig-zag wrinkle with base $D \subset K \times H$, height $\rho$, and membrane $A$.

Write $U_\varepsilon$ for the $\varepsilon$-neighbourhood of $\partial A$. Let $\partial^+ A$ and $\partial^- A$ be its outer and inner boundaries, respectively.

Then, there is a multi-section $h$, defined as a germ along $U_\varepsilon$, such that:
\begin{itemize}
\item $h = g$ in $\Op(\partial^\pm A)$.
\item $|j^rh - j^rg| < \varepsilon$.
\item All of the singularities of $h$ are contained in a subset $B \subset U_\varepsilon$.
\item $h|_B$ is a fish wrinkle with base $D$.
\end{itemize}
\end{proposition}
We provide a proof in Subsection \ref{ssec:fishSurgeryProof}.

\subsubsection{Holonomic approximation for multi-sections} \label{sssec:fishSurgerySec}

The first corollary of Proposition \ref{prop:fishSurgery} is that it is possible to approximate families of formal sections of $J^r(Y)$ by families of multi-sections whose $r$-th order lifts are embedded integral submanifolds:
\begin{theorem} \label{thm:fishSurgerySec}
Fix a constant $\varepsilon > 0$. Let 
\[ (\sigma_k)_{k \in K}: X \longrightarrow J^r(Y) \]
be a $K$-family of formal sections. Then, there exists a family of multi-sections $(f_k)_{k \in K}: X \to Y$ satisfying:
\begin{itemize}
\item the lifts $j^rf_k$ are integral embeddings.
\item If $\ell>1$ or $r$ is even, the multi-sections themselves are topological embeddings.
\item $|j^rf_k - \sigma_k|_{C^0} < \varepsilon$.
\end{itemize}

Moreover, if the family $(\sigma_k)_{k \in K}$ is holonomic on a neighborhood of a polyhedron $L \subset K \times M$, then we can arrange $j^rf = \sigma$ over $\Op(L)$.
\end{theorem}
\begin{proof}
When $r$ is even this follows immediately from Theorem \ref{thm:holonomicApproxParam}.

If $r$ is odd, we let $g$ be the family of multi-sections resulting from Theorem \ref{thm:holonomicApproxParam}. We then apply Proposition \ref{prop:fishSurgery} to each model wrinkle $g|_A$. This yields a fish wrinkle $h$ that agrees with $g$ slightly outside of $A$ and can therefore be used to replace $g$ in $\Op(A)$. Denote the resulting map by $\wtd g$.

Since $j^rh|_{\Op(A)}$ approximates $j^rg|_{\Op(A)}$ and $h$ is a fish wrinkle, there are no self-intersections of $j^r\wtd g$ involving the region $\Op(A)$. The same statement holds for $\wtd g$ itself if $\ell>1$. The three claimed properties hold once all the wrinkles have been surgered.
\end{proof}

\subsubsection{Holonomic approximation for singular submanifolds} \label{sssec:fishSurgerySubm}

The second corollary of Proposition \ref{prop:fishSurgery} reads:
\begin{theorem} \label{thm:fishSurgerySubm}
Suppose $r$ is even or $\ell>1$. Let
\[ F = ((F_{k,s})_{s \in [0,1]}: X_k \to J^r(Y,n))_{k \in K}. \]
be a $K$-family of $r$-jet homotopies. 

Then, there is a family of singular embeddings
\[ (f=(f_{k,s})_{k \in K, s \in [0,1]},\{A_i \subset [0,1] \times X\}) \]
satisfying:
\begin{itemize}
\item $f_{k,0}: X_k \to Y$ is the inclusion.
\item the lifts $j^rf_{k,s}$ are integral embeddings.
\item $|j^rf_{k,s} - F_{k,s}| < \varepsilon$.
\end{itemize}

Additionally: If $F$ is already holonomic in a neighbourhood of a polyhedron $L \subset [0,1] \times X$, the singularities of $f$ may be taken to be disjoint from $L$.
\end{theorem}
\begin{proof}
We apply Theorem \ref{thm:wrinkledEmbeddingsParam}. This solves the case of $r$ even. If $r$ is odd, we obtain a wrinkled family of embeddings $g$. We apply Proposition \ref{prop:fishSurgery} to each of its wrinkles $A$. This is just like in the multi-sections case: Since the surgery takes place in $\Op(\partial A)$, the process is unaffected by nesting.
\end{proof}
Observe that the analogous construction for $\ell=1$ and $r>1$ odd does not yield a family of singular embeddings due to the presence of self-intersections.

\subsection{Proof of Proposition \ref{prop:fishSurgery}} \label{ssec:fishSurgeryProof}

We let $0 < \delta, \nu, \eta < \varepsilon$ be small constants, to be fixed during the proof. We write $E$ for the equator of $g$ and $V_a$ for the $a$-neighbourhood of $E$.

\subsubsection{Introduce the fish wrinkle}

Write $C$ for the inner $\delta$-collar $(-1,0] \times \partial D$ of $\partial D$. It holds that $C \times \{0\} \subset V_{\delta}$. Pick an arbitrary height function $\rho$ adapted to $C$ and let $h'$ be the corresponding fish wrinkle.

Even though $h'$ is defined using a full $A_m$-membrane, we single out its $A_n$-membrane, which we denote by $B \subset K \times H \times \R$. We can reparametrise $K \times H \times \R$ in a fibered manner over $K \times H$ so that $B$ is contained in $U_\delta$.


Furthermore, a fibrewise compression in the $y_1$-coordinate (lifting to a point symmetry) allows us to assume that $j^rh'$ takes values in a $\nu$-neighbourhood of the zero section.

\subsubsection{Automatic matching at the equator}

By definition, $g(k,\widetilde x,0) = (k,\widetilde x,x_n=0,0)$, where the last zero denotes the zero section in jet space. Since $h'$ was defined using the Reidemeister move as a starting point, it similarly follows that $h'(k,\widetilde x,0) = (k,\widetilde x,x_n=0,0)$ on a neighbourhood of $E$.

Then, if $\delta$ is sufficiently small, we can assume that $j^rg|_{V_\delta}$ takes values in a $\eta$-neighbourhood of $j^rg|_E = 0$. Further, if $\nu$ is sufficiently small, we can also assume that $j^rh'|_{W \cap V_\delta}$ takes values in such a $\eta$-neighbourhood. It follows that $j^rh'$ and $j^rg$ are $2\eta$-close over $W \cap V_\delta$.

\begin{figure}[ht]
\centering
\includegraphics[width = \linewidth ]{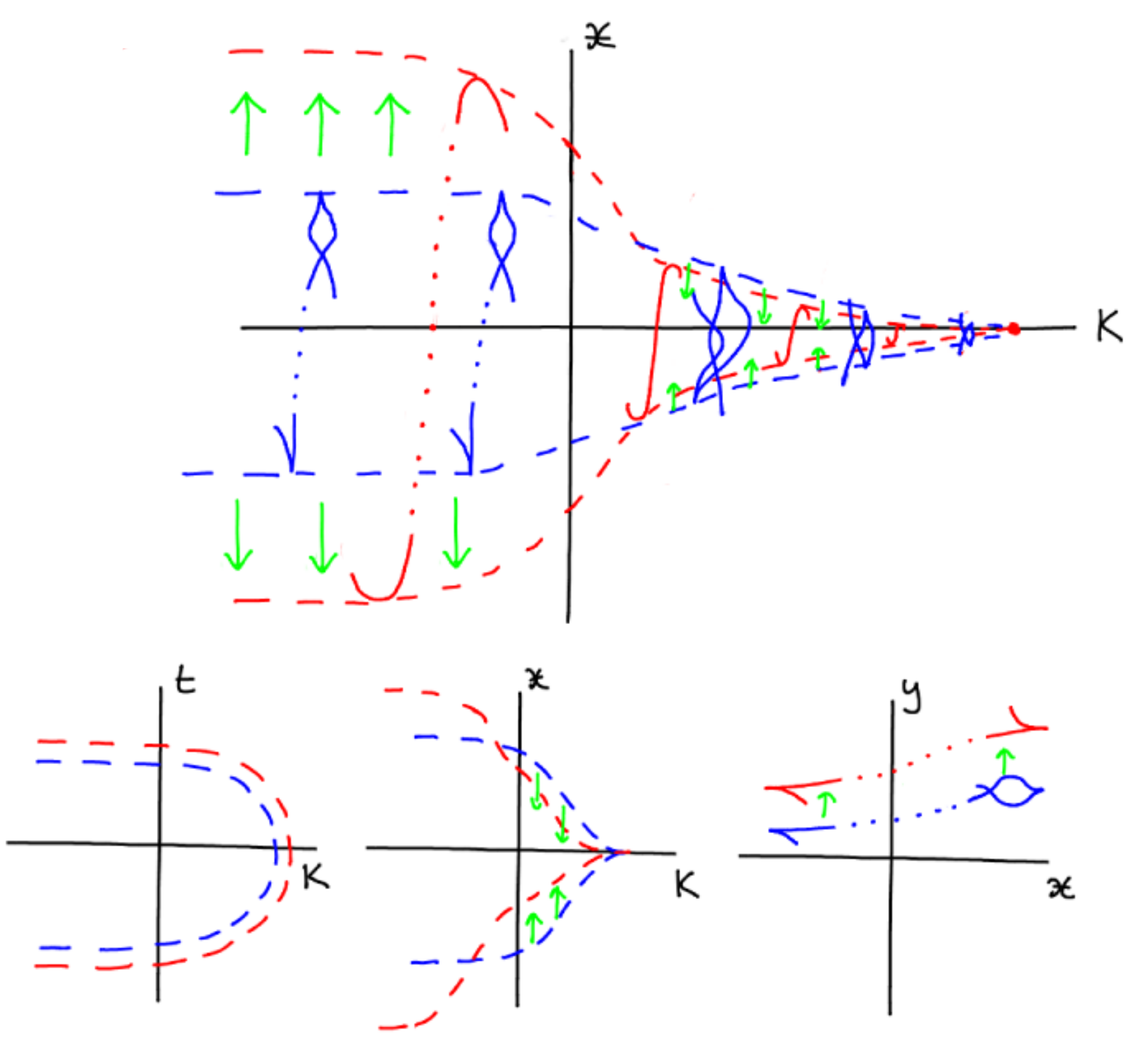}
\caption{Red represents the zig-zag wrinkle $g$ and blue the iterations of fish wrinkles that eventually yield $h$. Bottom right: $\partial D$ and its inner $\delta$-collar $C$; choosing a small $\delta$ ensures that the birth/death region associated to $h$ is thin and thus close to the equator of $g$. Middle (both bottom and top): we match the $x_n$-coordinates of $h$ and $g$, by adjusting the former away from the equator $E$. Left: we match $h$ to $g$ close to the cusp locus. }\label{fig:Surgery_FishSurgery}
\end{figure} 

\subsubsection{Adjusting at the cusp locus}

The next two steps in the upcoming proof use ideas from the proof of Theorem \ref{thm:wrinkledEmbeddings}.

We first approximately match the $x_n$-coordinates of $h'$ and $g$. If $\delta$ is sufficiently small, we can assume that $j^rg|_{U_\delta}$ takes values in a $\eta$-neighbourhood of $j^rg|_{\partial A} = 0$. We can then apply a point symmetry to $h'$ so that $x_n \circ h'|_W$ is $\eta$-close to $x_n \circ g_{\partial A}$. This point symmetry can be obtained from a suitable reparametrisation of $x_n$ in the region $U_\delta \setminus V_{\delta/2}$. We write $h''$ for the resulting map. It may be assumed to agree with $h'$ over $V_{\delta/2}$.

Now we match $h''$ and $g$, as multi-sections. We choose an auxiliary section 
\[ G: U_\delta \to K \times J^0(H \times \R,\R) \]
such that $j^rG$ and $j^rg$ are $\eta$-close on $U_\delta$. It can moreover be assumed that they agree over $\partial A \setminus V_\delta$ and that $G$ is the zero section in $V_{\delta/2}$. There is a fibered-over-$K \times H \times \R$ diffeomorphism $\Psi$ of $K \times J^0(H \times \R,\R)$, lifting the identity, taking the zero section to $G$. We can then set $h''' = j^r\Psi(h'')$. It follows that $j^rh'''$ and $j^rg$ are $4\eta$-close over $W$.

\subsubsection{Making the construction relative to the boundary of the membrane}

We now observe that $h''''$ is graphical in $\Op(B) \setminus B$. Similarly, $g$ is graphical in $U_\varepsilon \setminus \partial A$. By construction, the image of $h'''|_B$ is contained in a $5\eta$-neighbourhood of $g|_{\partial A}$. It follows that, if $\eta$ is sufficiently small it is possible to interpolate, over $U_\varepsilon \setminus B$, between $h$ and $g$, as we go from $\partial B$ to $\partial U_\varepsilon$. Write $h$ for the resulting multi-section defined over $U_\delta$.

We observe that the distance from one boundary component of the band $U_\varepsilon \setminus B$ to the other can be assumed to be bounded below by $\varepsilon/2$ by taking $\delta < \varepsilon/2$. It follows that a sufficiently small choice of $\eta$ guarantees the $\varepsilon$-closeness of $j^rh$ and $j^rg$.

I.e. we choose $\eta$ first, depending on $\varepsilon$. The constants $\delta$ and $\nu$ are chosen next, depending on $\eta$ (and indirectly on $\varepsilon$).

\subsubsection{Checking the claimed properties}

The argument presented is a sequence of point symmetries, so the resulting map is point equivalent over $B$ to $h'$, which was a fish wrinkle, and is graphical elsewhere. For the same reason, the base of all these maps is still $D$. The approximation property follows by construction, concluding the proof. \hfill$\Box$

\begin{remark}
The proof shows that the surgery procedure may be assumed to be localised to an arbitrarily small neighbourhood of the closed lower hemisphere of $\partial A$. \hfill$\triangle$
\end{remark}


.

%% file: PaperI-AppendixWrinkling.tex
%
%
%
%
%
%
\section{Other surgeries involving wrinkles} \label{sec:surgeries}

In this section we explain (Subsections \ref{ssec:chopping} and \ref{ssec:wrinklesBallBase}) how to pass from arbitrary zig-zag wrinkles to zig-zag wrinkles with ball base. We then show (Subsection \ref{ssec:wrinklesCylinderBase}) that a zig-zag wrinkle with ball base can be replaced by a zig-zag wrinkle with cylinder base. The cobordisms associated to these surgeries are in fact wrinkled families themselves.

Moreover, (the lifts of) these surgeries may be assumed to take place in an arbitrarily small neighbourhood of jet space. This implies that our holonomic approximation Theorems \ref{thm:holonomicApproxZZ} and \ref{thm:holonomicApproxParam} also hold if we restrict our attention to wrinkled families all whose wrinkles have ball base. This was stated already in Section \ref{sec:holonomicApproxParam} as Corollary \ref{cor:holonomicApproxBall}.

The ideas presented in this appendix are an adaptation to our setting of the usual ``chopping of wrinkles''; see for instance \cite[Section 2]{ElMiWrinI}.

\subsection{Removable zig-zag wrinkles} \label{ssec:removable}


Before we get to chopping, we need to provide some background on how a wrinkle may be removed.

\subsubsection{Formal removal}

As explained in Subsections \ref{sssec:regularisation} and \ref{sssec:regularisationWrin}, the wrinkle and the double fold are singularites of tangency that are formally inessential. I.e. the regularisation homotopes their Gauss maps to monomorphisms transverse to $V_\can$, relative to the boundary of the models. Similar statements hold for closed double folds and wrinkles, but in that case there is no homotopy between the Gauss map and the regularisation.

Even if there is no formal obstruction, it may not be possible to remove a wrinkle due to geometric reasons. It is easy to see that this is the case for the $1$-dimensional zig-zag: Since the model is only defined as a germ along the membrane, Bolzano's theorem tells us that we cannot replace it by a holonomic section. This motivates the upcoming discussion.

\subsubsection{Geometric removal}

Let $K$, $H$, $D$ and $\rho$ be as in Subsections \ref{ssec:wrinkles}, \ref{ssec:closedWrinkles} and \ref{ssec:zigzagParam}. If $r$ is even, let
\begin{align*}
f: K \times H \times \R \quad\longrightarrow\quad & K \times H \times \R \times \Sym^r(\R,\R)
\end{align*}
be the map defined in Subsection \ref{ssec:wrinkles}; the germ of $\Lift(f,0)$ along $A = \{t^2 + \rho(k,\widetilde x) \leq 0\}$ is the model wrinkle. If $r$ is odd, let $f$ be as in Subsection \ref{ssec:closedWrinkles} instead; the germ of $\Lift(f,0)$ along $A$ is then the model closed wrinkle. In both cases, the germ of $\pi_f \circ \Lift(f,0)$ along $A$ is the model zig-zag wrinkle.

\begin{definition}\label{def:removableZigzag}
A \textbf{removable zig-zag wrinkle} with base $D \subset K \times H$ and height $\rho$ is a pair $(g,U)$ consisting of:
\begin{itemize}
\item the germ $g$ of $\pi_f \circ \Lift(f,0)$ along
\[ B = \{t^2 + 3\rho(k,\widetilde x) \leq 0\}. \]
\item a set $U \subset K \times J^0(H \times \R,\R)$ containing the fibrewise convex hull of the image of $g$.
\end{itemize}
\end{definition}
The raison d'etre for this definition is that $g: \Op(B) \to \Op(U)$ is homotopic, relative to $\partial(\Op(B))$ and in a fibered manner over $K \times H$, to a graphical multi-section. See the right-hand side of Figures \ref{fig:Singularities_Regularization} and \ref{fig:Singularities_ClosedRegularization}. This homotopy may be assumed to be a wrinkled family of multi-sections itself. Furthermore, if $r$ is even, this homotopy lifts to $r$-jet space to a homotopy of integral embeddings. 

\begin{definition}
Let $h: K \times M \to K \times Y$ be a $K$-family of multi-sections. A zig-zag wrinkle $h|_A$, $A \subset K \times M$, is said to be \textbf{removable} if there are sets $B \supset A$ and $U \subset K \times Y$ such that:
\begin{itemize}
\item $(h|_B,U)$ is a removable zig-zag wrinkle.
\item $h^{-1}(U) = B$.
\end{itemize}
\end{definition}
The second property says that $h|_B$ takes values in $U$ and no other branches of $h$ enter $U$. This fact, together with the discussion above, proves that:
\begin{lemma} \label{lem:removable}
Let $h: K \times M \to K \times Y$ be a wrinkled $K$-family of multi-sections with a removable zig-zag wrinkle $(h|_A,U)$.

Then, there is a $K$-family of multi-sections $h'$ satisfying:
\begin{itemize}
\item $h' = h$ outside of $\Op(A)$.
\item $h'|_{\Op(A)}$ is graphical with image contained in $U$.
\item $h'|_{\Op(A)}$ and $h|_{\Op(A)}$ can be connected by a compactly-supported, wrinkled $(K \times [0,1])$-family of multi-sections.
\end{itemize}
\end{lemma}

\begin{remark}
Even though the homotopy between a removable zig-zag wrinkle $(g,U)$ and a section is localised to $\Op(U) \subset K \times Y$, it is not small upon lifting to $r$-jet space. This is problematic if one is interested in making statements in the vein of holonomic approximation. We will address this issue in the next subsections.

The reader should compare Definition \ref{def:removableZigzag} to the notion of \emph{loose chart} for a legendrian \cite{Mur}. Indeed, our definition is an $r$-jet analogue of Murphy's, but there is a caveat. Whereas a loose chart tells us that a legendrian has a closed double fold with a ``large neighbourhood'' in $J^1$, removability asks for a ``large neighbourhood'' in $J^0$. The latter is a stronger condition, but this generality is sufficient for our statements.

We will revisit these ideas in the sequel \cite{PT2} and we will provide a better analogue of the loose chart in higher jet spaces. \hfill$\triangle$
\end{remark}

\subsection{Chopping} \label{ssec:chopping}

The idea behind chopping is that one can replace any zig-zag wrinkle by a collection of smaller wrinkles with ball base that are stacked upon each other. 

\subsubsection{The toy model}

Before we provide a complete statement, let us explain the construction in its most elementary form. Consider the multi-section
\[ g: \R^2 \to J^0(\R^2,\R),\quad (x_1,t) \mapsto (x_1,t^2,t^3),\]
obtained as the front projection of the fold. Its singularity locus equals $\Sigma(g) = \{t = 0\}$ and consists of cusps.

Let $\wtd{\Sigma} \subset \R^2$ be the graph of a compactly-supported function $x_1 \mapsto t(x_1) \leq 0$. Denote by $U$ the region sitting between $\Sigma$ and $\wtd{\Sigma}$. We ask ourselves whether there is a multi-section $f$ that satisfies:
\begin{itemize}
    \item The lifts $j^1f$ and $j^1g$ are $C^0$-zero close.
    \item $f$ has a cusp along $\wtd{\Sigma}$.
		\item $f$ has a zig-zag wrinkle in $\Op(U)$.
		\item $f$ has no other singularities.
\end{itemize}
That is, we introduce an auxiliary zig-zag wrinkle in order to move the cusp locus of $g$. See Figure \ref{fig:ModifyingSingularitiesMovingFoldLocus}. Imagine now that $\Sigma$ is one of the cusps of a zig-zag. Applying this argument repeatedly would allow us to put $\Sigma$ very close to the other cusp $\Sigma'$, at the expense of adding many smaller zig-zags with ball base. Once $\Sigma$ and $\Sigma'$ are close enough, the zig-zag is removable.

\begin{figure}[ht]
\centering
\includegraphics[width = 0.7\linewidth ]{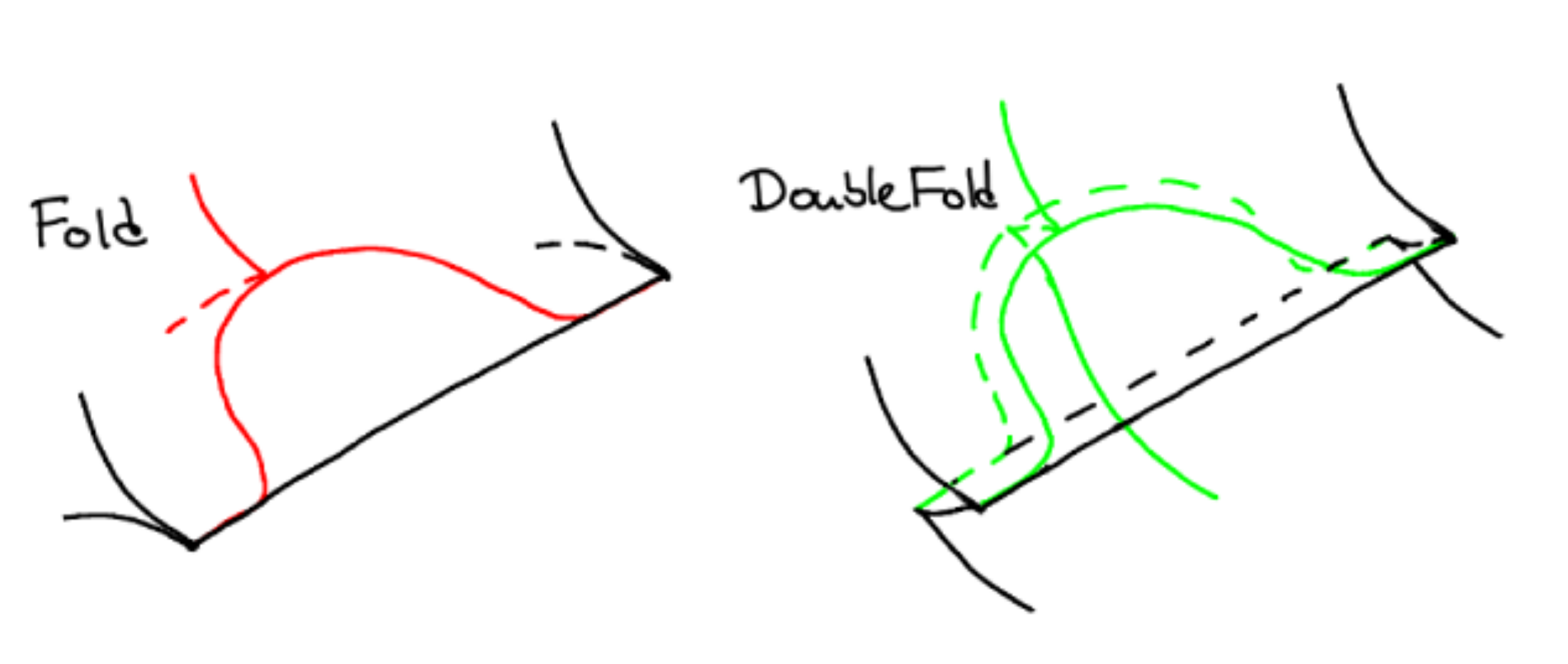}
\caption{On the left: Moving the cusp locus generally induces a $C^0$-big change. On the right: A removable double fold can be moved around in a $C^0$-small manner.}\label{fig:ModifyingSingularitiesFoldvsDoublefold}
\end{figure}

\begin{figure}[ht]
\centering
\includegraphics[width=\linewidth]{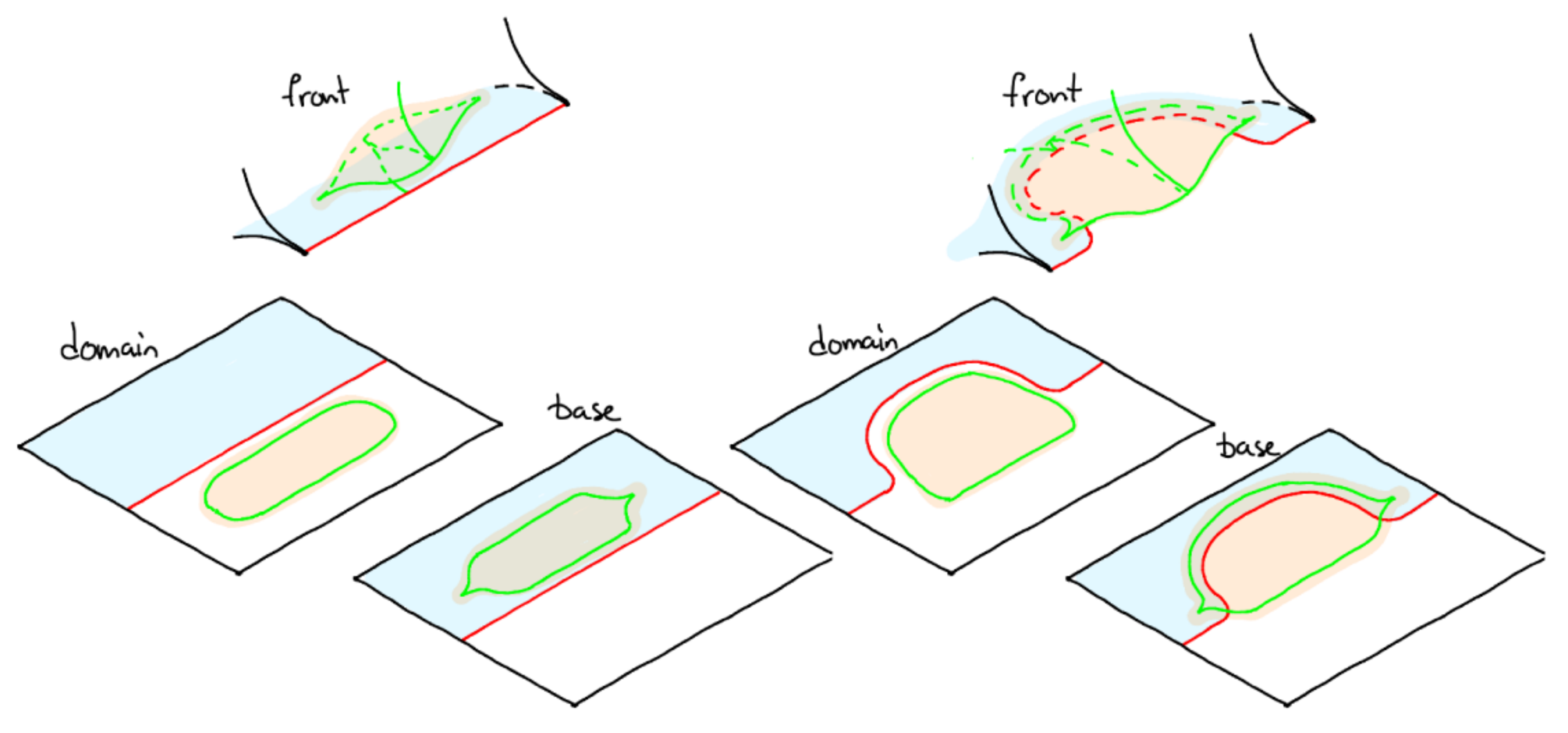}
\caption{In order to move the cusp locus $\Sigma$ of $g$ we first create a small zig-zag wrinkle in $\Op(U)$, as shown on the left. We imagine that $\Sigma$ and one of the hemispheres of the wrinkle are paired up, forming a zig-zag. As shown on the right, this allows us to move the two of them together towards $\wtd\Sigma$. The resulting map is $f$.}\label{fig:ModifyingSingularitiesMovingFoldLocus}
\end{figure}

\subsubsection{The chopping model}

The toy model introduced in the previous subsection relies on the existence of a certain nicely-behaved multi-section $F$, depicted in Figure \ref{fig:WrinkledEmbeddings_MovingSingularLocusTwo}. We explain its construction next.

Our starting point is the cusp 
\begin{align*}
G: \R^2   &\longrightarrow J^0(\R^2,\R) \\
(x_1,t)   &\mapsto (x_1,x_2(t) = t^2, y(t) = t^{2r+1}).
\end{align*}
Fix a constant $\varepsilon > 0$. Fix a non-increasing function $\chi: \R \to \R$ such that $\chi|_{\R^{\leq 0}} = 0$ and $\chi|_{\R^{\geq \varepsilon}} = -\id$. Then:
\begin{lemma} \label{lem:choppingModel}
There is a multi-section:
\[ F: \R^2 \longrightarrow J^0(\R^2,\R) \]
satisfying the following properties:
\begin{itemize}
\item $j^rF$ and $j^rG$ are $\varepsilon$-close.
\item $F = G$ if $x_1 \leq \varepsilon$ or $t \notin [-x_1-\varepsilon, \varepsilon]$.
\item $F$ has a cusp along $\{t=\chi(x_1)\}$.
\item $F$ has a model zig-zag wrinkle with membrane contained in $\{t \in (-x_1, \varepsilon/2)\}$.
\end{itemize}
\end{lemma}
\begin{proof}
The map $F$ is shown in Figure \ref{fig:WrinkledEmbeddings_MovingSingularLocusTwo}. Let $\eta, \delta >0$ be small constants to be fixed later.

Suppose $H: \R^2 \to J^0(\R^2,\R)$ is the zig-zag swallowtail; let $A$ be its membrane. Using a point symmetry compressing $x_2$ and $y$ we may assume that $j^rH|_A$ maps to a $\eta$-neighbourhood of the zero section over the line $\{x_2=0\}$. Similarly, we can reparametrise in $t$ in order to compress the membrane to a $\eta$-neighbourhood of $\{t=0\}$.

We can then shift in both domain and target, in a fibered manner over $x_1$, so that:
\begin{itemize}
\item $H$ maps the $(x_1,t)$-half-line $L = [\delta,\infty) \times \{\delta\}$ to a graph over the $x$-half-line $[\delta,\infty) \times \{\delta\}$,
\item the membrane $A$ contains $L$ and is contained in the $\eta$-neighbourhood of $L$.
\end{itemize}

Let $\Psi$ be a fibered diffeomorphism of $J^0(\R^2,\R)$, lifting the identity, mapping the zero section over $\{t \geq \delta\}$ to the upper branch of $G$. We can define a new multi-section $H'$ that is given by $\Psi \circ H$ close to $A$ and by $G$ outside of a $\delta/4$-neighbourhood of $A$. We may assume that $H'|_L = G|_L$. It may be assumed that $j^rH'$ is $\delta$-close to $j^rG$ if $\eta$ is small enough.

Let $U$ be a neighbourhood of $\{t \leq \delta\}$. We ask that it is small enough so that it is disjoint from the upper branch of the wrinkle of $H$ for every $x_1 \geq 2\delta$. We now observe that there is a diffeomorphism $\Phi$ of $\Op(H'(U)) \subset J^0(\R^2,\R)$, that satisfies:
\begin{itemize}
\item It is fibered over $x_1$.
\item Its support is contained in the $\delta$-neighbourhood of $H(\{x_1 \geq 0, \chi(x_1) \leq t \leq \delta\})$.
\item It is fibered over $\R^2$ in $\Op(H(\{t \leq 0\}))$.
\item It maps $H(U)$ to a neighbourhood of $H(\{t \leq \chi(x_1) + \delta\})$.
\end{itemize}
We set $F = \Phi \circ H'$. If $\eta$ and $\delta$ are sufficiently small, $F$ satisfies the desired properties.
\end{proof} 

\begin{figure}[ht]
\centering
\includegraphics[width=0.85\linewidth]{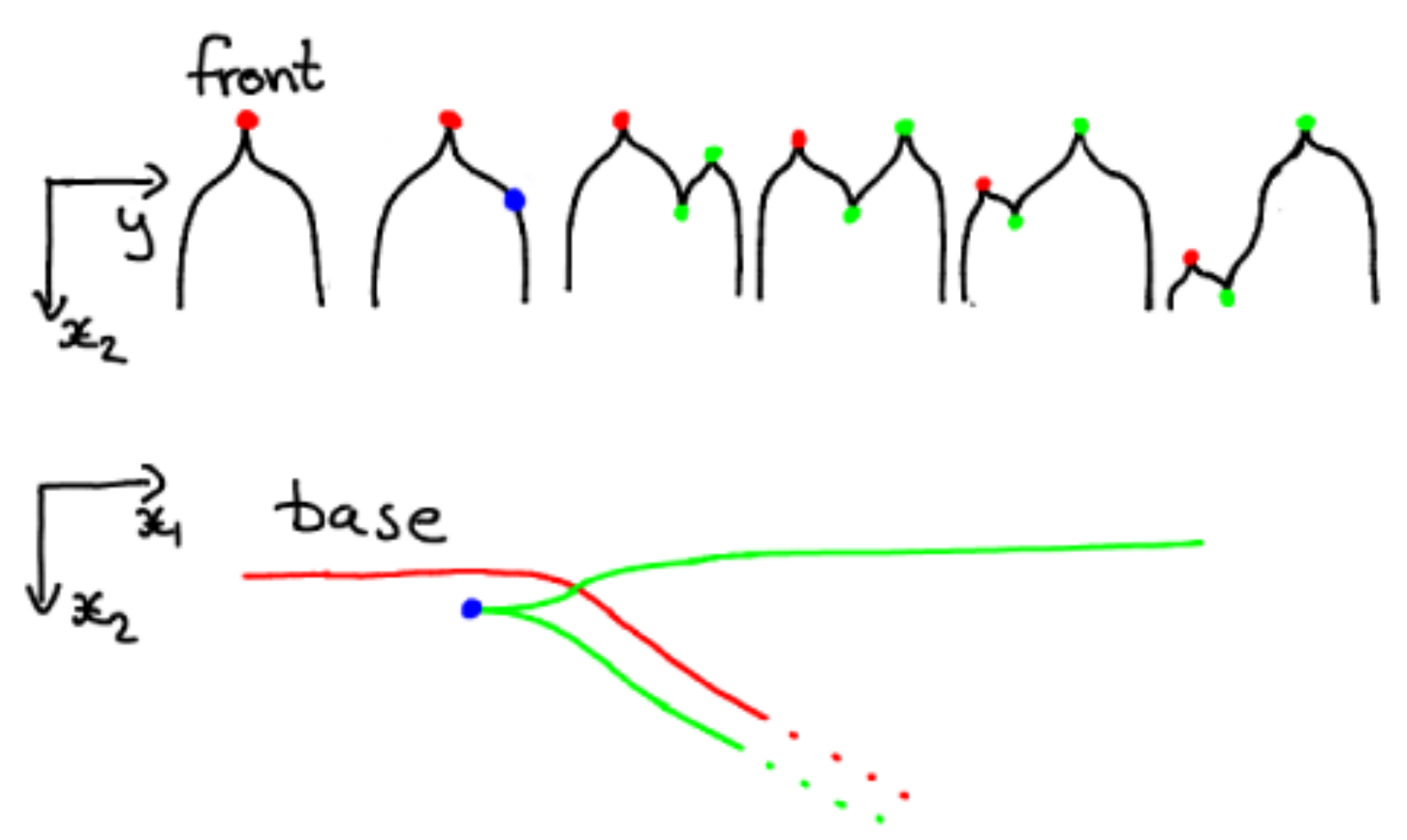}
\caption{The map $F$ introduced in Lemma \ref{lem:choppingModel}. The cusp is shown in red. The blue point denotes the swallowtail event. The cusp locus of the wrinkle is shown in green.}\label{fig:WrinkledEmbeddings_MovingSingularLocusTwo}
\end{figure}

\subsubsection{The chopping lemma}

Using the model introduced in the previous subsection we can finally state our surgery result:
\begin{proposition}\label{prop:WrinkledEmbeddings_MovingSingularLocus}
Fix a model zig-zag wrinkle $g$ with base $D \subset K \times H$ and membrane $A \subset K \times H \times \R$. Let $A' \subset A$ be obtained from $A$ by an isotopy $\psi$ fibered over $K \times H$. We require that the support of $\psi$ is disjoint from the lower hemisphere of $\partial A$ and lies over $\Op(B)$, where $B$ is ball contained in the interior of $D$.

Given $\varepsilon>0$, there exists a multi-section $f$ satisfying:
\begin{itemize}
    \item $|j^rf - j^rg|_{C^0} < \varepsilon$;
    \item $f$ has a model zig-zag wrinkle with base $D$ and membrane $A'$.
		\item The only other singularity of $f$ is a zig-zag wrinkle with base $B$ and membrane contained in $\Op(A \setminus A')$.
\end{itemize}
\end{proposition}
\begin{proof}
Since $B$ is contained in the interior of $D$, it is sufficient to prove the statement for a zig-zag wrinkle without equator. Namely: we have that $D = \Op(B)$, $A$ is an annulus with base $D$, and $A'$ is another annulus with the same lower boundary $\partial^- A$ as $A$. A point symmetry allows us to assume that $\partial^- A$ is situated in $t=0$ and it maps to the zero section over $x_n=0$. Let $\wtd A = A \setminus \Op(\partial^+ A)$. We can reparametrise $\wtd A$, in a fibered manner over $D$, so that:
\[ g(k,\widetilde x, t) = (k,\widetilde x,t^2, t^3). \]
The upper boundary of $A' \subset \wtd A$ is then given as the graph of a positive function $\psi$ over $D$.

Let $F$ and $G$ be as in Lemma \ref{lem:choppingModel}. The proof concludes by setting 
\[ f(k,\widetilde x,t) := (k,\widetilde x, x_n \circ F(\psi(k,\widetilde x),t), y \circ F(\psi(k,\widetilde x),t)), \]
as long as the constant bounding the distance between $j^rF$ and $j^rG$ is chosen to be sufficiently small.
\end{proof}

\begin{figure}[ht]
\centering
\includegraphics[width=0.85\linewidth]{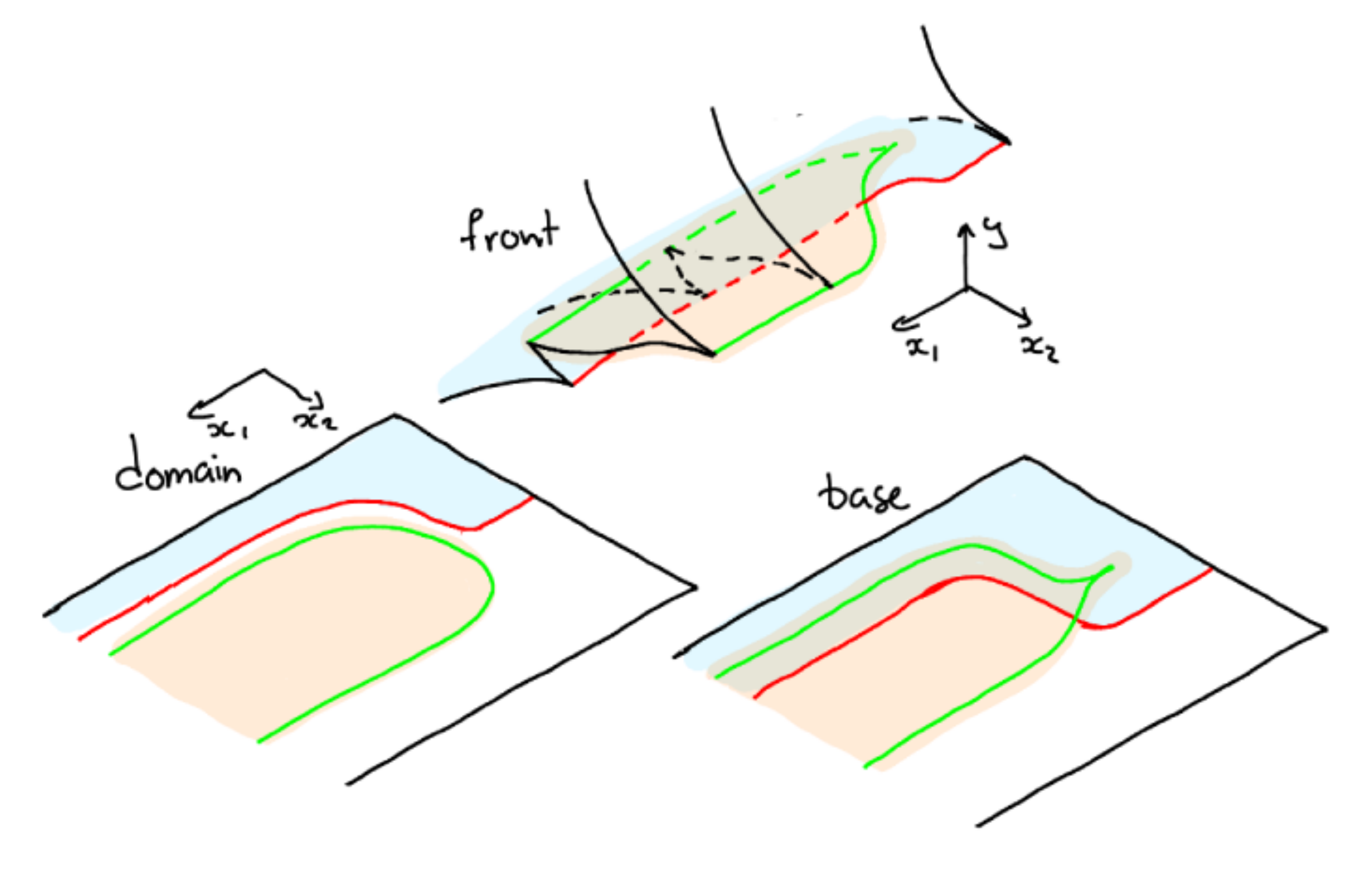}
\caption{The map $f$ produced by Proposition \ref{prop:WrinkledEmbeddings_MovingSingularLocus}.}\label{fig:WrinkledEmbeddings_MovingSingularLocus}
\end{figure}

\subsection{Reducing to wrinkles with ball base} \label{ssec:wrinklesBallBase}

Using Proposition \ref{prop:WrinkledEmbeddings_MovingSingularLocus}, we now explain how to replace wrinkled families of multi-sections by families whose wrinkles all have ball base. The size of the resulting wrinkles will be as small as we want.

\begin{theorem}\label{thm:ModifyingSingularitiesCuttingWrinkles}
Fix a constant $\varepsilon >0$. Let $g = (g_k : M \to J^r(Y))_{k \in K}$ be a wrinkled family of multi-sections. Suppose we are given $\{U_i\}$ a cover of $M$ and $\{V_j\}$ a cover of $Y$.

Then, there exists a wrinkled family of multi-sections $f$ satisfying:
\begin{itemize}
    \item $|j^rf - j^rg|_{C^0}< \varepsilon$;
    \item all the singularities of $f$ are wrinkles with ball base;
    \item the membrane of each wrinkle is contained in some $U_i$, and its image in some $V_j$.
\end{itemize}
\end{theorem}
\begin{proof}
We assume $K \times M$ is compact. If not, the argument has to be adapted using a exhaustion. By choosing refinements we can assume that for each $i$ there is a $j$ such that $g(U_i) \subset V_j$. Let $\delta, \eta >0$ be small constants to be fixed later.

We just need to explain how to cut a single zig-zag wrinkle $g|_A$. We assume that it is given by its model: I.e. we write $D \subset K \times H$ for its base, $\rho$ for the height, and we see the membrane $A$ as the subset  
\[ \{t^2 + \rho(k,\widetilde x) \leq 0\} \subset K \times H \times \R. \]
We think of $g|_{\Op(A)}$ as a map with domain $\Op(A)$ and target $K \times J^0(H \times \R,\R)$. The fibres of $Y$ may have dimension greater than $1$ but, according to the model, we can disregard these additional directions.

We can now see $\partial A$ as the union of the two hemispheres $\partial^+ A$ and $\partial^- A$; the first is given as the graph of $\sqrt{\rho}: D \subset K \times H \to \R$ and the latter as the graph of $-\sqrt{\rho}$. Let $D' \subset D$ be $D$ minus a $\delta$-neighbourhood of $\partial D$. Over $D'$, both functions $\pm\sqrt{\rho}$ are smooth.

Roughly speaking, our goal now is to cut the difference function $\sqrt{\rho}-(-\sqrt{\rho}) = 2\sqrt{\rho} > 0$ into small bump functions. This will then translate into the chopping of the wrinkle itself. Ultimately it will be more convenient to work with $\psi := 2\sqrt{\rho} - \eta$ instead.

Consider the functions $\phi_l := \min(\delta\ell,\psi)$; for every $l$ sufficiently large it holds that $\phi_l = \psi$. We now write $\psi_l$ for a smoothing of the difference $\psi_l - \psi_{l-1}$. It may be assumed that $\psi - \sum_l \psi_l$ is positive and $\delta$-small by taking $\eta$ sufficiently small. Note that the sum on the right is finite.

Let $\{D_m \subset D\}$ be a cover of $D'$ by opens of diameter $\delta$. We consider a partition of unity subordinated to $\{D_m\}$ and we apply it to each $\psi_l$. This yields bump functions $\psi_{l,m}: D \to [0,2\delta]$ whose supports have size $\delta$ and such that $\psi - \sum_{l,m} \psi_{l,m}$ is $\delta$-small. We can furthermore assume, by a small deformation, that $\psi_{l,m}$ is positive in the interior of a ball $B_{l,m}$, negative outside, and cuts the boundary sphere transversely.

We now do a double induction, with $l_0$ as the outer parameter and $m_0$ as the inner one. Write $P_{l_0,m_0}$ for the partial sum 
\[ \sum_{l > l_0, m} \psi_{l,m} + \sum_{m < m_0} \psi_{l_0,m}. \]
The induction hypothesis is that we have replaced $g$ by a wrinkled family $\wtd g$ that:
\begin{itemize}
\item has a zig-zag wrinkle with base $D$. Its membrane has $\partial^- A$ as lower hemisphere and $\partial^{l_0,m_0} A$ as upper hemisphere. Here $\partial^{l_0,m_0} A$ is the graph of $\sqrt{\rho} - P_{l_0,m_0}$.
\item its other singularities are zig-zag wrinkles with ball base contained in some $D_m$.
\end{itemize}
The inductive step follows as an application of Proposition \ref{prop:WrinkledEmbeddings_MovingSingularLocus} by taking $\Sigma = \partial^{l_0,m_0}A$ and $\wtd \Sigma = \partial^{l_0,m_0+1}A$ or $\wtd \Sigma=\partial^{l_0-1,0}A$, as needed.

Once the induction is complete, the new zig-zag wrinkles have absorbed most of the wrinkle we started with. Namely, there is still a wrinkle of base $D$ whose height can be assumed to be arbitrarily small by choosing $\eta$ and $\delta$ small. In particular, it is removable. If $\eta$ and $\delta$ are sufficiently small it can be removed yielding the claimed $f$ with $j^rf$ close to $j^rg$.
\end{proof}

\subsubsection{Applications}

Our main application says that holonomic approximation can be achieved using wrinkled families all whose wrinkles have ball base:
\begin{proof}[Proof of Corollary \ref{cor:holonomicApproxBall}]
We apply Theorem \ref{thm:ModifyingSingularitiesCuttingWrinkles} to the holonomic approximation produced by Theorem \ref{thm:holonomicApproxParam}.
\end{proof}

We also prove the analogous statement for submanifolds:
\begin{proof}[Proof of Corollary \ref{cor:wrinkledEmbeddingsBall}]
Unlike the multi-section case, we cannot apply Theorem \ref{thm:ModifyingSingularitiesCuttingWrinkles} directly to the holonomic approximation produced by Theorem \ref{thm:wrinkledEmbeddingsParam}. The reason is that our arguments for singular submanifolds require nesting, but our surgery methods apply only to model zig-zag wrinkles.

This can be addressed by adjusting the proof of Theorem \ref{thm:wrinkledEmbeddingsParam}. Namely, every time we work with a graphical $r$-jet homotopy, we invoke the multi-section result Theorem \ref{thm:holonomicApproxParam} first, followed by the surgery procedure Theorem \ref{thm:ModifyingSingularitiesCuttingWrinkles}. This will produce wrinkles with ball base at each step, as desired. 
\end{proof}

Lastly, we point out that Theorems \ref{thm:fishSurgerySec} and \ref{thm:fishSurgerySubm} from Appendix \ref{sec:desingularisationOdd} also hold for fish wrinkles with ball base. This can be achieved by applying first the chopping Theorem \ref{thm:ModifyingSingularitiesCuttingWrinkles} and then applying the fish surgery from Proposition \ref{prop:fishSurgery}.

\begin{figure}[ht]
\centering
\includegraphics[width=\linewidth]{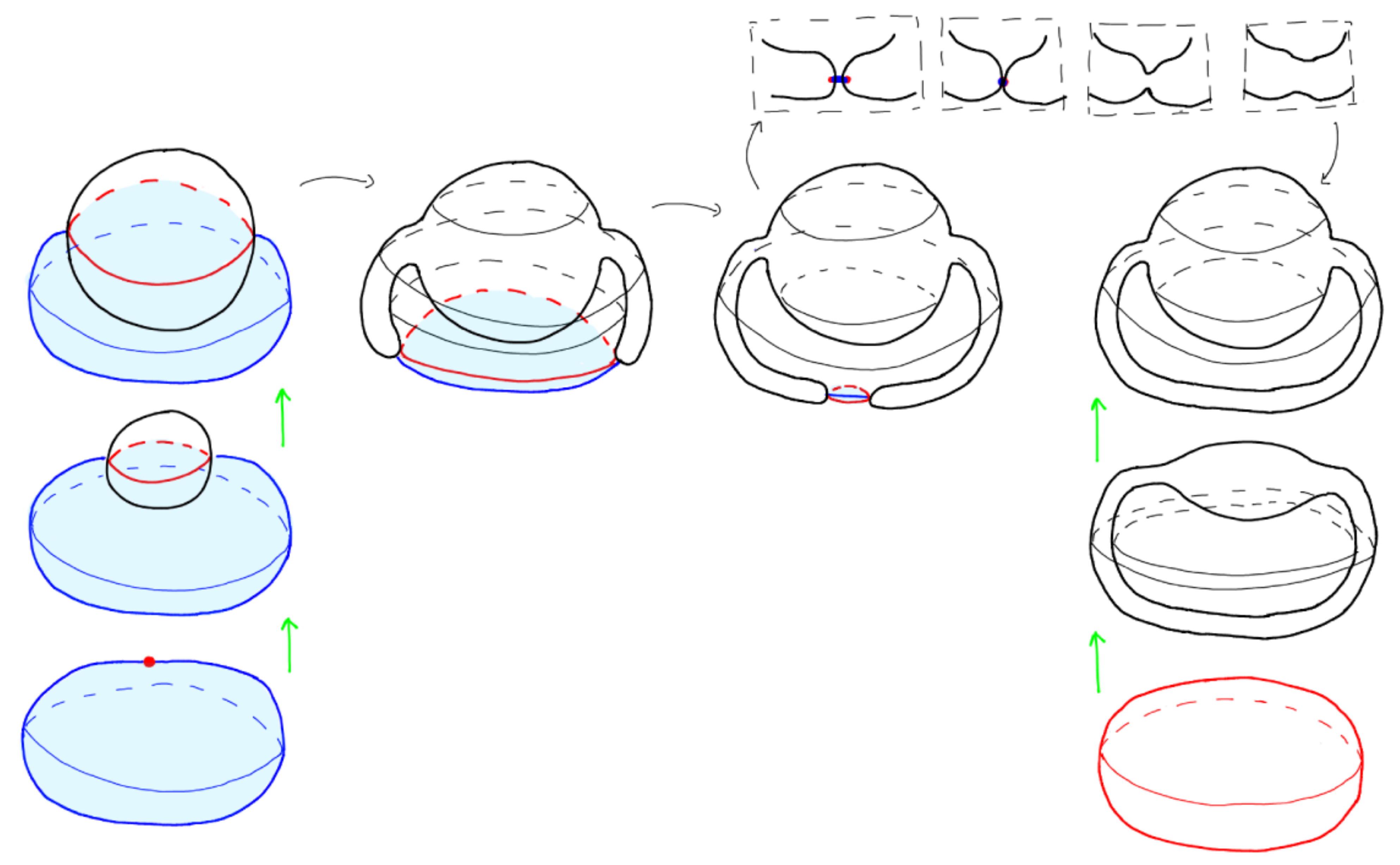}
\caption{The three pictures on the left-hand side depict the domain of a zig-zag wrinkle with ball base $D$. In red we show the equator; the bottom picture depicts an embryo event. We then choose a family of codimension-1 discs $\SD$ (in blue), that run close to the lower hemisphere of the membrane and that cap-off the equator. The model of the wrinkle can be extended to $\Op(\SD)$, where the map is graphical. We then replace the map over $\Op(\SD)$ in order to yield a wrinkle with cylinder base $D \cup \SD$. This is shown on the right; the bottom right depicts an embryo sphere. This replacement is in fact a cobordism given by a wrinkled family, as shown in the two middle figures.}\label{fig:Wrinkling_InverseSurgery}
\end{figure}

\subsection{Reducing to wrinkles with cylinder base} \label{ssec:wrinklesCylinderBase}

We now state the analogue of Theorem \ref{thm:ModifyingSingularitiesCuttingWrinkles} for wrinkles with cylinder base.
\begin{theorem}
Fix a constant $\varepsilon >0$. Let $g = (g_k : M \to J^r(Y))_{k \in K}$ be a wrinkled family of multi-sections all whose singularities have ball base.

Then, there exists a wrinkled family of multi-sections $f$ satisfying:
\begin{itemize}
    \item $|j^rf - j^rg|_{C^0}< \varepsilon$;
    \item all the singularities of $f$ are wrinkles with cylinder base;
    \item the wrinkles of $f$ and $g$ are in 1-to-1 correspondence and their membranes are $C^0$-close.
\end{itemize}
\end{theorem}
We leave the proof to the reader, but we illustrate it in Figure \ref{fig:Wrinkling_InverseSurgery}.